\documentclass[10pt]{amsart}
\theoremstyle{plain}
\usepackage{amsmath, amsfonts, amsthm, amscd}
\usepackage{graphics}
\usepackage{color}
\usepackage{epsfig}
\usepackage[all]{xy}
\graphicspath{ {Figures/} }

\title[Khovanov homology and Sutured Floer homology]{On the Colored Jones Polynomial, Sutured Floer Homology, and Knot Floer Homology}
\author{J. Elisenda Grigsby}
\thanks{JEG was partially supported by an NSF postdoctoral fellowship.}
\address{J. Elisenda Grigsby\\Columbia Math Dept.;2990 Broadway MC4406; NY, NY 10027}
\email{egrigsby@math.columbia.edu}

\author{Stephan Wehrli}
\thanks{SW was partially supported by a Swiss NSF fellowship for prospective researchers.}
\address{Stephan Wehrli\\Columbia Math Dept.;2990 Broadway MC4406; NY, NY 10027}
\email{wehrli@math.columbia.edu}

\theoremstyle{plain}
\newtheorem{theorem}{Theorem}[section]
\newtheorem{lemma}[theorem]{Lemma}
\newtheorem{proposition}[theorem]{Proposition}
\newtheorem{corollary}[theorem]{Corollary}
\newtheorem{conjecture}[theorem]{Conjecture}

\newtheorem*{theoremA}{Theorem~\ref{thm:SpecSeq}}
\newtheorem*{theoremB}{Theorem~\ref{thm:KnotFloerRel}}
\theoremstyle{definition}
\newtheorem{definition}[theorem]{Definition}
\newtheorem{remark}[theorem]{Remark}
\newtheorem{example}[theorem]{Example}

\newcommand{\cP}{\mathcal{P}}
\newcommand{\cI}{\mathcal I}
\newcommand{\Wedge}{\Lambda}
\newcommand{\Ztwo}{\mathbb{Z}_2}

\newcommand{\Szabo}{{Szab{\'o}} }
\newcommand{\Juhasz}{{Juh{\'a}sz} }

\newcommand{\Q}{\ensuremath{\mathbb{Q}}}
\newcommand{\R}{\ensuremath{\mathbb{R}}}
\newcommand{\Z}{\ensuremath{\mathbb{Z}}}

\newcommand{\boldalpha}{\ensuremath{\mbox{\boldmath $\alpha$}}}
\newcommand{\boldbeta}{\ensuremath{\mbox{\boldmath $\beta$}}}

\newcommand{\boldeta}{\ensuremath{\mbox{\boldmath $\eta$}}}
\newcommand{\boldSigma}{\ensuremath{\mbox{\boldmath $\Sigma$}}}

\begin{document}
\bibliographystyle{plain}

\begin{abstract} Let $K \subset S^3$, and let $\widetilde{K}$ denote the preimage of $K$ inside its double branched cover, $\boldSigma(S^3,K)$.  We prove, for each integer $n>1$, the existence of a spectral sequence whose $E^2$ term is Khovanov's categorification of the reduced $n$--colored Jones polynomial of $\overline{K}$ (mirror of $K$) and whose $E^\infty$ term is the knot Floer homology of $(\boldSigma(S^3,K), \widetilde{K})$ (when $n$ odd) and of $(S^3, K \# K)$ (when $n$ even). 
A corollary of our result is that Khovanov's categorification of the
reduced $n$-colored Jones polynomial detects the unknot whenever $n>1$.
\end{abstract}
\maketitle

\section{Introduction}
Since their introduction less than ten years ago, Khovanov homology \cite{MR1740682} and Heegaard Floer homology \cite{MR2113019} have generated a tremendous amount of activity and a stunning array of applications.  

Although they have quite different definitions, the knot invariants associated to the two theories share many formal properties:

\begin{enumerate}
\item They both {\em categorify} classical knot polynomials.  I.e., each is a bigraded homology theory whose Euler characteristic is a classical knot polynomial (Jones and Alexander, resp.).
\item Both theories come equipped with a filtration which yields a concordance invariant ($s$ \cite{GT0402131} and $\tau$ \cite{MR2065507,GT0306378}, resp.)
\item Both theories are ``uninteresting'' (determined by classical invariants) on quasi-alternating knots \cite{GT07083249}.
\end{enumerate}

Ozsv{\'a}th-Szab{\'o} provided the first clue about a relationship between the two theories:

\begin{theorem}\cite[Theorem 1.1]{MR2141852} Let $L \subset S^3$ be a link and $\overline{L} \subset S^3$ its mirror.  There is a spectral sequence whose $E^2$ term is $\widetilde{Kh}(\overline{L})$ and which converges to $\widehat{HF}(\boldSigma(S^3,L))$.
\end{theorem}

In the above, $\widetilde{Kh}(L)$ refers to the reduced Khovanov homology of a link $L \subset S^3$ \cite{MR2034399}, $\widehat{HF}(Y)$ refers to the ($\wedge$ version of the) Heegaard Floer homology of the closed, connected, oriented $3$--manifold $Y$ \cite{MR2113019}, and $\widehat{HFK}(Y,K)$ refers to the ($\wedge$ version of the) knot Floer homology of the nullhomologous knot $K \subset Y$ \cite{MR2065507,GT0306378}.  Furthermore, for a codimension $2$ pair $(B,\partial B) \subset (X,\partial X)$, we use $\boldSigma(X,B)$ to denote the double-branched cover of $X$ over $B$ and $\widetilde{B}$ to denote the preimage of $B$ inside $\boldSigma(X,B)$.  Throughout the paper, all Khovanov and Heegaard Floer homology theories will be considered with $\Z_2$ coefficients.

Our present aim is to position Ozsv{\'a}th-Szab{\'o}'s result in a more general context.  In particular, if $K \subset S^3$ is a knot, then Khovanov associates to $K$ a whole sequence of invariants, $\widetilde{Kh}_n(K)$, categorifying the reduced $n$--colored Jones polynomials \cite{MR2124557}.  

%Remove a small neighborhood of $p$ and consider the $n$-cable, $T^n \subset (D^2 \times I)$, of the resulting tangle, for $n \in \mathbb{Z}_+$.  Associated to any {\em balanced tangle} $T \subset (D^2 \times I)$ (see Definition \ref{defn:}), Khovanov associates a chain complex whose homology is an invariant of the tangle. Let $\widetilde{Kh}_n(K)$ denote the categorification of the reduced colored Jones polynomial for $n \in \mathbb{Z}_+$, defined in \cite{MR2124557} and $SFH(Y)$ denote the sutured Floer homology of the sutured manifold, $Y$, defined in \cite{MR2253454}.  All chain complexes will be considered with $\mathbb{Z}/2\mathbb{Z}$ coefficients, unless otherwise noted.
 
We prove:

\begin{theorem}\label{thm:Main}
Let $K \subset S^3$ be an oriented knot, $\overline{K} \subset S^3$ its mirror, and $K^r$ its orientation reverse.  For each $n \in \mathbb{Z}_{> 0}$, there is a spectral sequence, whose $E^2$ term is $\widetilde{Kh}_n(\overline{K})$ and whose $E^\infty$ term is
\[HF_n(K) := \left\{\begin{array}{cc}
    \widehat{HF}(\boldSigma(S^3,K)) & \mbox{if $n = 1$},\\
    \widehat{HFK}(\boldSigma(S^3,K), \widetilde{K}) & \mbox{if $n>1$ and odd}\\
    \widehat{HFK}(S^3, K \# K^r) & \mbox{if $n$ is even},
    \end{array}\right.\]
\end{theorem}

In the above, $HF_n(K)$ is actually a grading-shifted version of the stated homology group, where the shift depends, in a prescribed way, upon $n$. We compute this grading shift explicitly in Section \ref{sec:RelKnotFloer}.  In that section, we also mention a conjectural relationship between the Khovanov and Floer gradings which would imply a connection, for large $n$, between the so-called {\em homological width} of $\widetilde{Kh}_n(K)$ and the knot genus.  %Note that Theorem \ref{thm:Main} reduces to Ozsv{\'a}th-Szab{\'o}'s result in the case $n=1$.

Theorem \ref{thm:Main} yields the following easy corollary:

\begin{corollary}\label{cor:DetectsUnknot}
$\widetilde{Kh}_n(K)$ detects the unknot for all $n>1$.
\end{corollary}

\begin{proof}
\cite{MR2023281} tells us that if $K \subset Y$ is a nullhomologous knot and $g(K) > 0$, then \[rk(\widehat{HFK}(Y,K))>1.\]
Now, suppose that $K \subset S^3$ is not the unknot, $U$.

$g(K \# K^r) = 2g(K)$, and $g(\widetilde{K}) = g(K)$, so $rk(HF_n(K)) > 1$ for $n>1$, by Theorem \ref{thm:Main}.  

The existence of the spectral sequence 
\[\widetilde{Kh}_n(K) \rightarrow HF_n(K)\]
then implies that \[ rk(\widetilde{Kh}_n(K)) \geq rk(HF_n(K)) > 1\] for all $n > 1$.

In particular, $rk(\widetilde{Kh}_n(K)) \neq 1$ when $K \neq U$.
\end{proof}

We remark that Andersen \cite{JEAndersen} has announced a proof of a related result--namely, that the full collection of colored Jones polynomials is an unknot detector.  His result arises from a quite different perspective, via an argument relating the growth rate of the invariants to the nontriviality of $SU(2)$ representations of the fundamental group of surgeries on the knot.  

We would also like to mention the work of Hedden \cite{GT08054418}, who, after hearing a talk given by the second author on a weaker version of Theorem \ref{thm:Main} (see the Acknowledgements section), was able to prove via existing Heegaard Floer homology techniques that the Khovanov homology of the $2$--cable detects the unknot.

Before proceeding to the proof, we pause to say a few words about our techniques.  Throughout, we make heavy use of {\em sutured Floer homology}, a beautiful theory developed by \Juhasz which associates to a sutured $3$--manifold (in the sense of \cite{MR723813}) Heegaard Floer-type homology groups \cite{MR2253454}.  In fact, sutured Floer homology appears, in general, to have a tight connection to various Khovanov-type constructions associated to tangles.  In this direction, we explore the relationship between our work and that of Lawrence Roberts \cite{GT07060741} in an upcoming paper, where we interpret our spectral sequence (for odd $n$) as a special case of a direct summand of the one he constructs.  More generally, categorifications of Kauffman bracket skein modules of $I$--bundles over surfaces \cite{MR2113902} certainly merit further attention.

For the present application, we begin with a knot $K \subset S^3$, constructing from it a {\em balanced tangle} $T^n \subset D \times I$ (see Definition \ref{defn:balancedtangle} and the discussion preceding it) by removing a neighborhood of a point and taking the $n$--cable.  There are then two natural chain complexes one can associate to $T^n$: one obtained using a Khovanov-type functor (Section \ref{sec:KhFunctor}) and the other using a sutured Floer-type functor (Section \ref{sec:SFFunctor}).

The key observation is that the generators of these chain complexes agree on admissible, balanced, resolved tangles, defined in Section \ref{sec:Resolved}.  Furthermore, both theories satisfy certain skein relations which allow us to build the chain complex associated to a tangle from its cube of resolutions.  More specifically, associated to the $i$-th crossing of a projection $\mathcal{P}(T)$ of a balanced tangle $T$, one can form $\mathcal{P}^i_0(T)$ and $\mathcal{P}^i_1(T)$, the so-called $0$ and $1$ resolutions of the crossing (see Figure \ref{fig:Resolutions}).  In both the Khovanov and sutured Floer settings, the chain complex for the tangle can then be defined by iteratively resolving crossings.  In Khovanov's theory, this structure is part of the definition.  In sutured Floer homology, this structure arises because of the existence of a link surgeries spectral sequence, described in Section \ref{sec:LinkSurg}, which relates the sutured Floer homologies of sutured $3$--manifolds differing by triples of surgeries along an imbedded link.

Our main theorem is, therefore, really an amalgam of two theorems:

\begin{theoremA} 
Let $K \subset S^3$ be an oriented knot and $\overline{K} \subset S^3$ its mirror.  For each $n \in \mathbb{Z}_{>0}$, there is a spectral sequence, whose $E^2$ term is $\widetilde{Kh}_n(\overline{K})$ and whose $E^\infty$ term is $SFH(\boldSigma(D \times I,T^n))$.
\end{theoremA}

\begin{theoremB}
Let $n \in \mathbb{Z}_{>0}$.

\[SFH(\boldSigma(D \times I, T^n)) \cong \left\{\begin{array}{cc}
    \widehat{HF}(\boldSigma(S^3,K)) & \mbox{if $n = 1$},\\
    \widehat{HFK}(S^3, K \# K^r) & \mbox{if $n$ is even},\\
    \widehat{HFK}(\boldSigma(S^3,K), \widetilde{K}) & \mbox{if $n>1$ and odd}
    \end{array}\right.\] 
\end{theoremB}

Theorem \ref{thm:SpecSeq} will require us to review sutured Floer homology as well as extend some basic Heegaard Floer-type definitions and results to the sutured Floer setting (definition of sutured Floer multi-diagrams and the natural maps associated to them, validity of Lipshitz's Maslov index formula, finiteness of holomorphic disk counts for admissible sutured multi-diagrams, etc.).  These results are included for completeness and also because they may prove useful for future applications of sutured Floer homology; however, they may be safely skimmed on a first reading.

Theorem \ref{thm:KnotFloerRel} is proved using a simple topological observation coupled with some techniques from sutured Floer homology (in particular, its behavior under surface decompositions).  
%Informally, a spectral sequence $A \rightarrow B$ exists when both $A$ and $B$ can be realized as homology groups of two chain complexes $C_*(A)$ and $C_*(B)$ with the same generators, but where $\partial_B = \partial_A + \partial_{\mbox{HO}}$, where we think of $\partial_{\mbox{HO}}$ as ``higher order'' terms in the differential.

The paper is organized as follows:

In Section \ref{sec:StandardDefn}, we recall the necessary sutured Floer and Heegaard Floer homology background and introduce two sutured Floer operations--gluing and branched covering--that we will use repeatedly.  In this section, we also verify the validity of Lipshitz's Maslov index formula in the sutured Floer setting.

Section \ref{sec:MultiDiag} is a compilation of the technical results necessary for the statement and proof of the link surgeries spectral sequence in the next section.  We define sutured Heegaard multi-diagrams and and discuss how they can be used to define maps between sutured Floer chain complexes by examining moduli spaces of holomorphic polygons.  We also set up the appropriate admissibility hypotheses ensuring the finiteness of these polygon counts in the setting of interest to us.  We close the section by discussing polygon associativity and the naturality of triangle maps under a $\Lambda^*(H_1)$ action.

In Section \ref{sec:LinkSurg}, we prove the analogue of the link surgeries spectral sequence in the sutured Floer homology setting.  A corollary will be the identification of $SFH(\boldSigma(D \times I, T))$ for any admissible, balanced tangle $T$, with the homology of a certain iterated mapping cone.

In Section \ref{sec:SpecSeq}, we prove Theorem \ref{thm:SpecSeq} by demonstrating the equivalence of the Khovanov and sutured Floer functors on admissible, balanced, resolved tangles $T \subset D \times I$ and applying the link surgeries spectral sequence.

In Section \ref{sec:RelKnotFloer}, we give a proof of Theorem \ref{thm:KnotFloerRel}, followed by an explicit calculation of the grading shifts between $SFH(\boldSigma(D\times I, T^n))$ and $HF_n(K)$.  We conclude with a conjecture relating gradings in Khovanov homology to gradings in sutured Floer homology.

\begin{figure}
\begin{center}
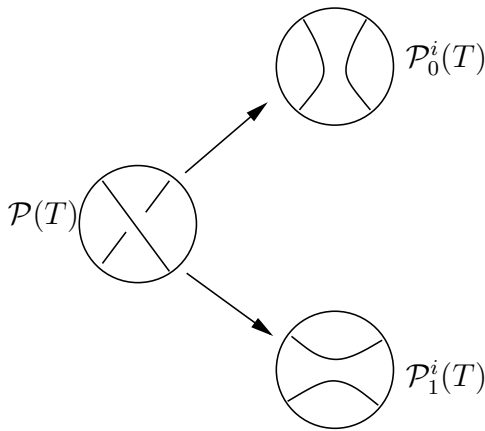
\end{center}
\caption{An illustration of the $0$ and $1$ resolutions of a crossing in a tangle projection.  Let $\mathcal{P}(T)$ be a projection of the tangle $T$.  Then $\mathcal{P}^i_0(T)$ and $\mathcal{P}^i_1(T)$ are obtained by replacing a small neighborhood of the $i$-th crossing as specified.}
\label{fig:Resolutions}
\end{figure}

%It is interesting to compare our results with those of Lawrence Roberts, who, in \cite{}, provides another generalization of Ozsv{\'a}th-Szabo{\'o}'s result, exhibiting the existence of a spectral sequence connecting the categorification of the Kauffmann bracket skein homology of $L \subset A \times I$ for $L$ a link in the standardly-imbedded annulus $A \subset S^3$, with the knot Floer homology of the preimage of the axial knot, $\widetilde{B} \subset \boldSigma(S^3, K)$

\subsection{Acknowledgements}
We warmly thank Robert Lipshitz and Peter Ozsv{\'a}th, who patiently answered our technical questions.  We are also greatly indebted to Ciprian Manolescu.  In a previous version of Theorem \ref{thm:Main}, we were able to compute $HF_n(K)$ explicitly only for sufficiently large $n$.  After hearing a talk on this result given by the second author at the ``Knots in Washington XXVI'' conference at George Washington University, Manolescu shared a key insight about the non-existence of certain holomorphic disks that led to a proof of the current much stronger form of Theorem \ref{thm:Main}.

\section{Floer homology background: Notation and Standard Constructions} \label{sec:StandardDefn}
In this section, we review standard definitions and notations, as well as prove a few simple results about sutured manifolds and sutured Floer homology.  Please see \cite{MR2253454} and \cite{MR2113019} for more details.

\begin{definition} \cite{MR723813} \label{defn:suturedman}
A {\em sutured manifold} $(Y,\Gamma)$ is a compact, oriented $3$--manifold with boundary $\partial Y$ along with a set $\Gamma \subset \partial Y$ of pairwise disjoint annuli $A(\Gamma)$ and tori $T(\Gamma)$.  The interior of each component of $A(\Gamma)$ contains a {\em suture}, an oriented simple closed curve which is homologically nontrivial in $A(\Gamma)$.  The union of the sutures is denoted $s(\Gamma)$.

Every component of $R(\Gamma) = \partial Y - \mbox{Int}(\Gamma)$ is assigned an orientation compatible with the oriented sutures.  More precisely, if $\delta$ is a component of $\partial R(\Gamma)$, endowed with the boundary orientation, then $\delta$ must represent the same homology class in $H_1(\Gamma)$ as some suture.  Let $R_+(\Gamma)$ (resp., $R_-(\Gamma)$) denote those components of $R(\Gamma)$ whose normal vectors point out of (resp., into) $Y$.
\end{definition}

Sutured manifolds can be described using sutured Heegaard diagrams.  Here (and throughout), $I$ denotes the interval $[-1,4]$.

\begin{definition} \cite[Defn. 2.7, 2.8]{MR2253454} \label{defn:suturedHD} A {\em sutured Heegaard diagram} is a tuple $(\Sigma,\boldalpha,\boldbeta)$, where $\Sigma$ is a compact, oriented surface with boundary, and $\boldalpha = \{\alpha_1, \ldots, \alpha_d\}$, $\boldbeta = \{\beta_1, \ldots, \beta_d\}$ are two sets of pairwise disjoint simple closed curves in Int$(\Sigma)$.  Every sutured Heegaard diagram uniquely defines the sutured manifold obtained by attaching $3$--dimensional $2$--handles to $\Sigma \times I$ along the curves $\alpha_i \times \{-1\}$ and $\beta_j \times \{4\}$ for $i,j \in \{1, \ldots d\}$.  $\Gamma$ is $\partial \Sigma \times I$, and $s(\Gamma) = \partial \Sigma \times \{\frac{3}{2}\}$.
\end{definition}

To define sutured Floer homology, Juh{\'a}sz restricts to a particular class of sutured manifolds.

\begin{definition} \cite[Defn. 2.2]{MR2253454} \label{defn:balancedSM} A sutured manifold $(Y,\Gamma)$ is said to be {\em balanced} if $\chi(R_+) = \chi(R_-)$, and the maps $\pi_0(\Gamma) \rightarrow \pi_0(\partial Y)$ and $\pi_0(\partial Y) \rightarrow \pi_0(Y)$ are surjective.\footnote{The equivalence of this definition to the original definition in \cite{MR2253454} is immediate.}
\end{definition}

There is a corresponding notion for Heegaard diagrams:

\begin{definition} \cite[Defn. 2.1]{MR2253454} \label{defn:balancedHD}A sutured Heegaard diagram $(\Sigma, \boldalpha, \boldbeta)$ is called {\em balanced} if $|\boldalpha| = |\boldbeta|$, $\Sigma$ has no closed components, and $\{\alpha_i\}$ (resp., $\{\beta_i\}$) are linearly-independent in $H_1(\Sigma)$.
\end{definition}

\Juhasz proves, in \cite[Prop. 2.14]{MR2253454}, that every balanced sutured manifold can be specified by means of a balanced Heegaard diagram.  From the data of a balanced Heegaard diagram \[(\Sigma, \boldalpha = \{\alpha_1, \ldots, \alpha_d\}, \boldbeta = \{\beta_1, \ldots, \beta_d\})\] and a generic (family of) complex structures on $\Sigma$, \Juhasz then defines a Floer chain complex in the standard way using the half-dimensional tori $\mathbb{T}_\alpha = \alpha_1 \times \ldots \times \alpha_d$ and $\mathbb{T}_\beta = \beta_1 \times \ldots \times \beta_d$ in $Sym^d(\Sigma)$.  Specifically, one obtains a chain complex with:

\begin{enumerate}
\item Generators: $\{{\bf x} \in \mathbb{T}_\alpha \cap \mathbb{T}_\beta\}$,
\item Differentials: \[\partial({\bf x}) = \sum_{{\bf y} \in \mathbb{T}_\alpha \cap \mathbb{T}_\beta} \sum_{\{\phi \in \pi_2({\bf x},{\bf y})| \mu(\phi) = 1\}} \widehat{\mathcal{M}}(\phi) \cdot {\bf y}.\]
\end{enumerate}

As usual, $\pi_2({\bf x},{\bf y})$ denotes the homotopy classes of disks connecting ${\bf x}$ to ${\bf y}$, $\mu(\phi)$ denotes the Maslov index of a representative of such a homotopy class, and $\widehat{\mathcal{M}}(\phi)$ denotes the moduli space of holomorphic representatives of $\phi$, modulo the standard $\R$ action.  

Denote by $CFH(Y,\Gamma)$ any chain complex associated to a balanced, sutured manifold $(Y,\Gamma)$ arising as above, and by $SFH(Y,\Gamma)$ the homology of such a chain complex.

We will repeatedly encounter the following examples of sutured manifolds:

\begin{example} \label{example:threemanhat}
Let $Y$ be a closed, oriented $3$--manifold, along with a thickening, $D \times I \subset Y$, of an imbedded, oriented disk $D$.  Then $(Y-D, \Gamma)$ will denote the sutured manifold $\overline{Y-(D \times I)}$ with $\Gamma = ((\partial D) \times I) \subset \partial(D \times I)$ and $s(\Gamma) = \partial D \times \{\frac{3}{2}\}$.  Note that \[SFH(Y-D,\Gamma) \cong \widehat{HF}(\pm Y).\]  See \cite[Ex. 2.3]{MR2253454} and the discussion in the proof of Theorem \ref{thm:KnotFloerRel} in the present paper.
\end{example}

\begin{example} \label{example:knot}
Let $K \subset Y$ be an oriented knot in a closed, oriented $3$--manifold.  Then $(Y - K, \Gamma)$ will denote the sutured manifold $Y - N(K)$, where $\Gamma$ is defined as follows.  Choose $\mu$ an imbedded curve on $T^2 = \partial(Y - N(K))$ representing an oriented meridian of $K$, $\mu'$ a parallel, oppositely-oriented copy of $\mu$.  Then $\Gamma = N(\mu) \cup N(\mu')$ and $s(\Gamma) = \mu \cup \mu'$.  Note that \[SFH(Y-K,\Gamma) \cong \widehat{HFK}(Y,K).\]  See \cite[Ex. 2.4]{MR2253454} and the discussion in the proof of Theorem \ref{thm:KnotFloerRel} in the present paper.
\end{example}

\begin{example} \label{example:prodsurface}
Let $F_{g,b}$ be an oriented surface with genus $g$ and $b$ boundary components.  Then let $(F_{g,b} \times I, \Gamma)$ be the sutured manifold with $\Gamma = \partial(F_{g,b}) \times I$ and $s(\Gamma) = \partial(F_{g,b}) \times\{\frac{3}{2}\}$.  We will denote by $D$ the disk, $F_{0,1}$.
\end{example}

For simplicity, whenever we refer to a sutured manifold of the type described in Examples \ref{example:threemanhat} - \ref{example:prodsurface}, we will drop any reference to the sutures in the notation.  E.g., $S^3 - K$ (for $K$ an oriented knot in $S^3$) will be used to denote the sutured manifold $(S^3 - K, \Gamma)$.

We will need the following three operations, which allow us to construct new balanced, sutured manifolds from old.

\begin{definition} (Gluing) \label{defn:gluing} Let $(Y_1,\Gamma_1)$, $(Y_2,\Gamma_2)$ be two sutured manifolds, and $\gamma_i \subset \Gamma_i$ a distinguished connected annular component of $\Gamma_i$ for $i=1,2$.  Then:
\[Y_1 \cup_{\gamma_i} Y_2\] will denote the sutured $3$--manifold
\[Y_1 \amalg_{\gamma_1 \sim -\gamma_2} Y_2\] illustrated in Figure \ref{fig:Gluing}.
\end{definition}

Note that 

\begin{enumerate}
\item $Y_1 \bigcup_{\gamma_i} Y_2$ has sutures $(\Gamma_1 - \gamma_1) \cup (\Gamma_2 - \gamma_2)$, and
\item $R^{\pm}(Y_1 \bigcup_{\gamma_i} Y_2) = R^{\pm}(Y_1) \amalg_{\partial\gamma_1 \sim -\partial\gamma_2} R^{\pm}(Y_2).$ 
\end{enumerate}

\begin{figure}
\begin{center}
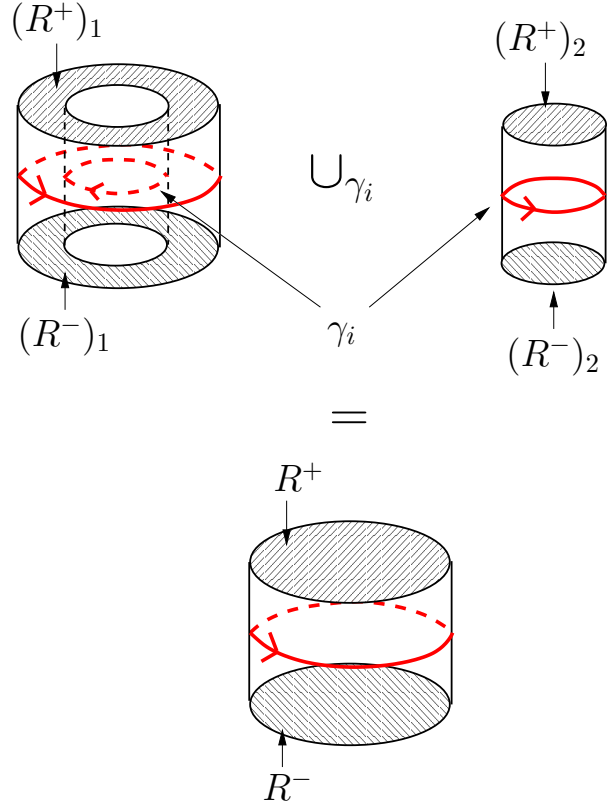
\end{center}
\caption{Two sutured manifolds $(Y_1, \Gamma_1)$ and $(Y_2, \Gamma_2)$ being glued along distinguished sutures $\gamma_i \subset \Gamma_i$ for $i=1,2$.}
\label{fig:Gluing}
\end{figure}

\begin{proposition} \label{prop:BalancedGluing}
Let $(Y_1, \Gamma_1)$ and $(Y_2, \Gamma_2)$ be balanced, sutured manifolds, and $\gamma_i \subset \Gamma_i$ distinguished connected components of $\Gamma_i$ in connected components $S_i \subset \partial(Y_i)$ for $i=1,2$.  If at least one of $\gamma_i$ for $i=1,2$ satisfies the additional property that $(S_i - \overline{\gamma}_i) \cap \Gamma_i \neq \emptyset$, then $Y_1 \cup_{\gamma_i} Y_2$ is balanced.
\end{proposition}

\begin{proof}
It is clear that $Y_1 \cup_{\gamma_i} Y_2$ has no closed components if $Y_1$ and $Y_2$ don't, since we are only gluing along a proper subset of $S_i$ for $i=1,2$.  Furthermore, the condition that $(S_i - \overline{\gamma}_i) \cap \Gamma_i \neq \emptyset$ for at least one of $i=1,2$ ensures that 
\[\pi_0(\Gamma(Y_1 \cup_{\gamma_i} Y_2)) \rightarrow \pi_0(\partial(Y_1 \cup_{\gamma_i} Y_2))\] is surjective.

The additivity of the Euler characteristic for surfaces glued along circular boundary components insures that 

\[\left(\chi(R_{+}(Y_i)) = \chi(R_{-}(Y_i)\right) \Longrightarrow \left(\chi(R_{+}(Y_1 \cup_{\gamma_i} Y_2)) = \chi(R_-(Y_1 \cup_{\gamma_i} Y_2))\right)\]
\end{proof}

There is a nice interpretation of the gluing operation in terms of their associated Heegaard diagrams.

\begin{lemma} \label{lemma:GlueHD}
If $(\Sigma, \boldalpha, \boldbeta)_i$ for $i=1,2$ are balanced sutured Heegaard diagrams representing the balanced sutured manifolds $(Y,\Gamma)_i$, and $\gamma_i \subset \Gamma_i$ satisfy the conditions in Proposition \ref{prop:BalancedGluing}, then \[\left(\Sigma_1 \,\amalg_{\gamma_1 \sim -\gamma_2}\, \Sigma_2\,,\, \boldalpha_1 \cup \boldalpha_2\,,\, \boldbeta_1 \cup \boldbeta_2\right)\] is a balanced sutured Heegaard diagram representing $Y_1 \cup_{\gamma_i} Y_2$. 
\end{lemma}

\begin{proof} Immediate from the definitions.
\end{proof}

\begin{definition} (Branched Covering)
Let $(Y,\Gamma)$ be a sutured manifold and \[(B,\partial B) \subset (Y, \partial Y)\] a smoothly imbedded, codimension $2$ submanifold satisfying 
\[\partial B \cap \Gamma = \emptyset.\]
Let $\widetilde{Y}$ be any cyclic branched cover of $Y$ over $B$ with covering map $\pi: \widetilde{Y} \rightarrow Y$.

Then we denote by $(\widetilde{Y}, \widetilde{\Gamma})$ the sutured manifold with sutures $s(\widetilde{\Gamma}) = \pi^{-1}(s(\Gamma))$.

Of special interest to us is the sutured branched double cover, which we will denote by $\boldSigma((Y,\Gamma),B)$.
\end{definition}

\begin{proposition} \label{prop:BalancedCovers}
Let $(Y,\Gamma)$ be a balanced, sutured manifold and $(B,\partial B) \subset (Y, \partial Y)$ a smoothly imbedded codimension $2$ submanifold satisfying 
\[\partial B \cap \Gamma = \emptyset.\]
Let $\widetilde{Y}$ be any cyclic branched cover of $(Y,B)$ with covering map $\pi: \widetilde{Y} \rightarrow Y$.
If \[\#(\partial B \cap R_+) = \#(\partial B \cap R_-),\]
then $(\widetilde{Y}, \widetilde{\Gamma})$ is balanced.
\end{proposition}

In the above, $\#$ denotes geometric, not algebraic, intersection number.

\begin{proof} 

To show that $\chi(\widetilde{R}_+) = \chi(\widetilde{R}_-)$, note that the branched covering restricts to a branched covering of the boundary over $\partial B \subset \partial Y$.  Let $n= \#(\partial B \cap R_{\pm})$ and $k$ be the order of the covering $\widetilde{Y} \rightarrow Y$.  By the Riemann-Hurwitz formula,
\[\chi(\widetilde{R}_{\pm}) = k(\chi(R_{\pm})) - (k1)n.\]
Since $Y$ is balanced, $\chi(R_+) = \chi(R_-)$, which implies $\chi(\widetilde{R}_+) = \chi(\widetilde{R}_-)$, as desired.

To show that $\pi_0(\widetilde{\Gamma}) \rightarrow \pi_0(\partial\widetilde{Y})$ is surjective, note that the surjectivity of $\pi_0(\Gamma) \rightarrow \pi_0(\partial Y)$ implies that for every point $p \in R_+ \cup R_-$ there exists some point $q \in s(\Gamma)$ and a path $\eta_t: [0,1] \rightarrow \partial Y$ from $p$ to $q$.

Now, suppose that there is some connected component $Y_0$ of $\partial \widetilde{Y}$ satisfying $Y_0 \cap \widetilde{\Gamma} = \emptyset$.  Then either $Y_0 \subset R_+$ or $Y_0 \subset R_-$.  For definiteness, assume the former.  Pick a point $\widetilde{p} \in Y_0$ and consider its projection, $p \in Y$.  As noted above, there exists a path $\eta_t$ from $p$ to some $q \in s(\Gamma)$.  By the path lifting property (cf. \cite{MR1867354}), $\eta_t$ lifts to a path $\widetilde{\eta}_t$ from $\widetilde{p}$ to $\widetilde{q} \in s(\widetilde{\Gamma})$, implying that, in fact, $Y_0 \cap \widetilde{\Gamma} \neq \emptyset$, as desired.

An analogous argument proves that $\pi_0(\partial\widetilde{Y}) \rightarrow \pi_0(\widetilde{Y})$ is surjective.
\end{proof}

We will also need to understand the behavior of sutured Floer homology under so-called surface decompositions:

\begin{definition} \cite[Defn. 2.4]{MR2390347}  Let $(Y, \Gamma)$ be a sutured manifold.  A {\em decomposing surface} is a properly imbedded oriented surface $S$ in $Y$ such that for every component $\lambda$ of $S \cap \Gamma$, one of the following holds:
\begin{enumerate}
  \item $\lambda$ is a properly imbedded non-separating arc in $\Gamma$ with $\#(\lambda \cap s(\Gamma)) = 1$.
  \item $\lambda$ is a simple closed curve in an annular component $A$ of $\Gamma$ in the same homology class as $s(\Gamma)$.
  \item $\lambda$ is a homotopically nontrivial curve in a torus component $T$ of $\Gamma$, and if $\delta$ is another component of $T \cap S$, $[\lambda] = [\delta] \in H_1(T)$.
\end{enumerate}
\end{definition}

\begin{definition} If $S$ is a {\em decomposing surface} in the sutured manifold $(Y,\Gamma)$, $S$ defines a {\em sutured manifold decomposition} 
\[(Y, \Gamma) \rightsquigarrow^{S} (Y',\Gamma'),\]
where $Y' = Y - \mbox{Int}(N(S))$ and

\begin{eqnarray*}
\Gamma' &=& (\Gamma \cap Y') \cup N(S_+' \cap R_-(\Gamma)) \cup N(S'_- \cap R_+(\Gamma)),\\
R_+(\Gamma') &=& ((R_+(\Gamma) \cap Y') \cup S_+') - \mbox{Int}(\Gamma'),\\
R_-(\Gamma') &=& ((R_-(\Gamma) \cap Y') \cup S_-') - \mbox{Int}(\Gamma').
\end{eqnarray*}

Here, $S'_+$ (resp., $S_-'$) is the component of $\partial N(S) \cap Y'$ whose normal vector points out of (resp., into) $Y'$.
\end{definition}

We refer the reader to \cite{MR2390347} for the remaining definitions and results about the behavior of sutured Floer homology under surface decompositions.  In particular, Theorem 1.3, Definition 4.3, and Lemmas 4.5 and 5.4 of \cite{MR2390347} will be indispensable to us in the proof of Theorem \ref{thm:KnotFloerRel}.

We close this section with a proof of the following algebro-topological fact, implicit in \cite[Sec. 3]{MR2253454}.  Compare also \cite[Sec. 2.5]{MR2113019}.

\begin{proposition}\label{prop:H1andH2}
Let $(\Sigma,\boldalpha,\boldbeta)$ be a balanced sutured Heegaard diagram for $(Y,\Gamma)$.  Then there is a natural identification
\[H_2(Y;\Z) \cong Ker\left(Span([\boldalpha,\boldbeta]) \rightarrow H_1(\Sigma;\Z)\right)\]
%\[H_1(Y;\Z) \cong Coker\left(Span([\boldalpha,\boldbeta]) \rightarrow H_1(\Sigma;\Z)\right).\]
\end{proposition}

\begin{proof} Use Mayer-Vietoris on \[Y = U_\alpha \bigcup_{\Sigma \subset \partial(U_\alpha) \sim -\partial(U_\beta)} U_\beta,\] where $U_\alpha := f^{-1}[-1,\frac{3}{2}]$ (resp., $U_\beta := f^{-1}[\frac{3}{2},4]$) for $f$ a self-indexing Morse function as in \cite[Prop. 2.13]{MR2253454}:
\[\begin{CD}
H_2(U_\alpha) \oplus H_2(U_\beta) @>>> H_2(Y) @>>> H_1(\Sigma) @>f>> H_1(U_\alpha) \oplus H_1(U_\beta).
\end{CD}\]
Since $H_2(U_\alpha) \cong H_2(U_\beta) = 0$ (using, e.g., the long exact sequence on the pair $(U_\alpha, \Sigma)$ or $(U_\beta,\Sigma)$), $H_2(Y) \cong Ker(f).$  But $Ker(f)$ consists of those elements of $H_1(\Sigma)$ which map to $0$ in both $H_1(U_\alpha)$ and $H_1(U_\beta)$ under the inclusion maps.  This is precisely \[Ker(Span([\boldalpha,\boldbeta]) \rightarrow H_1(\Sigma)),\] as desired.
\end{proof}

\subsection{Maslov Index}
Let $(\Sigma, \boldalpha,\boldbeta)$ be a balanced, sutured Heegaard diagram and ${\bf x} = (x_1, \ldots, x_d)$, ${\bf y} = (y_1, \ldots, y_d)$ two intersection points in $\mathbb{T}_{\alpha} \cap \mathbb{T}_{\beta} \subset Sym^d(\Sigma)$.  The purpose of this section is to review the arguments behind Lipshitz's formula \cite{MR2240908} for the Maslov index (formal expected dimension) of the moduli space of holomorphic representatives of $\phi \in \pi_2({\bf x},{\bf y})$, verifying that they are valid in the context of sutured Floer homology.  We will need this formula in order to understand the grading shifts discussed in Section \ref{sec:Gradings}.

Before stating Lipshitz's formula, we need a couple of definitions (from \cite{MR2240908}).  In what follows, let $D$ be a positive domain in $\Sigma$ representing $\phi \in \pi_2({\bf x},{\bf y})$.

\begin{definition}  Let $e(D)$ denote the {\it Euler measure} of $D$.  The Euler measure is additive under disjoint union and gluing components along boundaries (cf. \cite{MR2240908}).  Expressing $D$ as a $\mathbb{Z}_{\geq 0}$-linear combination of the connected components $\mathcal{D}_1, \ldots \mathcal{D}_n$ of $\Sigma - \vec{\alpha} - \vec{\beta}$, $e(D)$ is given by:\[e(D) = \sum_i e(\mathcal{D}_i),\]
where \[e(\mathcal{D}_i) = \chi(\mathcal{D}_i) - \frac{k(\mathcal{D}_i)}{4} + \frac{\ell(\mathcal{D}_i)}{4},\]
and $k$ (resp., $\ell$) is the number of acute (resp., obtuse) right-angled corners of $\mathcal{D}_i$.  $\chi(\mathcal{D}_i)$ is the Euler characteristic.
\end{definition}

\begin{definition}
\[n_{\bf x}(D) := \sum_{i=1}^d n_{x_i}(D),\] where $n_{x_i}(D)$ is the average of the coefficients of $D$ in the four domains adjacent to $x_i$.  In other words, if one chooses points $z_I, z_{II}, z_{III}, z_{IV}$ in the four domains adjacent to $x_i$, then
\[n_{x_i}(D) = \frac{1}{4}\left(n_{z_I}(D) + n_{z_{II}}(D) + n_{z_{III}}(D) + n_{z_{IV}}(D)\right).\]

\end{definition}

\begin{proposition} \cite[Cor. 4.3]{MR2240908} \label{prop:MasIndex}
Let $D$ be a positive domain in $\Sigma$ representing $\phi \in \pi_2({\bf x},{\bf y})$.  Then 
\[\mu(\phi) = e(D) + n_{\bf x}(D) + n_{\bf y}(D).\]
\end{proposition}

\begin{proof}
We must first check that Lipshitz's cylindrical reformulation of Heegaard Floer homology applies in the sutured Floer homology setting. To this end, let $\widehat{\Sigma}$ be the closed surface obtained by capping off the boundary components of $\Sigma$ with disks.  Choose a point $z_i$ in the interior of each of the capping disks.  Stabilize $\widehat{\Sigma}$, if necessary, to ensure that $d>1$.  

Now, as in \cite{MR2240908} (see also \cite[Sec. 5.2]{GT0512286}), we can form the $4$--manifold \[W = \widehat{\Sigma} \times [0,1] \times \mathbb{R}.\]  Let \[C_\alpha = \bigcup_{i=1}^d \alpha_i \times \{1\} \times \R, \,\,\, C_\beta = \bigcup_{i=1}^d \beta_i \times \{0\} \times \R,\]  and 
\[\pi_{\widehat{\Sigma}}: W \rightarrow \widehat{\Sigma}, \,\, \pi_{\mathbb{D}}: W \rightarrow [0,1] \times \mathbb{R}\] the projection maps.  Lipshitz proves, in \cite[App. A]{MR2240908}, that the chain complex he defines coincides with the Heegaard Floer chain complex.  In particular, he shows that with appropriately generic choices (see \cite[Sec. 1]{MR2240908}), a $J$-holomorphic map $u: S \rightarrow W$ of a surface with boundary $S$ with $d$ positive punctures ${\bf x} = x_1, \ldots, x_d$ and $d$ negative punctures ${\bf y} = y_1, \ldots, y_d$ which has zero intersection with the subvarieties $\{z_i\} \times [0,1]\times \mathbb{R}$ and which satisfies conditions ({\bf M0}) - ({\bf M6}) in \cite[Sec. 1]{MR2240908} corresponds to a $J$-holomorphic map $\phi: \mathbb{D} \rightarrow Sym^d(\Sigma)$ in the homotopy class $\pi_2({\bf x},{\bf y})$ in the sutured Floer homology setting.  None of the arguments proving equivalence of the two moduli spaces require $d= g(\widehat{\Sigma})$.

%\begin{enumerate}
%\item ({\bf M0}) $u$ is smooth (non-nodal),
%\item ({\bf M1}) $u(\partial(S)) \subset C_\alpha \cup C_\beta$,
%\item ({\bf M2}) There are no components of $S$ on which $\pi_{\mathbb{D}} \circ u$ is constant,
%\item ({\bf M3}) For each $i$, $u^{-1}(\alpha_i \times \{1\} \times \mathbb{R})$ and $u^{-1}(\beta_i \times \{0\} \times \mathbb{R})$ each consist of a single component of $\partial S \ \{x_1, \ldots, x_d, y_1, \ldots, y_d\}$, 
%\item ({\bf M4}) Any sequence of points in $S$ converging to $y_i$ (resp., $x_i$) is mapped under $\pi_{\mathbb{D}}$ to a sequence of points whose second coordinate converges to $\{-\infty\}$ (resp., $\{+\infty\}$),
%\item ({\bf M5}) The energy of $u$ is finite,
%\item ({\bf M6}) $u$ is an embedding,
%\end{enumerate}

Next, we must check that Lipshitz's calculation of the expected dimension of the moduli space of $J$-holomorphic maps $u: S \rightarrow W$ as above does not need $d= g(\widehat{\Sigma})$.

To verify this, notice that the Maslov index \cite[Eqn. 6]{MR2240908}:
\[\mbox{ind}(u) = d - \chi(S) + 2e(D)\]
depends only upon the pullback along $\pi_{\Sigma} \circ u$ of the complex bundle $T\Sigma$ to $S$.  In particular, this index formula is valid for any complex line bundle on $S$ satisfying the required ``matching conditions'' on the boundary of $S$ and, hence, does not require that $d=g(\widehat{\Sigma})$.

Furthermore, \cite[Prop. 4.2]{MR2240908} computes $\chi(S)$ in terms of data on the domain $D = (\pi_{\Sigma} \circ u)(S)$ which represents $\phi \in \pi_2({\bf x},{\bf y})$:
\[\chi(S) = d - n_{\bf x}(D) - n_{\bf y}(D) + e(D).\]
Lipshitz's proof of this Proposition does not depend upon $g(\widehat{\Sigma})$.
The conclusion:
\[\mbox{ind}(u) = \mu(\phi) = e(D) + n_{\bf x}(D) + n_{\bf y}(D)\]
follows.
\end{proof}

\section{Sutured Heegaard Multi-Diagrams and Polygons} \label{sec:MultiDiag}
To describe the differentials in the filtered chain complexes underlying the spectral sequences involved in the proof of Theorem \ref{thm:SpecSeq}, we will need to define sutured Heegaard multi-diagrams, the natural analogue of traditional Heegaard multi-diagrams, discussed in \cite{MR2113019} and \cite{MR2141852}.

\begin{definition}
A {\em balanced sutured Heegaard multi-diagram} is a tuple $(\Sigma, \boldeta^0, \boldeta^1, \ldots, \boldeta^n)$ where 
\begin{enumerate} 
\item $\Sigma$ is a compact, oriented surface with boundary, having no closed components.  %Although it is not necessary in general, we will assume throughout this paper that $\Sigma$ is connected.
\item $\boldeta^i = \{\eta^i_1, \ldots, \eta^i_d\}$ for $i=0,\ldots, n$, $d$ a fixed non-negative integer, is a collection of pairwise disjoint simple closed curves in $\mbox{Int}(\Sigma)$, which are linearly independent in $H_1(\Sigma)$.
\end{enumerate}
\end{definition}

As usual, this definition is closely related to a certain four-dimensional cobordism between sutured $3$--manifolds.  As in \cite{MR2253454}, we associate to each $d$-tuple, $\boldeta^i$, of linearly-independent curves a $3$--manifold $U_i$, obtained by attaching $2$--handles to $\Sigma \times I$ along $\boldeta^i \subset \Sigma \times \{-1\}$.

As in \cite[Sec. 8]{MR2113019}, we can now construct from the tuple $\{\boldeta^0, \ldots \boldeta^{n}\}$ the following natural $4$--dimensional identification space.  Let $P_{n+1}$ denote a topological $(n+1)$-gon, with vertices labeled $v_{i}$ for $i \in \Z_{n+1}$, labeled in a clockwise fashion.  Denote the edge connecting $v_i$ to $v_{i+1}$ by $e_i$.  Then let \[X_{\eta^0, \ldots, \eta^n} := \frac{(P_{n+1} \times \Sigma) \coprod_{i=0}^n (e_i \times U_i)}{(e_i \times \Sigma) \sim (e_i \times \partial U_i)}\]

We will often denote the four-manifold constructed above by $X$, for short.  Note that the identification between $\Sigma$ and $\partial U_i$ occurs only along the portion of $\partial U_i$ corresponding naturally to $\Sigma$.  More specifically, 
\begin{equation} \label{eqn:DecompUi}\partial U_i = (\Sigma) \cup (\partial\Sigma \times I) \cup (\Sigma_{\eta^i}),
\end{equation}
where $\Sigma_{\eta^i}$ is the result of performing surgery to $\Sigma$ along all of the imbedded $\eta^i_j$ curves.  The identification between $(e_i \times \Sigma)$ and $(e_i \times \partial U_i)$ takes place only along the first term of the decomposition in (\ref{eqn:DecompUi}).

In order to set up the appropriate admissibility hypotheses ensuring the finiteness of holomorphic polygon counts, we will need some basic results about the algebraic topology of $X$ and its relationship with homotopy classes of topological $(n+1)$--gons.

Let ${\bf x}_{i+1} \in \mathbb{T}_{\eta^i} \cap \mathbb{T}_{\eta^{i+1}}$ for $i \in \Z_{n+1}$ and let $\pi_2({\bf x}_0, \ldots, {\bf x}_n)$ denote the set of homotopy classes of Whitney $(n+1)$--gons in the sense of \cite[Sec. 8.1.2]{MR2113019}.  Proposition \ref{prop:AffineIdent} below implies that any two Whitney $(n+1)$--gons in $\pi_2({\bf x}_0, \ldots, {\bf x}_n)$ differ by the addition of a so-called  {\em $(n+1)$--periodic domain}.

\begin{definition} An {\em $(n+1)$--periodic domain} $\mathcal{P}$ is a $2$--chain on $\Sigma$ whose boundary is a $\mathbb{Z}$-linear combination of curves in $\boldeta^0, \ldots, \boldeta^n$. 
\end{definition}

\begin{proposition}  \label{prop:AffineIdent}Let $d>2$ and ${\bf x}_{i+1} \in \mathbb{T}_{\eta^i} \cap \mathbb{T}_{\eta^{i+1}}$ for $i \in \Z_{n+1}$.  If $\pi_2({\bf x}_0, \ldots, {\bf x}_{n})$ is non-empty, then 
\[\pi_2({\bf x}_0, \ldots, {\bf x}_{n}) \cong \mbox{Ker}\left(\bigoplus_{i=0}^n\mbox{Span }([\eta^i_j]_{j=1}^d) \rightarrow H_1(\Sigma;\Z)\right)\]
\end{proposition}

The above is an affine correspondence.

\begin{proof}
The proof follows exactly as in the proof of \cite[Prop. 8.3]{MR2113019}, except that in our case \[\pi_2(Sym^d(\Sigma)) \cong 0,\]
since $\Sigma$ has non-empty boundary.  To see this, we adapt the argument in the proof of \cite[Prop. 2.7]{MR2113019}.  Suppose that $\Sigma = F_{g,b}$ is a genus $g$ surface with $b>0$ boundary components.  Then $\Sigma$ is homotopy-equivalent to a wedge of $2g + (b-1)$ circles, hence, Sym$^d(\Sigma)$ is homotopy equivalent to Sym$^d(\mathbb{C} - \{z_1, \ldots, z_{2g+(b-1)}\})$, which is naturally identified with the space of monic polynomials $p$ of degree $d$, satisfying $p(z_i) \neq 0$.  By considering the coefficients of the polynomials, this space, in turn, is naturally identified with $\mathbb{C}^d$ minus $2g+(b-1)$ generic hyperplanes $H_i$.  \cite[Thm. 3]{MR0379883} then implies that $\pi_2(\mathbb{C}^d - H_1 - \ldots H_{2g+(b-1)}) = 0$ when $d>2$.

Hence, we obtain \[\pi_2({\bf x}_0, \ldots, {\bf x}_n) \cong \mbox{Ker}\left(\bigoplus_{i=0}^n\mbox{Span }([\eta^i_j]_{j=1}^d) \rightarrow H_1(\Sigma;\Z)\right)\]
as desired.
\end{proof}

Furthermore, we have:

\begin{proposition} \label{prop:PerDomains}(analogue of \cite[Prop. 8.2]{MR2113019}) There is a natural identification 
\[H_2(X;\Z) \cong \mbox{Ker}\left(\bigoplus_{i=0}^n\mbox{Span }([\eta^i_j]_{j=1}^d) \rightarrow H_1(\Sigma;\Z)\right),\]
and
\[H_1(X;\Z) \cong \mbox{Coker}\left(\bigoplus_{i=0}^n \mbox{Span }([\eta^i_j]_{j=1}^d) \rightarrow H_1(\Sigma;\Z)\right).\]

\end{proposition}

\begin{proof} Just as in the proof of \cite[Prop. 8.2]{MR2113019}, we examine the long exact sequence of the pair $(X,P_{n+1} \times \Sigma)$.  As in that proof, the boundary homomorphism $$\partial: H_2(U_{i},\Sigma;\Z) \rightarrow H_1(\Sigma;\Z)$$ is injective, and its image is $Span([\eta^i_j]_{j=1}^d)$.

The conclusion then follows by noting that 

\begin{enumerate}
\item $H_2(P_{n+1} \times \Sigma;\Z) = 0$, since $\Sigma$ is not closed,
\item $H_2(X, P_{n+1} \times \Sigma;\Z) \cong \bigoplus_{i=0}^n H_2 (U_{i}, \Sigma;\Z)$ by excision, and
\item $H_1(X,P_{n+1} \times \Sigma;\Z) = 0$ also by excision.
\end{enumerate}
\end{proof}

The correspondence in Proposition \ref{prop:PerDomains} can be made explicit by associating to each periodic domain $\mathcal{P}$ an element of $H_2(X;\mathbb{Z})$ as follows.  For each $\eta^i_j$, let $E^i_j$ denote the core disk of the associated $2$--handle and suppose
\[\partial(\mathcal{P}) = \sum_{i,j} e_{ij} \eta^i_j.\]
Then 

\[\mathcal{H}({\mathcal{P}}) := \mathcal{P} + \sum_{i,j} e_{ij} E^i_j \in H_2(X;\mathbb{Z}).\]

\subsection{Constructing Spin$^c$ Structures on $X_{\eta^{i}}$} \label{sec:Spinc}
As in \cite{MR2113019}, we obtain a natural map from homotopy classes of $n$-gons to Spin$^c$ structures on $X$.  We need to understand Spin$^c$ structures on $X_{\eta^i}$ in order to formulate the correct admissibility hypotheses for the sutured multi-diagrams of interest to us in the present work.

We begin with some standard definitions and facts about relative Spin$^c$ structures on $3$-- and $4$--manifolds (with and without boundary).  See \cite{MR1484699}, \cite[Sec. 2.6 \& 8.1.3]{MR2113019}, and \cite[Sec. 3.2]{GT0512286} for more details.

\begin{definition} Let $(X,Z)$ pe a pair, with $X$ a $3$-- or $4$--manifold (possibly with boundary), and $Z \subset X$ a closed, smoothly imbedded submanifold (possibly with boundary).  Then a {\em relative Spin$^c$ structure}, in Spin$^c(X,Z)$, is a homology class of pairs $(J,P)$, where 
    \begin{itemize}
      \item $P \subset X- Z$ is a finite collection of points, 
      \item $J$ is an almost-complex structure (equivalent to an oriented $2$--plane field, when $X$ is oriented and equipped with a Riemannian metric) defined over $X - P$, extending a particular fixed almost-complex structure on $Z$.  
    \end{itemize}
Two pairs $(J_1, P_1)$ and $(J_2,P_2)$ are said to be homologous if there exists a compact $1$--manifold $C \subset X - Z$ with $P_1,P_2 \subset \partial(C)$ satisfying $J_1|_{X-C} \sim J_2|_{X-C}$, where $\sim$ denotes isotopy fixing the almost complex structure on $Z$.  Spin$^c(X,Z)$ is an affine set for the action of $H^2(X,Z;\Z)$.
\end{definition}

\begin{remark}
  If $(Y,\Gamma)$ is a connected, balanced, sutured $3$--manifold, then we denote by $Spin^c(Y,\Gamma)$ the set of relative Spin$^c$ structures for the pair $(Y,\partial Y)$ in the sense of the above definition.  In other words, elements of $Spin^c(Y,\Gamma)$ are homology classes of nowhere-vanishing vector fields on $Y$, all of which agree with a particular vector field, $v_0$, on $\partial Y$. Note that this is equivalent to an oriented $2$--plane field on $Y$ extending a particular $2$--plane field on $\partial Y$, when $Y$ is oriented and equipped with a Riemannian metric.  To define $v_0$ (see \cite{MR2253454}), recall that if $Y$ is a sutured manifold, then 
\[\partial(Y) = R^+ \cup R^- \cup \Gamma.\]
Furthermore, $\Gamma$ is naturally identified with $s(\Gamma) \times I$, with $\partial(R^-) = s(\Gamma) \times \{-1\}$, and $\partial(R^+) = s(\Gamma) \times \{4\}$.  Then $v_0$ is defined to point out of $Y$ along $R^+$, into $Y$ along $R^-$, and along the gradient of the height function $s(\gamma) \times I \rightarrow I$ along $\gamma$.  Here, two vector fields $v_1$ and $v_2$ are said to be homologous if there exists a finite set of points, $P \subset (Y - \partial Y)$, such that $v_1|_{Y-P} \sim v_2|_{Y-P}$.  Spin$^c(Y)$ is an affine set for the action of $H^2(Y,\partial Y;\Z)$.
\end{remark}

Now let 

\begin{itemize}
\item $X$ be the $4$--manifold associated to a sutured Heegaard multi-diagram, 
\item $Y = \partial X$,
\item $Y' = Y_{\eta^0, \eta^1} \cup \ldots \cup Y_{\eta^{n-1},\eta^n} \cup -Y_{\eta^0,\eta^n} \subset Y$, and
\item $Z = \overline{Y - Y'}$.  
\end{itemize}

Note that $Y_{\eta^i,\eta^{i+1}}$ is a sutured Heegaard diagram for each $i \in \Z_n$.  \Juhasz defines a map \cite[Defn. 4.5]{MR2253454}:
\[\mathfrak{s}: \mathbb{T}_{\eta^i} \cap \mathbb{T}_{\eta^{i+1}} \rightarrow \mbox{Spin}^c(Y_{\eta^i,\eta^{i+1}}, \Gamma).\] 

%Let $R_i^{\pm}$ be the positive and negative subsets $\partial(Y_{\eta^i,\eta^{i+1}})$.

%\begin{proposition} $Y = \frac{\coprod Y' \cup (S^1 \times D^2)}{R_i^+ \sim -R_{i+1}^-}$ subject to the additional identification of each $S^1 \times \{z\}$ where $z \in \partial(D^2)$ with $\gamma \times I$, where $\gamma$ is the suture on $Y_{\eta^i,\eta^{i+1}}$.
%\end{proposition}

%(*****Need to say the above more precisely; probably also need a figure.****).
We further have:

\begin{proposition}
There is a well-defined map $\mathfrak{s}_X: \pi_2({\bf x}_0, \ldots, {\bf x}_n) \rightarrow \mbox{Spin}^c(X,Z)$ satisfying the property that \[\mathfrak{s}_X(\Psi)|_{Y_{\eta^i,\eta^{i+1}}} = \mathfrak{s}({\bf x}_{i+1}) \,\, \forall \,\, i \in \Z_{n+1}\]
\end{proposition}

\begin{proof}
We will closely follow the construction given in \cite[Sec. 8.1.4]{MR2113019}, making alterations as necessary.

Let $U_{i}$ be the (relative) handlebody constructed by attaching $3$--dimensional $2$--handles to $\Sigma \times [-1, \frac{3}{2}]$ along $\{\eta^i\} \times \{-1\}$.  Extend the product orientation on $\Sigma \times [-1,\frac{3}{2}]$ in the standard way to obtain the orientation on $U_i$.  Note that $\partial U_i$ is the union of three pieces:

\begin{itemize}
\item $\Sigma := \Sigma \times \{\frac{3}{2}\}$,
\item $\gamma^- := \partial \Sigma \times [-1,\frac{3}{2}]$,
\item $R^-_i := -\left(\overline{\partial U_i - \Sigma - \gamma^-}\right)$, the surgered surface obtained by adding compression disks to $\Sigma$ along $\{\eta^i\}$.  The $-$ sign above indicates that we equip $R_i^-$ with the orientation {\em opposite} to the boundary orientation on $U_i$.
\end{itemize}

Begin by constructing Morse functions $f_i: U_i \rightarrow [-1,\frac{3}{2}]$ with the following properties:

\begin{itemize}
\item $f_i^{-1}(-1) = R_i^-$,
\item $f_i^{-1}(\frac{3}{2}) = \Sigma$,
\item $f_i$ has $d$ index $1$ critical points in the interior of $Y_{\eta^i}$,
\item $f|_{\gamma^-}$ is the projection map to $[-1,\frac{3}{2}]$ described above.
\end{itemize}

We can now specify a $2$--plane field on $Z = \overline{Y - Y'}$ which extends any $2$--plane field on $Y'$ coming from a Spin$^c$ structure on $Y'$ as follows.  Let $Z = Z_1 \cup Z_2,$ where \[Z_1 = \partial(\Sigma) \times P_{n+1},\] and \[Z_2 = \bigcup_{i=0}^n (R^-_i \times e_i).\]

Then, along $Z_1$, we choose the $2$--plane field tangent to $\Sigma$, and along $Z_2$, we choose the $2$--plane field tangent to $R_i^-$.  By construction, this $2$--plane field on $Z$ agrees with the $2$--plane field associated to any Spin$^c$ structure on $Y$.

Given a generic map $u: P_{n+1} \rightarrow Sym^d(\Sigma)$ representing $\phi \in \pi_2({\bf x}_0, \ldots, {\bf x}_n)$, we proceed to extend this $2$--plane field to the complement of a contractible $1$--complex in $X$, producing an element of Spin$^c(X,Z)$, as desired.

Define $F$ to be the immersed surface whose intersection with $U_i \times e_i$ is the $d$-tuple of gradient flowlines of $f_i$ connecting the $d$ index one critical points of $f_i$ with the points $(x,u(x)) \subset e_i\times \Sigma$ and whose intersection with $P_{n+1} \times \Sigma$ is the collection of points $(x,\sigma)$ such that $\sigma \in u(x)$.

In the complement of $F$, we let the $2$--plane field agree with $T\Sigma$ inside $P_{n+1} \times \Sigma$ and $\ker(df_i)$ in $TU_i \subset T(U_i \times e_i)$.  This $2$--plane field is clearly well-defined on $(P_{n+1} \times \Sigma) - F$, and it is well-defined on $(U_i \times e_i) - F$ since $crit(f_i) \subset F$.  Furthermore, these $2$--plane fields agree on $\partial(P_{n+1}) \times \Sigma$ by the properties imposed on $f_i$. 

We extend the $2$--plane field to the complement of a contractible $1$-complex by choosing a point $x \in P_{n+1}$ and $n+1$ straight paths $a_0, \ldots a_n$ to the edges $e_0, \ldots e_n$.  There then exists a natural foliation of $P_{n+1} - \bigcup_{i=0}^n a_i$ by line segments connecting pairs of edges.  The extension of the $2$--plane field to the complement in $F$ of $\bigcup_{i=0}^n a_i \cup (F \cap \Delta)$, where $\Delta$ is the diagonal subspace of Sym$^d(\Sigma)$ (i.e., unordered $d$-tuples in $\Sigma$ with at least one repeated entry), now proceeds exactly as in \cite[Sec. 8.1.4]{MR2113019}, yielding an oriented $2$--plane field in the complement of a contractible $1$--complex of $X$ which agrees with the standard $2$--plane field on $Z$.  In this way, one produces an element of Spin$^c(X,Z)$ associated to a map of an $(n+1)$--gon $u: P_{n+1} \rightarrow Sym^d(\Sigma)$

Note that when $d=0$, $X = P_{n+1} \times \Sigma$, and there is a unique map $u: D^n \rightarrow (Sym^d(\Sigma) \sim pt.)$.  In this case, the construction described above yields a $2$--plane field on $X$ which is everywhere tangent to $\Sigma$.

Now that we have specified an extension,\[\mbox{Spin}^c(Y',\partial Y') \rightarrow \mbox{Spin}^c(X,Z),\]
associated to a particular $(n+1)$--gon representative of $\pi_2({\bf x}_0, \ldots {\bf x}_n)$, we will show that this extension is well-defined.  In other words,

\begin{lemma}
Let $\Psi, \Psi'$ be two $(n+1)$--gons representing the same homotopy class in $\pi_2({\bf x}_0, \ldots, {\bf x}_n)$.  Then they induce the same Spin$^c$ extension \[\mbox{Spin}^c(Y',\partial Y') \rightarrow \mbox{Spin}^c(X,Z).\]
\end{lemma}

\begin{proof}
By definition, a Spin$^c$ structure on $(Y',\partial Y')$ is a $\Z$ lift of the relative cohomology class, \[w_2 \in H^2(Y',\partial Y';\Z_2).\]  By excision, this lift is canonically identified with a relative cohomology class, $c \in H^2(Y,Z;\Z)$.  We have the following commutative diagram:
\[
\xymatrix{
  H^2(Y,Z;\Z) \ar[r]^{q_*} & H^2(Y,Z;\Z_2)\\
  H^2(X,Z;\Z) \ar[r]^{q_*} \ar[u]^{i^*} & H^2(X,Z;\Z_2) \ar[u]^{i^*}\\
  H_2(X;\Z) \cong  H^2(X,Y;\Z) \ar[u]^{\pi^*}& \\
  H_2(Y';\Z) \cong  H^1(Y,Z;\Z) \ar[u]^{\delta}&}
\]

The horizontal maps are the induced maps on cohomology coming from the short exact sequence on coefficients:
\[ 
\xymatrix{
  0 \ar[r] & \Z \ar[r]^{\times 2} & \Z \ar[r]^{q} & \Z_2 \ar[r] & 0}
\]

and the vertical maps are the induced maps on cohomology coming from the short exact sequence of the triple $(X,Y,Z)$.

Suppose that $\Psi,\Psi'$ induce Spin$^c$ structures $c_\Psi$, $c_\Psi'$, extending a particular Spin$^c$ structure $c \in H^2(Y,Z;\Z).$  Since $\Psi, \Psi'$ represent the same element of $\pi_2({\bf x}_0, \ldots, {\bf x}_n)$, their difference is an $(n+1)$--periodic domain $\phi$ satisfying \[[\phi] = 0  \in H_2(X;\Z) \cong \pi_2({\bf x}_0, \ldots, {\bf x}_n).\]  Since \[\pi^*(\phi) = \mathfrak{s}_X(\Psi) - \mathfrak{s}_X(\Psi') = 0 \in H^2(X,Z;\Z),\]
$\Psi$ and $\Psi'$ induce the same Spin$^c$ extension.
\end{proof}
\end{proof}

The following proposition tells us that two $(n+1)$--gons in the same homotopy class, $\pi_2({\bf x}_0, \ldots, {\bf x}_n)$, represent the same element of Spin$^c(X,Z)$ iff their difference is a linear combination of doubly-periodic domains.

\begin{proposition}  Let $\Psi, \Psi' \in \pi_2({\bf x}_0, \ldots, {\bf x}_n)$ be two $(n+1)$--gons in the same homotopy class.  Then $\mathfrak{s}(\Psi) = \mathfrak{s}(\Psi')$ iff the $(n+1)$--periodic domain $\phi = \Psi - \Psi'$ can be written as a $\Z$--linear combination of doubly-periodic domains.
\end{proposition}

\begin{proof} We use the same commutative diagram as before:
\[
\xymatrix{
  H^2(Y,Z;\Z) \ar[r]^{q_*} & H^2(Y,Z;\Z_2)\\
  H^2(X,Z;\Z) \ar[r]^{q_*} \ar[u]^{i^*} & H^2(X,Z;\Z_2) \ar[u]^{i^*}\\
  H_2(X;\Z) \cong  H^2(X,Y;\Z) \ar[u]^{\pi^*}& \\
  H_2(Y';\Z) \cong  H^1(Y,Z;\Z) \ar[u]^{\delta}&}
\]
which tells us that
\[\mathfrak{s}_X(\Psi) = \mathfrak{s}_X(\Psi') \,\,\Longleftrightarrow \,\,\phi = \Psi- \Psi' \in \ker(\pi^*) = \mbox{im}(\delta).\]  But, since $H_2(Y';\Z)$ is identified with the space of doubly-periodic domains (Proposition \ref{prop:H1andH2}), the condition on the right is precisely the condition that $\phi$ can be expressed as the sum of doubly-periodic domains.
\end{proof}

\subsection{Admissibility} \label{sec:Admissible}
In order to ensure that the relevant holomorphic $(n+1)$--gon counts are finite, we will need to prove that our sutured Heegaard multi-diagrams can be made admissible in a suitable sense.

\begin{definition} A sutured Heegaard multi-diagram is {\em admissible} if every non-trivial $(n+1)$--periodic domain has both positive and negative coefficients.
\end{definition}

\begin{remark}  To count holomorphic $(n+1)$--gons representing a particular (equivalence class of) $\mathfrak{s} \in \mbox{Spin}^c(X,Z)$, one needs a slightly more involved notion of admissibility.  In particular, one needs to arrange that each periodic domain which is a sum of doubly-periodic domains has some local multiplicity which is sufficiently large.  

More precisely, let $\mathfrak{S}$ denote an equivalence class in Spin$^c(X,Z)$, where $\mathfrak{s} \sim \mathfrak{s'}$ if, for each complete splitting of $X_{\eta^0, \ldots \eta^n}$ along imbedded $Y_{\eta^{i_1},\eta^{i_2}}$'s into triangular cobordisms $X_{\eta^{i_1},\eta^{i_2},\eta^{i_3}}$ (see Figure \ref{fig:TriangCob}), we have \[\mathfrak{s}|_{X_{\eta^{i_1},\eta^{i_2},\eta^{i_3}}} = \mathfrak{s}'|_{X_{\eta^{i_1},\eta^{i_2},\eta^{i_3}}}.\]

\begin{figure}
\begin{center}
\input{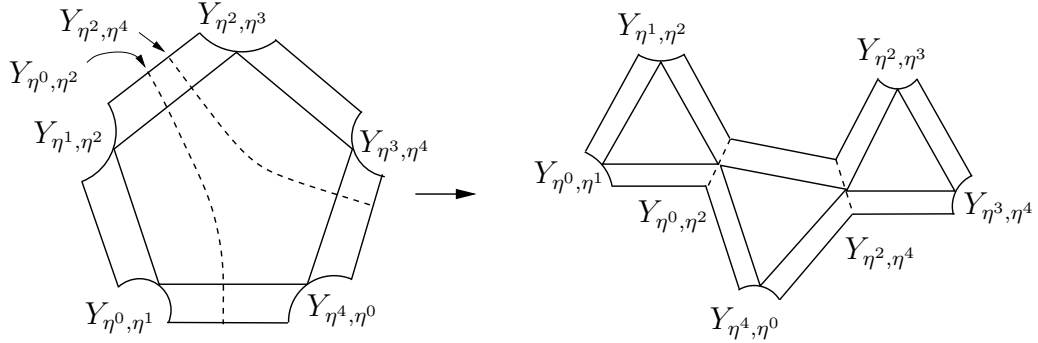}
\end{center}
\caption{Splitting of $X$ along embedded $3$--manifolds into a union of triangular cobordisms.}
\label{fig:TriangCob}
\end{figure}

We then say that a sutured Floer multi-diagram is {\em strongly admissible} for the equivalence class $\mathfrak{S}$ if for each $\mathfrak{s} \in \mathfrak{S}$ and each non-trivial $n$--periodic domain $\mathcal{P}$ which can be written as a sum of doubly-periodic domains:
\[ \mathcal{P} = \sum_{\{\boldeta^{i_1},\boldeta^{i_2}\}
\subset \{\boldeta^0, \ldots, \boldeta^n\}} \mathcal{D}_{\eta^{i_1},\eta^{i_n}}\]
with the property that
\[ \sum \langle c_1(\mathfrak{s}|_{Y_{\eta^{i_1},\eta^{i_2}}}), H(\mathcal{D}_{\eta^{i_1},\eta^{i_2}}) \rangle = 2n \geq 0,\]
it follows that some local multiplicity of $\mathcal{P}$ is strictly greater than $n$.

In the situation of interest in the present work, we will only be interested in strong $\mathfrak{S}$-admissibility for Spin$^c$ equivalence classes whose representatives satisfy
\[ \langle c_1(\mathfrak{s}|_{Y_{\eta^{i_1},\eta^{i_2}}}), H(\mathcal{D}_{\eta^{i_1},\eta^{i_2}}) \rangle = 0\]
for all $Y_{\eta^{i_1},\eta^{i_2}}$.  This is because each of the $3$--manifolds $Y_{\eta^{i_1},\eta^{i_2}}$ for which $SFH(Y_{\eta^{i_1},\eta^{i_1}}) \neq 0$ in the application of the link surgeries spectral sequence used to prove Theorem \ref{thm:SpecSeq} is of the form \[Y \# n(S^1 \times S^2)\] where $Y = F_{g,b} \times I$ for some surface $F_{g,b}$ of genus $g$ with $b$ boundary components (see Section \ref{sec:SFFunctor}).  This in particular implies that its sutured Floer homology is supported in the unique Spin$^c$ structure whose $c_1$ evaluates to $0$ on every doubly-periodic domain.

This observation allows us to use the less restrictive definition of admissibility defined above.
\end{remark}

\begin{lemma}
Every balanced, sutured multi-diagram $(\Sigma, \boldeta^0, \ldots, \boldeta^n)$ is isotopic to one which is admissible in the above sense.
\end{lemma}

\begin{proof}  We proceed by induction on $n$.  We follow the procedure and adopt the notation used in \cite[Prop. 3.15]{MR2253454}, where the base case $n=2$ is proved.  In particular, we assume that we have chosen a set of pairwise disjoint, oriented, properly embedded arcs $\gamma_1, \ldots, \gamma_l$ which are linearly independent in $H_1(\Sigma, \partial\Sigma)$ along with nearby oppositely-oriented parallel curves $\gamma_j'$.

By the induction hypothesis, we may assume that we have constructed an admissible diagram for $(\Sigma, \boldeta^0, \ldots, \boldeta^{n-1})$.  Now introduce the curves $\boldeta^n$ and perform an isotopy of each $\eta^n_j$ in a regular neighborhood of each of $\gamma_j,\gamma_j'$ (in the orientation direction of $\gamma_j,\gamma_j'$) so that there is a point $z_j$ (resp., $z_j'$) on $\gamma_j$ (resp., $\gamma_j'$) which lies after (with respect to the orientation on $\gamma_j$) every other curve in $\boldeta^k$ for $k < n$ and before the given curve of $\boldeta^n$.  Note that if there are several curves of $\boldeta^n$ intersecting a single $\gamma_j$, this isotopy can be accomplished without introducing illegal intersections between them by forcing the finger isotopies to lie in successively smaller regular neighborhoods of $\gamma_i$.  See Figure \ref{fig:Admissible} for an illustration.

\begin{figure}
\begin{center}
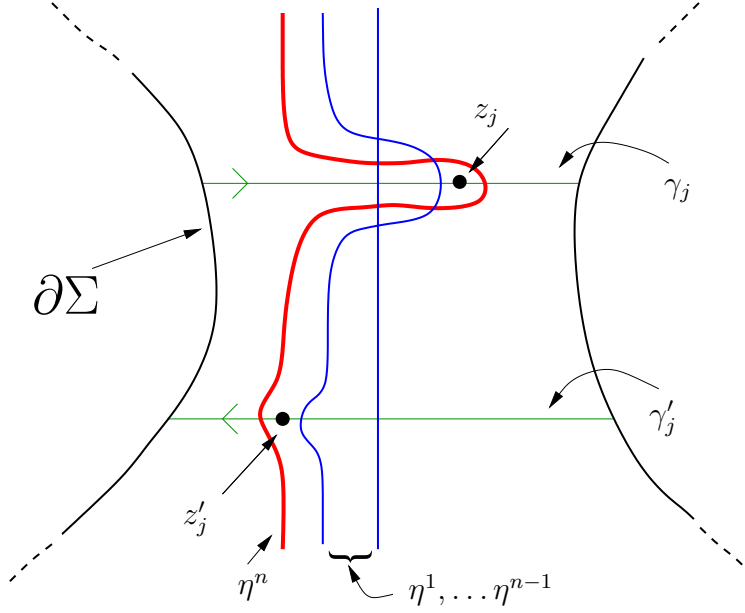
\end{center}
\caption{Winding to achieve an admissible sutured Heegaard multi-diagram.}
\label{fig:Admissible}
\end{figure}

Now let $\cP$ be an $(n+1)$--periodic domain, with 
\[\partial\mathcal{P} = \sum_{i,j} e_{i,j} \eta^i_j.\]
Then if $e_{n,j} \neq 0$ for some $j$, then $n_{z_j}(\mathcal{P}) = e_{i,j}$ and $n_{z_j'}(\mathcal{P}) = -e_{i,j}$, by construction, so $\mathcal{P}$ has both positive and negative coefficients.  If $e_{n,j} = 0$ for all $j$, then $\mathcal{P}$ is an $n$--periodic domain for $(\Sigma, \boldeta^0, \ldots, \boldeta^{n-1})$, which, by induction, has both positive and negative coefficients when non-trivial.
\end{proof}

\begin{definition} 
Borrowing notation from Section 3 of \cite{MR2253454}, let $D(\Sigma, \boldeta^0,\ldots, \boldeta^n)$ denote the set of domains ($2$--chains) in $\Sigma$ with boundary contained in the $\boldeta^i$ curves.  Note that every such domain can be written as a $\Z$--linear combination of the closures of the connected components of \[\Sigma - \left(\bigcup \boldeta^0 \cup \ldots \cup \bigcup \boldeta^n\right),\] which we call {\em elementary domains}.

Let $\mathcal{D} \in D(\Sigma, \boldeta^0, \ldots, \boldeta^n)$.  We say $\mathcal{D}$ is a {\em positive domain}, if it is a $\mathbb{Z}_{\geq 0}$--linear combination of the elementary domains.

For \[({\bf x}_0, \ldots, {\bf x}_n) \in (\mathbb{T}_{\eta^n} \cap \mathbb{T}_{\eta^0}) \times \ldots \times (\mathbb{T}_{\eta^{n-1}} \cap \mathbb{T}_{\eta^n}),\] let $D({\bf x}_0, \ldots, {\bf x}_n)$ denote the set of domains representing homotopy classes in $\pi_2({\bf x}_0, \ldots, {\bf x}_n)$.  
\end{definition}

\begin{proposition} If $(\Sigma, \boldeta^0, \ldots, \boldeta^n)$ is admissible, then for every \[({\bf x}_0, \ldots, {\bf x}_n) \in (\mathbb{T}_{\eta^n} \cap \mathbb{T}_{\eta^0}) \times \ldots \times (\mathbb{T}_{\eta^{n-1}} \cap \mathbb{T}_{\eta^n}),\]
the set $\{\mathcal{D} \in D({\bf x}_0, \ldots, {\bf x}_n): \mathcal{D} \mbox{ is a positive domain.}\}$ is finite.
\end{proposition}

\begin{proof}
The proof follows exactly as in \cite[Lem. 4.13]{MR2113019} (see also \cite[Lem. 3.14]{MR2253454}).
\end{proof}

\begin{corollary}
If $(\Sigma, \boldeta^0, \ldots, \boldeta^n)$ is admissible, then for every \[({\bf x}_0, \ldots, {\bf x}_n) \in (\mathbb{T}_{\eta^n} \cap \mathbb{T}_{\eta^0}) \times \ldots \times (\mathbb{T}_{\eta^{n-1}} \cap \mathbb{T}_{\eta^n}),\]
the set
\[\{\mathcal{M}(\phi)\,|\,\phi \in \pi_2({\bf x}_0, \ldots, {\bf x}_n), \mu(\phi) = 0\}\]
is finite.
\end{corollary}

\begin{proof}
By intersection positivity, non-trivial holomorphic $(n+1)$ gons must be represented by domains with positive coefficients.
\end{proof}

\subsection{Moduli Spaces of Polygons and Associativity}
As is the case in Heegaard Floer homology (see \cite[Sec. 8]{MR2113019} and \cite[Sec. 4]{MR2141852}), counts of holomorphic maps $P_{n+1} \rightarrow Sym^d(\Sigma)$ can be used to define maps between chain complexes associated to sutured Heegaard multi-diagrams.  In particular, let $(\Sigma, \boldeta^0, \ldots, \boldeta^n)$ be a sutured Heegaard multi-diagram, where each $\boldeta^i = \{\eta^i_1, \ldots, \eta^i_d\}$.  Then we define:
\[f_{\eta^0, \ldots, \eta^n}: \bigotimes_{i=1}^n CFH(Y_{\eta^{i-1},\eta^{i}}) \rightarrow CFH(Y_{\eta^{0},\eta^{n}}).\]

As usual, the map involves counting holomorphic maps
\[\phi: P_{n+1} \rightarrow Sym^d(\Sigma)\]

with appropriate boundary conditions.

Specifically, let ${\bf x}_i \in \mathbb{T}_{\eta^{i-1}} \cap \mathbb{T}_{\eta^i}$ for $i \in 1, \ldots n$.  Then
\[f_{\eta^0, \ldots, \eta^n}({\bf x}_1 \otimes \ldots \otimes {\bf x}_n) := \left\{\begin{array}{cc}
\sum_{{\bf y} \in \mathbb{T}_{\eta^0} \cap \mathbb{T}_{\eta^n}} \sum_{\{\phi \in \pi_2({\bf y},{\bf x}_1, \ldots,{\bf x}_n)|\mu(\phi) = 0\}} \mathcal{M}(\phi) \cdot {\bf y} & \mbox{when $n > 1$,}\\
\sum_{{\bf y} \in \mathbb{T}_{\eta^0} \cap \mathbb{T}_{\eta^n}} \sum_{\{\phi \in \pi_2({\bf y},{\bf x}_1, \ldots,{\bf x}_n)|\mu(\phi) = 1\}} \widehat{\mathcal{M}}(\phi) \cdot {\bf y} & \mbox{when $n = 1$,}
  \end{array}\right.\]

Here, $\pi_2({\bf y}, {\bf x}_1, \ldots, {\bf x}_n)$ denotes the set of homotopy classes of Whitney $(n+1)$--gons connecting $({\bf y}, {\bf x}_1, \ldots {\bf x}_n)$ in the sense of \cite[Sec. 8.1.2]{MR2113019}, $\mathcal{M}(\phi)$ (resp., $\widehat{\mathcal{M}}(\phi)$) denotes the moduli space of holomorphic representatives of $\phi$ (resp., the moduli space, quotiented by the natural $\R$ action), and $\mu(\phi)$ denotes the Maslov index of $\phi$ (expected dimension of $\mathcal{M}(\phi)$).  
As in \cite[Sec. 4]{MR2141852}, these maps can be seen to satisfy a generalized associativity property, by examining the ends of $1$--dimensional moduli spaces of holomorphic $(n+1)$--gons:
\[\sum_{0\leq i < j \leq n} f_{\eta^i, \ldots, \eta^j}(f_{\eta^0, \ldots, \eta^i,\eta^j,\ldots, \eta^n}(-) \otimes (-)) = 0.\]

This polygon associativity property can be used to prove, in the sutured Floer homology context, that the map $D$ arising in the iterated mapping cone described in Section \ref{sec:LinkSurg} indeed satisfies $D^2 = 0$.  To prove this, one needs to show that certain terms in the sum above (corresponding to certain ends of $1$--dimensional moduli spaces) yield the zero map in our situation.  
%In particular, we want all terms not corresponding to compositions of maps of the type descrived in Section \ref{sec:} to evaluate to zero.

%For this, we need the analogue of Lemma 4.5 in \cite{MR2141852}:

%\begin{lemma}
%Fix $I, J \in \{0,1,\infty\}^{\ell}$.  Then

%\[\sum_{I = I^0 < \ldots < I^k = J} f_{\eta(I^0),\ldots, \eta(I^k)}(\widehat{\theta}_1 \otimes \ldots \otimes \widehat{\theta}_k) \equiv 0,\]

%where the sum is taken over sequences with the property that $I^{i+1}$ is an immediate successor of $I^i$.
%\end{lemma}

%\begin{proof}
%This lemma is proved exactly as in \cite{MR2141852}.  We include some supplementary details as in Section 2.2 of \cite{O-SLectures} for the benefit of the reader.
%\end{proof}
\subsection{$\Lambda^*H_1$-action} \label{sec:H1Action}
As in \cite[Sec. 4.2.5]{MR2113019}, the sutured Floer chain complex admits an action of $\Lambda^*(H_1(Y,\partial Y;\Z)/\mbox{Tors})$.  A good understanding of this action will allow us to prove that the Khovanov differential matches the $D_1$ differential on the iterated mapping cone in Section \ref{sec:LinkSurg}.  Throughout Subsections \ref{sec:H1Action} and \ref{sec:H1action4man}, all (co)homology groups, where unspecified, are taken with $\Z$ coefficients.  Furthermore, let $H_*^\circ(-)$ denote $H_*(-;\Z)/\mbox{Tors}$ and $Hom(-)$ denote $Hom(-,\Z)$.

\begin{proposition} \label{H1ActionY} Let $(\Sigma,\boldalpha,\boldbeta)$ be a balanced sutured Heegaard diagram representing $(Y,\Gamma)$ with $d = |\boldalpha| = |\boldbeta| > 2$.  There is an action of $H^1(\Omega(\mathbb{T}_\alpha,\mathbb{T}_\beta)) \cong H^\circ_1(Y,\partial Y)$ on $SFH(Y)$ which lowers homological degree by $1$.  Furthermore, this action descends to give a well-defined action of the exterior algebra, $\Lambda^*(H^\circ_1(Y,\partial Y))$.
\end{proposition}

Here, $\Omega(\mathbb{T}_\alpha \cap \mathbb{T}_\beta)$ denotes the space of paths in $Sym^d(\Sigma)$ which begin on $\mathbb{T}_\alpha$ and end on $\mathbb{T}_\beta$.  We stabilize $\Sigma$, if necessary, to achieve $d > 2$ for any $(Y,\Gamma)$.

\begin{proof}
Suppose $\zeta \in Z^1(\Omega(\mathbb{T}_\alpha,\mathbb{T}_\beta))$ is a cocycle in $\Omega(\mathbb{T}_\alpha,\mathbb{T}_\beta)$.  Then for ${\bf x} \in \mathbb{T}_\alpha \cap \mathbb{T}_\beta$, the action is defined by
\begin{equation}
\label{eqn:H1Action}A_\zeta({\bf x}) = \sum_{{\bf y} \in \mathbb{T}_\alpha \cap \mathbb{T}_\beta}\sum_{\{\phi \in \pi_2({\bf x},{\bf y})\,|\,\mu(\phi) = 1\}} \zeta(\phi)\cdot\left(\#\widehat{\mathcal{M}}(\phi)\right) \, {\bf y},
\end{equation}
where we are viewing $\phi$ as a (homotopy class of) $1$-chain in $\Omega(\mathbb{T}_\alpha, \mathbb{T}_\beta)$ and, hence, $\zeta(\phi)$ is well-defined.
The proofs that 

\begin{enumerate}
\item $A_\zeta$ is a chain map, hence induces a well-defined map on homology,
\item the induced map on homology associated to $A_\zeta$ depends only upon the cohomology class of $\zeta$, hence provides a well-defined action of $H^1(\Omega(\mathbb{T}_\alpha,\mathbb{T}_\beta))$,
\item $A_\zeta \circ A_\zeta$ is the zero map on homology, hence we have a well-defined action of $\Lambda^*(H^1(\Omega(\mathbb{T}_\alpha,\mathbb{T}_\beta)))$ on $SFH(Y)$, 
\end{enumerate}

follow without change as in the proofs of Lemma 4.18 and 4.19 and Proposition 4.17 of \cite{MR2113019} by examining ends of $1$-dimensional moduli spaces.

To understand why $H^\circ_1(Y,\partial Y) \cong H^1(\Omega(\mathbb{T}_\alpha, \mathbb{T}_\beta))$ when $d>2$, we use an adaptation of the argument used in the proof of \cite[Prop. 2.15]{MR2113019}.

Namely, we arrive at a homotopy long exact sequence
\[\xymatrix{0 \cong \pi_2(Sym^d(\Sigma)) \ar[r] & \pi_1(\Omega(\mathbb{T}_\alpha,\mathbb{T}_\beta)) \ar[r] & \pi_1(\mathbb{T}_\alpha \times \mathbb{T}_\beta) \ar[r]^i & \pi_1(Sym^d(\Sigma)).}\]
Recall that we proved $\pi_2(Sym^d(\Sigma)) \cong 0$ in the proof of Proposition \ref{prop:AffineIdent}.

Under the identification $\pi_1(Sym^d(\Sigma)) \cong H_1(\Sigma)$ (see Lemma 2.6 and Definition 2.11 of \cite{MR2113019}), $i(\pi_1(\mathbb{T}_\alpha \times \mathbb{T}_\beta))$ corresponds to \[\mbox{Span}([\boldalpha],[\boldbeta]) \subset H_1(\Sigma).\]The above long exact sequence therefore yields the short exact sequence:
\[\xymatrix{0 \ar[r] & \pi_1(\Omega(\mathbb{T}_\alpha,\mathbb{T}_\beta)) \ar[r] & \mbox{Ker}\left[\mbox{Span}([\boldalpha],[\boldbeta]) \rightarrow H_1(\Sigma)\right] \ar[r] & 0}\]
But Proposition \ref{prop:H1andH2} tells us that 
\[\mbox{Ker}\left[\xymatrix{\mbox{Span}([\boldalpha],[\boldbeta]) \ar[r]^{i} & H_1(\Sigma)}\right] \cong H_2(Y).\]
Thus,
\[\pi_1(\Omega(\mathbb{T}_\alpha,\mathbb{T}_\beta)) \cong H_2(Y) \cong H^1(Y,\partial Y).\]
Applying the $Hom(-,\Z)$ functor, we arrive at the desired conclusion:
\[H^1(\Omega(\mathbb{T}_\alpha,\mathbb{T}_\beta)) \cong \mbox{Hom}(H^1(Y,\partial Y)) \cong H^\circ_1(Y,\partial Y).\]
\end{proof}

\begin{remark} Since we are considering sutured Floer homology with $\Ztwo$ coefficients, we will be most interested in the corresponding $\Wedge^*(H^\circ_1(Y,\partial Y)\otimes_\Z \Ztwo)$ action.
\end{remark}

\subsection{Naturality of Triangle Maps}\label{sec:H1action4man}
As before, singular (co)homology groups, where unspecified, will be taken with $\Z$ coefficients, and we will use $H_*^\circ(-)$ to denote $H_*(-;\Z)/\mbox{Tors}$ and $Hom(-)$ denote $Hom(-,\Z)$.

Let $X = X_{\eta^0,\eta^1,\eta^2}$ be the $4$--manifold associated to a sutured Heegaard triple-diagram $(\Sigma,\boldeta^0,\boldeta^1,\boldeta^2)$, with

\begin{enumerate}
\item $Y = \partial(X)$,
\item $Y' = Y_{\eta^0,\eta^1} \cup Y_{\eta^1,\eta^2} \cup - Y_{\eta^0,\eta^2} \subset Y$, and
\item $Z = \overline{Y - Y'}$.
\end{enumerate}

Then (see \cite[Lem. 2.6]{MR2222356}) the map 
\[f_{\eta^0,\eta^1,\eta^2}: SFH(Y_{\eta^0,\eta^1}) \otimes SFH(Y_{\eta^1,\eta^2}) \rightarrow SFH(Y_{\eta^0,\eta^2})\]
induced by counting triangles admits an action of $H^\circ_1(X,Z)$ as follows.  Let $h \in H^\circ_1(X,Z)$, and abbreviate the map $f_{\eta^0,\eta^1,\eta^2}$ by $f$.  Then we claim (and prove during the proof of Proposition \ref{prop:H1action4man}) that $\exists$ \[(h_{01},h_{12},h_{02}) \in \bigoplus_{i \in \Z_3}H^\circ_1(Y_{\eta^i,\eta^{i+1}},\partial Y_{\eta^i,\eta^{i+1}}) \cong H^\circ_1(Y,Z)\] satisfying $i(h_{01},h_{12},h_{02}) = h$ under the inclusion \[i:H^\circ_1(Y,Z) \rightarrow H^\circ_1(X,Z),\] which allows us to define:
\begin{equation}\label{equation:H1action4man}
(h\circ f)(\xi_{01} \otimes \xi_{12}):= f((h_{01} \circ \xi_{01}) \otimes \xi_{12}) + f(\xi_{01} \otimes (h_{12} \circ \xi_{12})) + h_{02} \circ f(\xi_{01} \otimes \xi_{12}).
\end{equation}

\begin{proposition} \label{prop:H1action4man}
The $H^\circ_1(X,Z)$-action defined in Equation \ref{equation:H1action4man} yields a well-defined map
\[f:\Lambda^*(H^\circ_1(X,Z)) \otimes SFH(Y_{\eta^0,\eta^1}) \otimes SFH(Y_{\eta^1,\eta^2}) \rightarrow SFH(Y_{\eta^0,\eta^2}).\]
\end{proposition}

\begin{proof}  We begin by verifying that the map \[i:H^\circ_1(Y,Z) \rightarrow H^\circ_1(X,Z)\] is surjective, as claimed above.  For this, we use:

\begin{enumerate}
\item $H^\circ_1(Y,Z) \cong Hom(H^1(Y,Z))$
\item $H^\circ_1(X,Z) \cong Hom(H^1(X,Z))$
\end{enumerate}

along with a dualized version of the cohomology long exact sequence on the triple $(X,Y,Z)$:
\[\xymatrix{Hom(H^2(X,Y)) \ar[r]^\partial & Hom(H^1(Y,Z)) \ar[r]^i & Hom(H^1(X,Z)) \ar[r] & Hom(H^1(X,Y))}.\]
$H^1(X,Y) \cong H_3(X) = 0$, since $X$ is homotopy-equivalent to a $2$-complex; hence, $i$ is surjective, as desired.

To see that the action is well-defined, we need to check that it is trivial on $\ker(i) = \mbox{im}(\partial)$.  For this, notice first that
\begin{eqnarray*}
H^2(X,Y) &\cong& H_2(X)\\
H^1(Y,Z) &\cong& H^1(Y',\partial Y') \cong H_2(Y').
\end{eqnarray*}
We must therefore show that elements in $Hom(H_2(Y'))$ coming from $Hom(H_2(X))$ act trivially.  This follows exactly as in the proof of \cite[Lem. 2.6]{MR2222356}, once we note that

\begin{eqnarray*}
Hom(H_2(Y')) &\cong& \bigoplus_{i \in \Z_3} \left[H^1(\mathbb{T}_{\eta^i}) \oplus H^1(\mathbb{T}_{\eta^{i+1}})/H^1(Sym^d(\Sigma))\right]\\
Hom(H_2(X)) &\cong& H^1(\mathbb{T}_{\eta^0}) \oplus H^1(\mathbb{T}_{\eta^1}) \oplus H^1(\mathbb{T}_{\eta^2})/H^1(Sym^d(\Sigma)).
\end{eqnarray*}

The fact that the action of $H_1^\circ(X,Z)$ extends to a well-defined action of the whole exterior algebra comes from the corresponding fact for the SFH groups at the three ends. 
\end{proof}

%The following corollary will be useful to us in proving that the maps in the reduced Khovanov chain complex match the maps induced by triangles in the sutured Floer multi-complex.

%\begin{corollary} \label{cor:H1Action4Man}
%Let \[F: SFH(Y_{\eta^0,\eta^1}) \rightarrow SFH(Y_{\eta^0,\eta^2})\] be the map defined by
%\[F(\xi) := f_{\eta^0,\eta^1,\eta^2}(\xi \otimes \theta),\] where $\theta \in SFH(Y_{\eta^1,\eta^2})$ is the canonical top-degree generator.
%Then \[F(h_{01}\circ\xi) = h_{02}\circ(F(\xi))\] if $(h_{01},h_{02}) \in H^\circ_1(Y_{\eta^0,\eta^1},\partial Y_{\eta^0,\eta^1} ) \oplus H^\circ_1(Y_{\eta^0,\eta^2},\partial  Y_{\eta^0,\eta^1})$ satisfies $i(h_{01},h_{02}) = 0 \in H^\circ_1(X,Z)$.
%\end{corollary}

%\begin{proof}
%Apply Proposition \ref{prop:H1action4man} in the case $h=0$, using \[(h_{01},0,h_{02}) \in \bigoplus_{i \in \Z_3} H^\circ_1(Y_{\eta^i,\eta^{i+1}},\partial Y_{\eta^i,\eta^{i+1}})\] satisfying $i(h_{01},0,h_{02}) = h$.
%\end{proof}

\section{Link Surgeries Spectral Sequence} \label{sec:LinkSurg}
The spectral sequence of Theorem \ref{thm:SpecSeq} is a special case of a more general phenomenon, which we now describe.

Let $(Y,\Gamma)$ be a sutured manifold and let $L = K_1 \cup \ldots \cup K_\ell \subset Y$ be an $\ell$ component oriented, framed link in the interior of $Y$.  Recall that a {\em framing} for an $\ell$--component link, $L$, is a choice of non-zero section, $\vec{\lambda} = (\lambda_1, \ldots, \lambda_\ell)$, of its normal disk bundle, where the pushoff, $\lambda_i$, and the zero section, $K_i$, are oriented compatibly for each $i$.  I.e., if $\mu_i$ is the oriented meridian of $K_i$, then $lk(\lambda_i,\mu_i) = +1$.     

Just as in \cite[Sec. 4]{MR2141852}, we can associate to $L$ a link surgeries spectral sequence coming from an iterated mapping cone construction.  We focus on stating the necessary results, adding details of the proofs only where they differ in some crucial way from the analogous proofs in \cite{MR2141852}.

Let $\mathcal{I} = (m_1, \ldots, m_\ell)$ be a multi-framing on $L$, where $m_i \in \{0,1,\infty\}$ and let $Y(\mathcal{I})$ denote the $3$--manifold obtained from $Y$ by doing surgery on $Y$ along $L$ corresponding to $\mathcal{I}$.  As usual, ``$\infty$'' denotes the meridional ($\mu$) slope, ``$0$'' denotes the specified longitudinal framing ($\lambda$) slope, and ``$1$'' denotes the meridian + longitude $(\gamma)$ slope.  Giving the set $\{0,1,\infty\}$ the dictionary ordering, we call $\mathcal{I'} \in \{0,1,\infty\}^\ell = (m_1, \ldots, m_\ell)$ an {\em immediate successor} of $\mathcal{I}$ if there exists some $j$ such that $m_i = m_i'$ if $i \neq j$ and $(m_j, m_j')$ is either $(0,1)$ or $(1,\infty)$.  

One can construct a sutured Heegaard multi-diagram for the collection of $Y(\mathcal{I})$ as described in \cite[Sec. 4]{MR2141852}, by constructing a {\em bouquet} for the framed link.  A bouquet in the sutured setting is a choice of arcs $a_1, \ldots a_\ell$ from a point on each of $K_1, \ldots, K_\ell$ to a chosen boundary component for each connected component of $Y$.  Let $L'$ denote a neighborhood of $L \cup a_1 \cup \ldots \cup a_\ell$.  

Given such a bouquet, we may construct a sutured Heegaard diagram for each of the $Y(\mathcal{I})$ by first constructing a Morse function for $Y-L'$, as in \cite[Prop. 2.13]{MR2253454}, with $d$ index $1$ critical points and $d-\ell$ index $2$ critical points.  Each associated (non-balanced) sutured Heegaard multi-diagram can then be completed to a balanced sutured Heegaard diagram for each $Y(\mathcal{I})$ by extending the Morse function for $Y-L'$ to one for $Y(\mathcal{I})$ that has $d$ additional index $2$ critical points.

More precisely, let $(\Sigma,\boldalpha,\boldbeta)$ be a (non-balanced) sutured Heegaard diagram for $Y-L'$ with $\boldalpha = \{\alpha_1, \ldots, \alpha_d\}$ and $\boldbeta = \{\beta_{\ell+1}, \ldots, \beta_d\}$.  Now consider the following three $\ell$--tuples of curves:

\begin{enumerate}
\item $\mu_1, \ldots, \mu_{\ell} \subset \Sigma$, representing the images in $Y-L'$ of meridians ($\infty$ slopes) of $K_1, \ldots K_\ell \subset Y$,
\item $\lambda_1, \ldots, \lambda_\ell \subset \Sigma$, representing the images of longitudes ($0$ slopes) of $K_1, \ldots, K_\ell \subset Y$,
\item $\gamma_1, \ldots, \gamma_\ell \subset \Sigma$, representing the images of longitudes + meridians ($1$ slopes) of $K_1, \ldots, K_\ell \subset Y$.
\end{enumerate} 

Then, given a Heegaard diagram $(\Sigma,\boldalpha, \boldbeta)$ for $Y-L'$ and a particular $\mathcal{I} \in \{0,1,\infty\}^\ell$, we form a balanced sutured Heegaard diagram $(\Sigma_{\mathcal{I}}, \boldalpha_{\mathcal{I}},\boldeta_{\mathcal{I}})$ for $Y(\mathcal{I})$, where $\Sigma_{\mathcal{I}} = \Sigma$, $\boldalpha_{\mathcal{I}} = \boldalpha$, and $\boldeta(\mathcal{I}) = \{\eta_1, \ldots, \eta_d\}$ is given by \[ \eta_i = \left\{\begin{array}{cl}
  \beta_i & \mbox{if $i>\ell$}\\
  \mu_i & \mbox{if $m_i = \infty$}\\
  \lambda_i & \mbox{if $m_i = 0$}\\
  \gamma_i & \mbox{if $m_i = 1$}\\
  \end{array}\right.\]

Note that, with the above choice, we can construct a sutured Heegaard multi-diagram for any ordered set $\{\mathcal{I}^0, \ldots \mathcal{I}^k\}$ with each $\mathcal{I}^i \in \{0,1,\infty\}^\ell$.  Furthermore, given a sequence $\mathcal{I}^0 < \ldots < \mathcal{I}^k$ of multi-framings, note that \[Y_{\eta(\mathcal{I}^i),\eta(\mathcal{I}^{i+1})} = (F \times I)\,\, \# \,\,j( S^1 \times S^2)\] for some surface with boundary, $F$.  Here, $j$ will be the number of $\eta$ curves upon which $\mathcal{I}^i$ and $\mathcal{I}^{i+1}$ agree.  In particular, by Propositions 9.4 and 9.15 of \cite{MR2253454}, we know that $SFH(Y_{\eta(\mathcal{I}^i),\eta(\mathcal{I}^{i+1})}) = V^{\otimes j}$, where $V = \mathbb{Z}_2 \oplus \mathbb{Z}_2$.  In addition, there is a canonical top-degree generator of $\widehat{HF}(\#^j S^1 \times S^2)$, hence of $Y_{\eta(\mathcal{I}^i),\eta(\mathcal{I}^{i+1})}$, denoted $\theta$.  We get an induced map \[D_{\mathcal{I}^0 < \ldots < \mathcal{I}^k}: CFH(Y(\mathcal{I}^0)) \rightarrow CFH(Y(\mathcal{I}^k))\] defined by 

\begin{equation} \label{eqn:Diff}
D_{\mathcal{I}^0 < \ldots < \mathcal{I}^k}(\xi) = f_{\alpha, \eta(I^0), \ldots, \eta(I^k)}(\xi \otimes \theta_1 \otimes \ldots \otimes \theta_k),
\end{equation}

where $\theta_i$ represents the top-degree generator of $SFH(Y_{\eta(I^i),\eta(I^{i+1})})$ and $f_*$ is the map defined by counting holomorphic $k+2$--gons, defined in the last section.
We can then define \[X = \bigoplus_{\mathcal{I} \in \{0,1,\infty\}^\ell} CFH(Y(\mathcal{I})),\] endowed with a map \[D: X \rightarrow X,\] defined by \[D\xi = \sum_{\mathcal{J}}\sum_{\{\mathcal{I} = \mathcal{I}^0<\ldots<\mathcal{I}^j = \mathcal{J}\}} D_{I^0<\ldots<I^j}(\xi),\] where the index set of the inner sum is taken over the set of all increasing sequences $\mathcal{I}$ to $\mathcal{J}$ having the property that $\mathcal{I}^{i+1}$ is an immediate successor of $\mathcal{I}^i$.  Let $X^{(0,1)}$ denote the subset of $X$ corresponding to $\mathcal{I} = \{0,1\}^\ell \subset \{0,1,\infty\}^\ell$ and $D^{(0,1)}$ denote the restriction of $D$ to this subset.

\begin{proposition} \label{prop:LinkSurg} $X^{(0,1)}$ is a filtered chain complex, with differential $D^{(0,1)} = D_0 + D_1 \ldots D_\ell$, where \[D_k(\xi) = \sum_{\mathcal{J}} \sum_{\{\mathcal{I} = \mathcal{I}^0 < \ldots < \mathcal{I}^j = \mathcal{J}|j=k\}}D_{I^0<\ldots<I^j}(\xi).\]  Associated to this filtered chain complex is a spectral sequence whose $E^1$ term is \[\bigoplus_{\mathcal{I}\in \{0,1\}^\ell} SFH(Y(\mathcal{I}))\] and whose $E^\infty$ term is $SFH(Y(\mathcal{I}_\infty))$, where $\mathcal{I}_\infty := (\infty, \ldots \infty)$.
\end{proposition}

\begin{proof}
The proof of the proposition follows exactly as in the proof of \cite[Thm. 4.1]{MR2141852}.  In particular, polygon associativity combined with an analogue of \cite[Lem. 4.5]{MR2141852} implies that $(X,D)$ is a chain complex ($D^2 = 0$).  An argument analogous to the one in the proof of \cite[Thm. 4.7]{MR2141852} (which relies on the algebraic \cite[Lem. 4.4]{MR2141852}) implies that if $\mathcal{I}_0 = (m^0_1, \ldots m^0_\ell)$, $\mathcal{I}_1 = (m^1_1,\ldots, m^1_\ell)$, $\mathcal{I}_\infty = (m^\infty_1, \ldots, m^\infty_\ell)$ are three multi-framings for which $\exists$ $j$ such that $m^0_i = m^1_i = m^\infty_i$ for all $i\neq j$ and $m^*_j = *$ for $* \in \{0,1,\infty\}$, then $CFH(Y(\mathcal{I}_\infty))$ is quasi-isomorphic to the mapping cone of the map \[f: CFH(Y(\mathcal{I}_0)) \rightarrow CFH(Y(\mathcal{I}_1)).\] Here, $f$ is the map induced by the restricted differential, $D$.  The proposition then follows by induction on the number of link components, just as in the proof of \cite[Thm. 4.1]{MR2141852}.

To see that the $E^1$ term is as stated, note that the $D_0$ term in the differential, for each direct summand of the iterated mapping cone, is just the internal differential (defined by counting holomorphic disks) for that summand.  

The $E^\infty$ term, quasi-isomorphic to $SFH(Y_\infty)$, will be the homology of the complex $X^{(0,1)}$ with differential $D^{(0,1)}$ given by all maps counting polygons, i.e., $D^{(0,1)} = D_0 + \ldots +D_{\ell}$. 
\end{proof}

\section{Spectral Sequence from Khovanov to Sutured Floer} \label{sec:SpecSeq}
In this section, we will use the link surgeries spectral sequence and the equivalence of certain Khovanov and sutured Floer homology functors on a restricted class of tangles to prove Theorem \ref{thm:SpecSeq}.

\subsection{Admissible, Balanced, Resolved Tangles} \label{sec:Resolved}
Recall that $I := [-1,4]$ and $D \times I$ denotes the product sutured manifold $F_{0,1} \times I$ as in Example \ref{example:prodsurface}.   Whenever we write $D \times I$, we shall always assume we have fixed an identification with a standard subset of $\R^3$: \[D \times I := \{(x,y,z)\in \R^3\,\,|\,\,x^2 + y^2 \leq 1, z \in [-1,4]\}.\]

Let $D_+$ (resp., $D_-$) denote $D \times \{4\}$ (resp., $D \times \{-1\}$) and $\mbox{Int}(D_{\pm})$ denote the interior of $D_\pm$.  More generally, let $D_{a}$ denote $D \times \{a\}$ for each $a \in [-1,4]$.

\begin{definition} \label{defn:tangle} An {\em admissible tangle} in $D \times I$ is an equivalence class of smoothly imbedded, unoriented $1$--manifolds $T$ satisfying $\partial T \subset (\mbox{Int}(D_+) \cup \mbox{Int}(D_-))$, where $T_1 \sim T_2$ if there is an ambient isotopy connecting $T_1$ to $T_2$ which acts trivially on $(\partial D) \times I$.
\end{definition}

\begin{definition} \label{defn:balancedtangle} An admissible tangle is said to be {\em balanced} if each equivalence class representative, $T$, satisfies $\#(T \cap D_+) = \#(T \cap D_-)$.
\end{definition}

\begin{definition} \label{defn:projection}
Let \[\pi_y: D \times I \rightarrow \left(A := \{(x,y,z) \in \R^3\,\,|\,\, x \in [-1,1], y=0, z \in [-1,4]\}\right),\] given by $\pi_y(x,y,z) = (x,0,z)$, be the projection to the $xz$ plane.  For any tangle representative $T$ for which $\pi_y(T) \subset A$ is a smooth imbedding away from finitely many transverse double points, we denote by $\mathcal{P}(T)$ the enhancement of $\pi_y(T)$ which records over/undercrossing information.  We call $\mathcal{P}(T)$ the {\em projection} of $T$.
\end{definition}

Note that, by transversality, a generic representative, $T$, of a tangle equivalence class has a well-defined projection (i.e., satisfies the condition above).

\begin{definition} \label{defn:resolved}
A tangle representative, $T$, is said to be {\em resolved} if $\pi_y(T) \subset A$ is a smooth imbedding.
\end{definition}

\begin{definition}
We call an admissible, balanced, resolved tangle representative $T \subset D \times I$ an {\em ABR}.
\end{definition}

%We may assume without loss of generality (by replacing $T$ by another representative of $[T]$ if necessary) that $T_i$ and $D_{\frac{3}{2}}$ intersect transversely and $T_i \cap D_{\frac{3}{2}} \neq \emptyset$ for all $i \in \{1, \ldots, k+a\}$.

\begin{definition}
A {\em saddle cobordism} $S\subset A\times [0,1]$ is a smooth cobordism between two ABR projections $\cP(T')$ and $\cP(T'')$ with the property that $\exists$ a unique $c \in [0,1]$ such that 

\begin{enumerate}
\item $S \cap (A \times \{c\})$ is a smooth $1$--dimensional imbedding away from a single double-point.\\
\item $S \cap (A \times \{s\})$ is a smooth $1$--dimensional imbedding whenever $s \neq c$.
\end{enumerate}

Let $|T'|$ (resp., $|T''|$) denote the number of connected components of $T'$ (resp., $T''$).  There are three cases:

\begin{enumerate}
  \item When $|T'| = |T''| + 1$, we call $S$ a {\em merge saddle cobordism}, 
  \item when $|T'| = |T''|-1$, we call $S$ a {\em split saddle cobordism}, and 
  \item when $|T'| = |T''|$, we call $S$ a {\em zero saddle cobordism}.  Note that in this case, exactly one of $T'$ or $T''$ {\em backtracks} (see Definition \ref{defn:backtrack}).
\end{enumerate}
\end{definition}

We can associate to each ABR a module and to each saddle cobordism between ABR projections a map between modules in two ways: via a Khovanov-type procedure (which we call a {\em Khovanov functor}) and via a sutured Floer homology procedure (which we call a {\em sutured Floer functor}).  We will show that these two procedures yield the same modules and maps for ABR's, a key step in the proof of Theorem \ref{thm:SpecSeq}.

We use the language of functors informally, only to organize our arguments. 
\subsection{Khovanov Functor}\label{sec:KhFunctor}
Let $T$ be an ABR with connected components $T_1, \ldots, T_{k+a}$, where $T_1, \ldots, T_k$ satisfy $T_i \cap \partial(D \times I) = \emptyset$ and $T_{k+1}, \ldots, T_{k+a}$ satisfy $T_i \cap \partial(D \times I) \neq \emptyset$.

\begin{definition} \label{defn:KhFunctor} For $T$ an ABR, let $Z(T)$ denote the $\Ztwo\,$--vector space formally generated by the circle components $[T_1],\ldots,[T_k]$. For convenience, we identify $Z(T)$ with the quotient
\[
Z(T) = Span_{\Ztwo}
([T_1],\ldots,[T_{k+a}])/[T_{k+1}]\sim\ldots\sim [T_{k+a}]\sim 0\]
\end{definition}

\begin{definition} \label{defn:backtrack}
We say that $T = T_1 \amalg \ldots \amalg T_{k+a}$ {\em backtracks} if there exists a $j \in \{k+1, \ldots, k+a\}$ such that 
$\partial T_j\subset D_-$
or $\partial T_j\subset D_+$.
\end{definition}

Now let $V(T)$ be the $\mathbb{Z}_2$-vector space
$$
V(T):=
\begin{cases}
0 &\mbox{if $T$ backtracks,}\\
\Wedge^* Z(T) &\mbox{otherwise.}
\end{cases}
$$
Here, $\Wedge^* Z(T)$ denotes the
exterior algebra of $Z(T)$, i.e.
the polynomial algebra over $\mathbb{Z}_2$
in formal variables $[T_1],\ldots,[T_{k+a}]$,
modulo the relations $[T_i]^2=0$ for $i\leq k$
and $[T_i]=0$ for $i>k$.

\begin{remark} \label{remark:KhovNot}
Our notation is related to that of Khovanov (see \cite[Sec. 2]{MR1928174} and \cite[Sec. 5]{MR2124557})
as follows.
Given $T$, an ABR, Khovanov defines
a left $H^n$-module $\mathcal{F}(T)$, for
a graded $\mathbb{Z}$-algebra $H^n$,
which depends on \[n = \#(T \cap D_+) = \#(T \cap D_-).\]

Note that Khovanov actually takes $T \subset D \times [0,1]$ to be a tangle with $2n$ upper endpoints $E = \{e_1, \ldots, e_{2n}\} = T \cap D_+$ and no lower endpoints.  To match his notation, we choose a smoothly imbedded circle, $C \subset D_+$, separating $\partial(D \times I)$ into two connected components, $R_+$ and $R_-$, satisfying \[R_+ \cap E = \{e_1, \ldots, e_n\} \,\, \mbox{  and  } \,\, R_- \cap E = \{e_{n+1}, \ldots, e_{2n}\}.\]  Orienting $C$ compatibly with $R_+ \subset D$, and thinking of $C$ as the suture of $D \times I$, we can now reparameterize in the obvious way to identify Khovanov's notation with ours.
\end{remark}
%\begin{enumerate}
%\item a neighborhood $N(C)$ with $C \times I$,
%\item $R_+$ with $D_+$,
%\item $R_- \cup \partial D \times [0,1]
%\end{enumerate}

There is then an isomorphism
$$
V(T)\cong \mathbb{Z}_2\otimes_{\mathbb{Z}}
(_e\mathbb{Z}\otimes_{H^n}\mathcal{F}(T))
$$
where $_e\mathbb{Z}$ denotes the
right $H^n$-module associated to the Jones-Wenzl projector described in \cite[Sec. 5]{MR2124557}.

Now consider a merge saddle cobordism $S_m \subset A\times [0,1]$
between two ABR projections $\cP(T')$ and $\cP(T'')$, where the saddle merges two components of $T'$ labeled $T_i'$ and $T_j'$.
Then there is a natural identification
$$
Z(T')/[T'_i]\sim [T'_j]=Z(T'')
$$
and a corresponding isomorphism
$$
\alpha : V(T')/[T'_i]\sim [T'_j]\stackrel{\cong}{\longrightarrow} V(T'').
$$
Associated to $S_m$ is a map $V(T') \rightarrow V(T'')$, usually referred to as the {\em multiplication} map:

\begin{definition}\label{defn:saddlemult}
The {\em multiplication map}, \[\mathcal{V}_m : V(T')\rightarrow V(T''),\]
is the composite
\[
\begin{CD} V(T') @>{\pi}>> 
\frac{V(T')}{[T'_i]\sim [T'_j]
}@>{\alpha}>> V(T'').
\end{CD}
\]
where $\pi$ denotes the quotient map.
\end{definition}

If $S_m$ is a merge cobordism, then running it backwards produces $S_\Delta$, a split cobordism from $\cP(T'')$ to $\cP(T')$.  Using the notation from above, we define the {\em comultiplication} map:

\begin{definition}
The {\em comultiplication map}, \[\mathcal{V}_\Delta : V(T'')\rightarrow V(T'),\]
is defined to be the composite
$$
\begin{CD}V(T'') @>{\alpha^{-1}}>> 
\frac{V(T')}{[T'_i]\sim [T'_j]}@>{\varphi}>> V(T').
\end{CD}
$$
where the map $\varphi$ is defined by
$\varphi(a):=([T'_i]+[T'_j])\wedge \widetilde{a}$, where $\widetilde{a}$ is any lift of $a$ in $\pi^{-1}(a)$.
\end{definition}

The preceding definitions are used to define a chain complex associated to an admissible, balanced tangle projection $\cP(T) \subset A$ as follows.   

Label the crossings of $\cP(T)$ by $1, \ldots, \ell$.  For any $\ell$--tuple $\cI = (m_1, \ldots, m_{\ell}) \in \{0,1,\infty\}^\ell$, we denote by $\cP_{\cI}(T)$ the tangle projection obtained from $\cP(T)$ by

\begin{itemize}
  \item leaving a neighborhood of the $i$th crossing unchanged, if $m_i = \infty$,
  \item replacing a neighborhood of the $i$th crossing with a ``0'' resolution, if $m_i = 0$, and
  \item replacing a neighborhood of the $i$th crossing with a ``1'' resolution, if $m_i = 1$.
\end{itemize}

See Figure \ref{fig:Resolutions} for an illustration of the ``0'' and ``1'' resolutions.

We have chosen the above notation to coincide with that of Section \ref{sec:LinkSurg}.  Incorporating the language of that section, we define:

\begin{definition}
Given $\cP(T) \subset A$, a projection of an admissible, balanced tangle, $T$, we have the chain complex \[CV(\cP(T)) = \left(\bigoplus_{\cI \in \{0,1\}^\ell} V(P_{\cI}(T)),D\right),\] where $D = \sum_{\cI,\cI'} D_{\cI,\cI'}$, with the sum taken over all pairs $\cI,\cI' \in \{0,1\}^\ell$ such that $\cI'$ is an immediate successor of $\cI$, and \[D_{\cI, \cI'}: V(\cP_{\cI}(T)) \rightarrow V(\cP_{\cI'}(T))\] is
\begin{itemize}
  \item the map, $0$, if $\cP_{\cI}(T)$ and $\cP_{\cI'}(T)$ are related by a zero saddle cobordism (and, hence, exactly one of $V(\cP_{\cI}(T))$, $V(\cP_{\cI'}(T))$ is $0$),
  \item the map, $\mathcal{V}_m$ (resp., $\mathcal{V}_\Delta$) associated to the merge (resp., split) saddle cobordism, $\cP_{\cI}(T)$ to $\cP_{\cI}(T)$, otherwise.
\end{itemize}
Let $V(T)$ denote the homology, $H_*(CV(\cP(T)))$, of $CV(\cP(T))$.
\end{definition}

Khovanov proves, in \cite{MR2124557}, that $D^2 = 0$, and the homology, $V(T)$, of the resulting chain complex is an invariant of the tangle equivalence class (i.e., independent of the choice of projection, $\cP(T)$).
Furthermore, if $T$ is equipped with an orientation, Khovanov endows the complex $CV(\cP(T))$ with
a pair of gradings, called
the {\em cohomological} and the {\em quantum} grading.

\begin{definition} (Cohomological grading)
Let $a$ be an element of $CV(\cP(T))$, and suppose
that $a$
is contained in $V(\cP_{\cI}(T))\subset CV(\cP(T))$,
for an $\ell$-tuple
$\cI=(m_1,\ldots,m_\ell)\in\{0,1\}^\ell$. Then
$$
i(a):=-n_++\sum_i m_i
$$
where $n_+$ denotes the number of positive crossings
in $\cP(T)$.
\end{definition}

\begin{definition} (Quantum grading)
Let $a$ be an element
of $CV(\cP(T))$,
and suppose that $a$ is contained in
$\Wedge^d (Z(\cP_{\cI}(T)))\subset V(\cP_{\cI}(T))
\subset CV(\cP{T})$, for an
$\ell$-tuple $\cI=(m_1,\ldots,m_\ell)\in\{0,1\}^\ell$
and a non-negative integer $d\geq 0$.
Then
$$
j(a) := \dim_{\Ztwo} (Z(\cP_{\cI}(T))) - 2d
+ n_- -2n_+ + \sum_i m_i
$$
where $n_+$ (resp. $n_-$) denotes the
number of positive (resp. negative) crossings
in $\cP(T)$.
\end{definition}

Corresponding to the two gradings, there is a
a decomposition of $CV(\cP(T))$ into subspaces
$$
CV(\cP(T))=\bigoplus_{i,j\in\mathbb{Z}}CV(\cP(T))^{i,j}
$$
where $CV(\cP(T))^{i,j}\subset CV(\cP(T))$
denotes the subspace consisting of all elements
which have cohomological degree $i$ and quantum
degree $j$. The differential in $CV(\cP(T))$ is bigraded (and in fact
carries $CV(\cP(T))^{i,j}$ to $CV(\cP(T))^{i+1,j}$),
and hence there is an induced bigrading on
homology:
$$
V(T)=\bigoplus_{i,j}V(T)^{i,j}.
$$

%(********This is a good place to mention the Khovanov $i$ and $j$ gradings********)
%For each vertex $\cI=(\cI_1,\ldots,\cI_c)$ of the cube $\{0,1\}^c$,

%If $\cI'$ is an immediate successor of $\cI$, then the resolutions
%$\cP(\cI)$ and $\cP(\cI')$ differ at a single crossing, and
%we have a map
%%$$
%D^V_{\cI<\cI'} : V(\cP_{\cI}(T))\longrightarrow V(\cP_{\cI'}(T))
%$$
%which is either a multiplication, a comultiplication (or
%the zero map, in the case where one of the two spaces
%$V(\cP_{\cI}(T))$, $V(\cP_{I'}(T))$ is equal to zero).

%Let $CV(\cP(T))$ be the chain complex whose $i$th chain group
%is equal to
%$$
%CV(\cP(T),i) := \bigoplus_{\{\cI\in\{0,1\}^c
%\big|\sum_{i=1}^c \cI_i-n_+(\cP(T))=i\}}V(\cP(T)(\cI))
%$$
%and whose the differential is given by
%$$
%D^V=\sum_{\cI<\cI'}D^V_{\cI<\cI'},
%$$
%where the sum is taken over all pairs $(\cI,\cI')$
%such that $\cI'$ is an immediate successor of $\cI$.

%We endow $CV$ with an additional grading, called the
%Jones grading, by regarding $V(\cP_{\cI}(T))$ as a
%graded vector space
%$$
%V(\cP\cI))=\bigoplus_j V(\cP(\cI),j)
%$$
%where ..
%$$
%V(\cP_{\cI}(T),j):=\bigsum\Wedge^k Z(\Diagram(\cI))
%$$
%for
%$$
%2k = \dim_{\Ztwo} Z(\Diagram(\cI)) - j + n_-(\Diagram) -2n_+(\Diagram) + \sum_i \cI_i.
%$$

%$$
%V(T):=H_*(CV(\cP(T)),D^V).
%$$

\begin{remark} \label{rmk:ReltoKhov}
In \cite{MR2124557}, Khovanov associates to a knot, $K \subset S^3$, a
bigraded homology group for each $n\in\Z_{>0}$,
here denoted $\widetilde{Kh}_n(K)$, whose graded Euler characteristic is the reduced $n$--colored Jones polynomial, $\widetilde{J}_n(K)$, of $K$:
$$
\widetilde{J}_n(K)=\sum_{i,j}(-1)^iq^j\dim_{\mathbb{Z}_2}\widetilde{Kh}_n(K)^{i,j}.
$$
These groups are related to the complexes described in this section as follows:
\[\widetilde{Kh}_n(\overline{K}):=V(T^n),\] where $\overline{K}$ is the mirror of $K$, and $T^n$ is the admissible, balanced tangle obtained by removing the neighborhood of a point $p \in K$ and taking the $n$--cable, $T^n$, of the resulting tangle.  $\overline{K}$ appears above, as in \cite{MR2141852}, because Khovanov and Ozsv{\'a}th-\Szabo use opposite conventions for the ``0'' and ``1'' resolutions.  To define the absolute $(i,j)$ gradings, we use the orientation convention for $T^n$ specified in \cite[Sec. 4]{MR2124557}.
\end{remark}

\subsection{Sutured Floer functor} \label{sec:SFFunctor}
In the previous subsection, we described a Khovanov-type ``functor'' which assigns to each ABR, $T \subset D \times I$, a free module over $\Wedge^*(Z(T))$ which is 
\begin{itemize} 
  \item rank $0$ if $T$ backtracks,
  \item rank $1$ otherwise,
\end{itemize}
and which assigns a module homomorphism to each saddle cobordism between ABR projections.

We now describe a sutured Floer-type ``functor''.   Proposition \ref{prop:KhovSFHFunctEquiv} will prove the equivalence of the two.  As before, singular (co)homology groups, where unspecified, will be taken with $\Z$ coefficients.  We will use $H_*^\circ(-)$ to denote $H_*(-;\Z)/\mbox{Tors}$ and $Hom(-)$ denote $Hom(-,\Z)$.

The sutured Floer-type ``functor'' associates
\begin{itemize}
  \item to an ABR, $T$, the sutured Floer homology of $ Y = \boldSigma(D \times I, T)$, considered as a module over $\Wedge^*(H^\circ_1(Y,\partial Y)\otimes_\Z \Ztwo)$ as described in Section \ref{sec:H1Action},
  \item to a saddle cobordism $S$ between two ABR projections, $\cP(T')$ and $\cP(T'')$ the induced map \[SFH(\boldSigma(D \times I, T')) \rightarrow SFH(\boldSigma(D \times I, T''))\] on sutured Floer homology obtained by counting triangles.  More precisely, if $\cP(T')$ and $\cP(T'')$ are related by a saddle cobordism, then $\boldSigma(D \times I, T'')$ can be obtained from $\boldSigma(D \times I, T')$ by means of a single surgery on an imbedded knot.  After constructing a sutured Heegaard triple-diagram, $(\Sigma, \boldalpha, \boldeta', \boldeta'')$, subordinate to this knot as in Section \ref{sec:LinkSurg}, the map (equipped with an action of $H_1$) \[SFH(Y_{\alpha,\eta'}) \rightarrow SFH(Y_{\alpha, \eta''})\] is given as in Section \ref{sec:H1action4man}.

We denote the induced map described above by $\mathcal{F}_m$ (resp., $\mathcal{F}_\Delta$) if $S$ is a merge (resp., split) saddle cobordism.  If $S$ is a zero saddle cobordism, the induced map will be $0$, by Lemma \ref{lemma:FloerFunctor}.
\end{itemize}

\begin{lemma} \label{lemma:FloerFunctor}Let $T$ be an ABR.  Then \[SFH(\boldSigma(D \times I, T)) = \begin{cases}
  0 & \mbox{if $T$ backtracks}\\
  \Lambda^*(H^\circ_1(Y,\partial Y)\otimes_\Z \Ztwo) & \mbox{otherwise}
  \end{cases}\]
\end{lemma}

\begin{proof}
Let $T$ be an ABR tangle representative with connected components $T_1, \ldots, T_{k+a}$, where $T_1, \ldots, T_k$ satisfy $T_i \cap \partial(D \times I) = \emptyset$ and $T_{k+1}, \ldots, T_{k+a}$ satisfy $T_i \cap \partial(D \times I) \neq \emptyset$.

We may assume without loss of generality (by replacing $T$ by another representative in its equivalence class if necessary) that $T_i$ and $D_{\frac{3}{2}}$ intersect transversely and $T_i \cap D_{\frac{3}{2}} \neq \emptyset$ for all $i \in \{1, \ldots, k+a\}$.  Let $T_{\alpha}$ (resp., $T_\beta$) denote $T \cap D \times [-1,\frac{3}{2}]$ (resp., $T \cap D \times [\frac{3}{2},4]$).  Figures \ref{fig:SFHNoBackTrack} and \ref{fig:SFHBackTrack} then illustrate how to construct a Heegaard diagram for $\boldSigma(D \times I, T)$.

In particular, the sutured Heegaard surface, $\Sigma$, for a sutured Heegaard decomposition of $\boldSigma(D \times I, T)$ is $\boldSigma(D_{\frac{3}{2}},\vec{p})$, where $\vec{p} = T \cap D_{\frac{3}{2}}$.  Furthermore, if \[\pi:\boldSigma(D\times I,T) \rightarrow D \times I\] is the branched covering projection, and $\pi_{\Sigma}$ is its restriction to the $\{\frac{3}{2}\}$ level, then if $a_i$ (resp., $b_i$) is the image of any cup (resp., cap) under an isotopy that fixes $T_i \cap D_{\frac{3}{2}}$ and moves the cup (resp., cap) into $D_{\frac{3}{2}}$, then $\pi_{\Sigma}^{-1}(a_i)$ (resp., $\pi_\Sigma^{-1}(b_i)$) bounds a disk in $\boldSigma(D \times [-1,\frac{3}{2}],T_\alpha)$ (resp., $\boldSigma(D \times [\frac{3}{2},4], T_\beta)$) and hence is an $\alpha$ (resp., $\beta$) curve on $\Sigma$.  To see this, simply observe that if $A_i$ (resp., $B_i$) is the associated isotopy, then $\pi^{-1}(A_i)$ (resp., $\pi^{-1}(B_i)$) is a disk in in $\boldSigma(D \times I, T)$ with boundary $\alpha_i$ (resp., $\beta_i$).

If $T$ backtracks, then the above procedure produces an admissible Heegaard diagram with at least one pair $(\alpha_i,\beta_j)$ of $\alpha$ and $\beta$ curves that do not intersect.  This implies that $\mathbb{T}_\alpha \cap \mathbb{T}_\beta = \emptyset$ and, hence, $SFH(\boldSigma(D\times I, T)) = 0$, as desired.

If $T$ does not backtrack, then $\boldSigma(D \times I,T)$ is of the form \[(F \times I) \# k(S^1 \times S^2),\] where $F$ is an oriented surface with boundary.  Juh{\'a}sz's connected sum formula \cite[Prop. 9.15]{MR2253454}, coupled with the fact that $\widehat{HF}(Y)$ is a rank one free $\Wedge^*(H^\circ_1(Y)\otimes_\Z \Ztwo)$-module when $Y = \# k(S^1 \times S^2)$ (cf. \cite[Prop. 6.1]{MR2141852}) completes the proof.
\end{proof}

\begin{figure}
\begin{center}
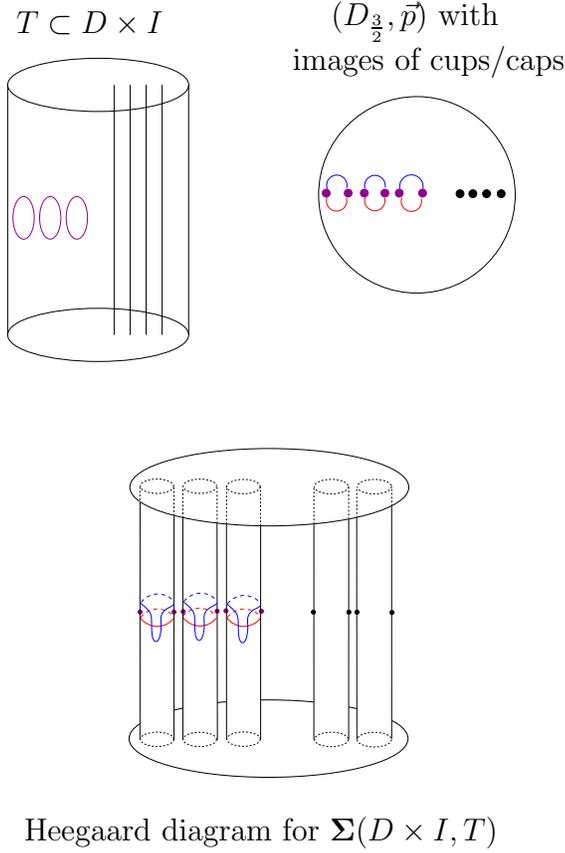
\end{center}
\caption{Construction of an admissible Heegaard diagram for $\boldSigma(D \times I, T)$ when $T$ is an ABR with no backtracking.  By an ambient isotopy relative to $(\partial D) \times I$, arrange for each connected component of $T$ to intersect $D_{\frac{3}{2}}$.  Then the Heegaard surface for a Heegaard decomposition of $\boldSigma(D \times I, T)$ is $\boldSigma(D,\vec{p})$, where $\vec{p} = T \cap D_{\frac{3}{2}}$ and the $\alpha$ (resp., $\beta$) curves correspond to the preimages of the projections of the cups (resp., caps) to $D_{\frac{3}{2}}$.}
\label{fig:SFHNoBackTrack}
\end{figure}

\begin{figure}
\begin{center}
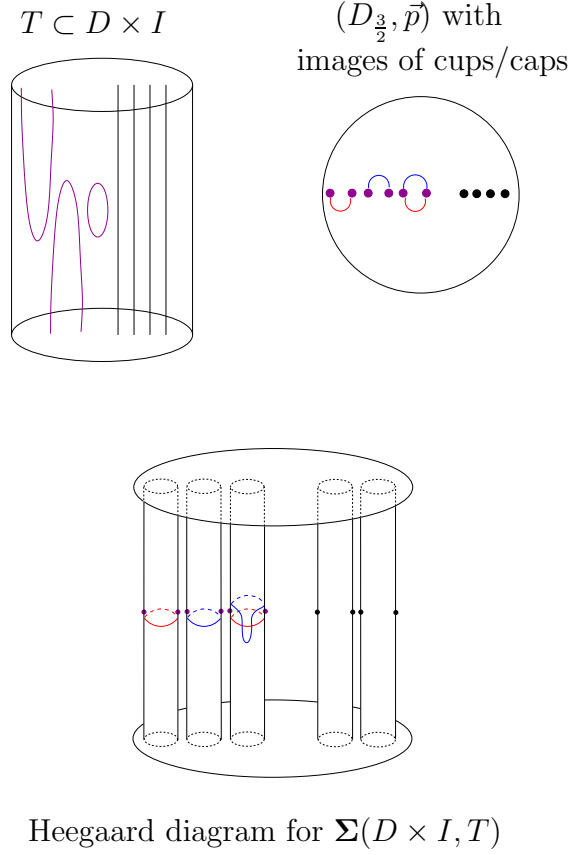
\end{center}
\caption{Construction of an admissible Heegaard diagram for $\boldSigma(D \times I, T)$ when $T$ is an ABR with backtracking.  The Heegaard diagram is constructed as described in the previous figure.  The $\alpha$ and $\beta$ curves corresponding to a backtracking cup/cap pair do not intersect, hence $\mathbb{T}_\alpha \cap \mathbb{T}_\beta = \emptyset$, implying that $SFH(\boldSigma(D \times I, T)) = 0$.}
\label{fig:SFHBackTrack}
\end{figure}

\subsection{Equivalence of Khovanov and Sutured Floer functors} \label{sec:KhSFFunctorEquiv}
We will now prove that the two functors introduced in
Subsections \ref{sec:KhFunctor} and \ref{sec:SFFunctor} are naturally isomorphic.  As before, singular (co)homology groups, where unspecified, will be taken with $\Z$ coefficients.  We will use $H_*^\circ(-)$ to denote $H_*(-;\Z)/\mbox{Tors}$ and $Hom(-)$ to denote $Hom(-,\Z)$.

\begin{proposition}\label{prop:KhovSFHFunctEquiv}
(analogue of \cite[Prop. 6.1]{MR2141852})
For each ABR tangle $T \subset D \times I$ there is a canonical isomorphism:
$$
\psi_T\colon V(T)\stackrel{\cong}{\longrightarrow}
SFH(\boldSigma(D\times I,T))
$$
which is natural, in the sense that, whenever $S:T'\rightarrow T''$
is a saddle cobordism, the following
diagram commutes:
$$
\begin{CD}
V(T')@>{\mathcal{V}_S}>>V(T'')\\
@V{\psi_{T'}}VV @VV{\psi_{T''}}V \\
SFH(\boldSigma(D\times I,T'))@>{\mathcal{F}_S}>> 
SFH(\boldSigma(D\times I,T''))
\end{CD}
$$
\end{proposition}

Here, $\mathcal{V}_S, \mathcal{F}_S$ are $\mathcal{V}_m, \mathcal{F}_m$ (resp., $\mathcal{V}_\Delta$, $\mathcal{F}_\Delta$) if $S$
is a merge (resp., split) saddle cobordism, and  $\mathcal{V}_S, \mathcal{F}_S$ are both zero if $S$ is a zero saddle cobordism (and, hence, one of $T'$
or $T''$ backtracks).

\begin{proof}
We begin by exhibiting a canonical isomorphism \[\Psi_T: V(T) \rightarrow SFH(\boldSigma(D \times I, T)).\]

If $T$ backtracks, $SFH(\boldSigma(D\times I,T))=0=V(T)$ by Definition~\ref{defn:KhFunctor} and Lemma~\ref{lemma:FloerFunctor}, so there is nothing to prove.

Now suppose $T$ does not backtrack. Definition~\ref{defn:KhFunctor} and Lemma~\ref{lemma:FloerFunctor} then tell us that
$$
V(T)=\Wedge^*Z(T)
$$
and
$$
SFH(Y)=\Wedge^*(H^\circ_1(Y,\partial Y)\otimes_\Z \Z_2)
$$
where $Y:=\boldSigma(D\times I,T)$. Since $V(T)$ and $SFH(Y)$ are both rank one modules over a freely-generated exterior algebra, it suffices to exhibit an isomorphism between their generating sets.  More precisely, we will constuct generators for $H^\circ_1(Y,\partial Y)\otimes_\Z \Ztwo$ which are in one to one correspondence with the formal exterior algebra generators of $V(T)$.

%To prove that
%$V(T)$ and $SFH(Y)$ are isomorphic, it is therefore
%sufficient to show that there is an isomorphism
%$i:Z(T)\rightarrow H_1(Y,\partial Y)$;
%we can then define $\psi_T:V(T)\rightarrow Z(T)$
%as $\psi(T):=\Wedge^*i$.

To this end, we describe a procedure analogous to the one given in the proof of \cite[Prop. 6.2]{MR2141852}, which produces a canonical basis for $H^\circ_1(Y,\partial Y)\otimes_\Z \Ztwo$.  Begin by picking a basepoint $c\in\{(x,y,z)\in (\partial D)\times I\,\,|\,\,y>0\}$ and a path, $\tau_i$, connecting $c$ to a point, $p_i$, on each connected component $T_i \subset T$.  Recall that we are assuming that $T$ is the disjoint union of $k+a$ connected components

Let \[\pi:(Y = \boldSigma(D\times I, T))\rightarrow (D \times I)\] denote the branched covering projection, and consider $\mathcal{T}_i := \pi^{-1}(\tau_i) \subset Y$.

\begin{lemma}
\[H^\circ_1(Y,\partial Y) \otimes_\Z \Ztwo = Span_{\Ztwo}([\mathcal{T}_1], \ldots, [\mathcal{T}_{k+a}])/[\mathcal{T}_{k+1}]\sim\ldots\sim[\mathcal{T}_{k+a}]\sim 0.\]
\end{lemma}

\begin{proof}
It is convenient to make the identification $H^\circ_1(Y,\partial Y)\otimes_\Z \Ztwo \cong H_1(Y,\partial Y;\Ztwo)$.  This follows from the universal coefficient theorem, since $H_1(Y,\partial Y)$ has no torsion.  
%That $H_1(Y,\partial Y)$ is torsion-free follows, in turn, by examining the long exact sequence on the pair $(Y,\partial Y)$: \[H_1(\partial Y) \stackrel{i_{1^*}}{\longrightarrow} H_1(Y) \stackrel{q_*}\longrightarrow H_1(Y,\partial Y)\longrightarrow H_0(\partial Y)\stackrel{i_{0^*}}\longrightarrow H_0(Y).\]  Since $i_{0^*}$ is an isomorphism, $q_*$ is surjective; hence, \[H_1(Y,\partial Y)\,\,\cong\,\,
%H_1(Y)/i_{1^*}(H_1(\partial Y))\,\,\cong\,\, H_1(\Sigma)/\left(Span\{[\boldalpha],[\boldbeta]\}, i_{1^*}(H_1(\partial Y))\right),\] where here $(\Sigma, \boldalpha,\boldbeta)$ is a sutured Heegaard diagram for $Y$, constructed as described in Section \ref{sec:SFFunctor}, and $i_{1^*}(H_1(\partial(Y))) \subset H_1(\Sigma)$ makes sense, since $Y$ is a product away from the $2$--handles attached along $\boldalpha$ and $\boldbeta$.  In particular, $\{\boldalpha, \boldbeta, i_{1^*}(H_1(\partial Y))$ are all primitive in $H_1(\Sigma)$.  Hence, $H_1(Y,\partial Y)$ is torsion-free, and $H^\circ_1(Y,\partial Y)\otimes_\Z \Ztwo\cong H_1(Y,\partial Y;\Ztwo)$.

To see this, note that $Y = (F \times I) \# k(S^1 \times S^2)$.  Therefore, $H_1(Y,\partial Y) \cong H^2(Y) \cong H^2(k(S^1 \times S^2))$ is torsion-free.
Furthermore, it is easy to see that $\{\mathcal{T}_i\}_{i=1}^k$ form a basis, as claimed, since the set of Hom duals, $\{[\mathcal{S}_1^*], \ldots, [\mathcal{S}_k^*]\}$, of the $k$ linearly-independent $S^2$'s forms a basis for $H^2(Y;\Ztwo)$.  But $[\mathcal{S}_i^*] = PD([\mathcal{T}_i])$ for each $i$, thus $\{\mathcal{T}_i\}_{i=1}^k$ is a basis for $H^\circ_1(Y,\partial Y)\otimes_\Z \Ztwo$.
\end{proof}

%The inclusions
%$C\stackrel{i_1}{\hookrightarrow}\partial Y
%\stackrel{i_2}{\hookrightarrow} Y$
%induce an exact homology sequence
%$$ 
%H_1(\partial Y)
%\stackrel{i_{2*}}{\longrightarrow}
%H_1(Y,C)
%\longrightarrow
%H_1(Y,\partial Y)\longrightarrow
%H_0(\partial Y,C)=0
%$$
%and hence
%$$
%H_1(Y,\partial Y)\cong
%H_1(Y,C)/i_{2*}(H_1(\partial Y,C)).
%$$
%...
%Define paths $\tau_i\subset \{(x,y,z)\in D\times I|y\geq 0\}$
%which start at $\tau_i(0)=c$ and whose endpoint $\tau_i(1)$
%lies on the $i$th component of the tangle $T$.
%Let $\tilde{\tau_i}:=\pi^{-1}(\tau_i)$.
%...
%$$
%H_1(Y,\partial Y)\cong
%Span_{\Ztwo}([\tilde{\tau}_1],\ldots,[\tilde{\tau}_{k+l}])/
%[\tilde{\tau}_{k+1}]\sim\ldots\sim[\tilde{\tau}_{k+l}]\sim 0.
%$$
Comparing with the definition of $Z(T)$, it is clear that the assignment $[T_i] \mapsto [\mathcal{T}_i]$
induces the desired isomorphism \[\Psi_T:V(T) \mapsto SFH(\boldSigma(D\times I, T)).\] 

Naturality of this isomorphism under saddle cobordisms follows exactly as in \cite[Prop. 6.1 \& Prop. 6.2]{MR2141852}.  In particular:

\begin{enumerate}
\item If $S_m$ is a merge saddle cobordism from the ABR projection, $\cP(T')$, to the ABR projection, $\cP(T'')$,
and $Y'$ (resp., $Y''$) denote the
sutured manifold $\boldSigma(D\times I,T')$
(resp., $\boldSigma(D\times I,T'')$).
Then $\mathcal{F}_m$ is the composition of the maps
\[
\begin{CD} H_1(Y',\partial Y';\Ztwo) @>p>> 
\frac{H_1(Y',\partial Y';\Ztwo)}{[\mathcal{T}_i']\sim [\mathcal{T}_j']
}@>a>> H_1(Y'',\partial Y'';\Ztwo)
\end{CD}
\]
\item If $S_\Delta$ is a split saddle cobordism from $\cP(T'')$ to $\cP(T')$,
and $Y''$ (resp., $Y'$) denote the
sutured manifold $\boldSigma(D\times I,T'')$
(resp., $\boldSigma(D\times I,T')$).
Then $\mathcal{F}_\Delta$ is the composition of the maps
$$
\begin{CD}H_1(Y'',\partial Y'';\Ztwo) @>{a^{-1}}>> 
\frac{H_1(Y',\partial Y';\Ztwo)}{[\mathcal{T}_i']\sim [\mathcal{T}_j']}@>f>>
H_1(Y',\partial Y';\Ztwo).
\end{CD}
$$
\end{enumerate}

A direct inspection of the associated Heegaard triple-diagram as in \cite[Prop 6.1]{MR2141852} tells us that the map behaves as stated on the canonical top-degree generator.  The rest follows from the naturality of triangle maps under the $4$--manifold $H^\circ_1$--action (see Proposition \ref{prop:H1action4man}), once we note that \[[\mathcal{T}'_m] = [\mathcal{T}''_n] \in H^\circ_1(X, Z)\] iff $T'_m \subset T'$ is one of the two components which merges to yield $T''_n \subset T''$.
\end{proof}

\subsection{Proof of Spectral Sequence from Khovanov to Sutured Floer}

Our aim in the present section is to prove:

\begin{theorem}\label{thm:SpecSeq}
Let $K \subset S^3$ be an oriented knot and $\overline{K} \subset S^3$ its mirror.  For each $n \in \mathbb{Z}_{> 0}$, there is a spectral sequence, whose $E^2$ term is $\widetilde{Kh}_n(\overline{K})$ and whose $E^\infty$ term is $SFH(\boldSigma(D \times I,T^n))$.
\end{theorem}

\begin{proof}
Recall (see Remark \ref{rmk:ReltoKhov}) that $\widetilde{Kh}_n(\overline{K}) = V(T^n)$, where $T^n \subset D \times I$ is obtained by removing a point from $K$ and taking the $n$--cable.  Therefore, Theorem \ref{thm:SpecSeq} can be seen as a specific instance of the following more general result:  

\begin{proposition} \label{prop:SpecSeq}
Let $T \subset D \times I$ be an admissible, balanced tangle.  Then there is a spectral sequence whose $E^2$ term is $V(T)$ and whose $E^\infty$ term is $SFH(\boldSigma(D \times I, T))$. 
\end{proposition}

\begin{proof}
Choose a projection $\mathcal{P}(T)$ as in Definition \ref{defn:projection} and assign labels $1, \ldots, \ell$ to its crossings.  Letting $T^i_0$ (resp., $T^i_1$) denote the tangle whose projection is obtained from $\mathcal{P}(T)$ by replacing a neighborhood of the crossing by the appropriate resolution as in Figure \ref{fig:Resolutions}, we see that $\boldSigma(D\times I,T)$, $\boldSigma(D \times I,T^i_0)$, and $\boldSigma(D\times I,T^i_1)$ are sutured manifolds related by a triple ($\infty, 0,$ and $1$, respectively) of surgeries on the knot in $\boldSigma(D \times I, T)$ which is the preimage of the dotted arc in Figure \ref{fig:PreimageK}.  

\begin{figure}
\begin{center}
\input{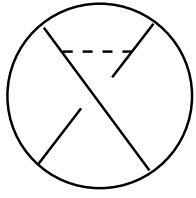}
\end{center}
\caption{$\boldSigma(D \times I, T_0)$ and $\boldSigma(D \times I, T_1)$ are obtained from $\boldSigma(D \times I, T)$ by doing $0$ and $1$ surgery, resp., on the preimage of the dotted arc.}
\label{fig:PreimageK}
\end{figure}

Now consider the link $L \subset \boldSigma(D \times I, T)$ consisting of the preimages of all such dotted arcs at all crossings of $\mathcal{P}(T)$.  Every sutured manifold obtainable as the double branched cover of some combination of resolutions on some subset of the crossings $1, \ldots, \ell$ can then be described as $Y(\mathcal{I})$, where $\mathcal{I} \in \{0,1,\infty\}^\ell$, in the notation of Section \ref{sec:LinkSurg}.  Note that $Y(\infty, \ldots, \infty) = \boldSigma(D \times I, T)$.

By the link surgeries spectral sequence (Proposition \ref{prop:LinkSurg}), we conclude that there is a spectral sequence whose $E^2$ term is given by the homology of the complex \[\left(\bigoplus_{\mathcal{I} \in \{0,1\}^\ell} SFH(Y(\mathcal{I})),D_1\right)\] where \[D_1|_{SFH(Y(\mathcal{I}))}: SFH(Y(\mathcal{I})) \rightarrow SFH(Y(\mathcal{I}'))\] is the map induced by counting holomorphic triangles (see Equation \ref{eqn:Diff}) for $\mathcal{I}'$ an immediate successor of $\mathcal{I}$.

But Proposition \ref{prop:KhovSFHFunctEquiv} tells us that this complex is the same as the complex whose homology is $V(T)$.  The proposition (and, hence, the theorem) follows.
\end{proof}
\end{proof}

\section{Relationship to Knot Floer Homology} \label{sec:RelKnotFloer}
%Let $K \subset S^3$ be a knot and $n \in \mathbb{Z}_{> 0}$.  Denote by $T^n$ the $n$--cable of the tangle obtained by removing the neighborhood of a point on $K$, and $\boldSigma(D \times I,T^n)$ the double-branched cover of $D \times I$ with branch set $T^n$.  
In this section we prove:

%  Then there is a very simple description of $\Sigma(T^n)$ in terms of the trivial (when $n$ is even) and non-trivial (when $n$ is odd) double-covers of $S^3 - K$.  This will allow us to relate $SFH(\Sigma(T^n))$ for $n=1$ to $\widehat{HF}(\Sigma(K))$.  For $n>1$, we relate $SFH(\Sigma(T^n))$ to $\widehat{HFK}(S^3, K \# K)$ (when $n$ is even) and $\widehat{HFK}(\Sigma(K),\widetilde{K})$ when $n$ is odd.
%More precisely, we have

\begin{theorem} \label{thm:KnotFloerRel}
Let $n \in \mathbb{Z}_{> 0}$.

\[SFH(\boldSigma(D \times I,T^n)) \cong \left\{\begin{array}{cc}
    \widehat{HF}(\boldSigma(S^3,K)) & \mbox{if $n = 1$},\\
    \widehat{HFK}(S^3, K \# K) & \mbox{if $n$ is even},\\
    \widehat{HFK}(\boldSigma(S^3,K), \widetilde{K}) & \mbox{if $n>1$ and odd}
    \end{array}\right.\] 
\end{theorem}

The behavior of the relative Spin$^c$ (Alexander) and homological (Maslov) gradings under the above isomorphisms is explained in detail in Proposition \ref{prop:gradingsodd}.

\begin{proof}
The proof depends upon the observation that, as a sutured manifold,  $\boldSigma(D \times I,T^n)$ can be constructed by gluing two sutured pieces in the sense of Definition \ref{defn:gluing}:

\begin{enumerate}
\item the appropriate double-cover of $S^3 - K$ (trivial if $n$ is even and non-trivial if $n$ is odd), and 
\item  $\boldSigma(D,\vec{p}) \times I$, where $\vec{p} = \{p_1, \ldots, p_n\}$ are $n$ distinct points on a standard, oriented disk $D$.
\end{enumerate}

More precisely, we claim:

\begin{proposition} \label{prop:GlueSutMan}
\[\boldSigma(D \times I, T^n) = Y \,\, \cup_{\gamma_i} \,\, \boldSigma(D \times I,\vec{p} \times I),\] where
\[ Y = \left\{\begin{array}{cc}
                (S^3 - K)_1 \,\amalg\, (S^3-K)_2 & \mbox{if $n$ is even,}\\
                \boldSigma(S^3,K) - \widetilde{K} & \mbox{if $n$ is odd,}
		\end{array}\right.\]

$\gamma_2 = \partial\boldSigma(D ,\vec{p}) \times I,$ and

\[\gamma_1 = \left\{\begin{array}{cc}
                 N(\mu_1) \amalg N(\mu_2') & \mbox{if $n$ is even,}\\
                 N(\pi^{-1}(\mu)) & \mbox{if $n$ is odd}
		\end{array}\right.\]
\end{proposition}

Here, $\widetilde{K}$ denotes the preimage of $K$ in $\boldSigma(S^3,K)$, $\mu_i$ (resp., $\mu_i'$) for $i=1,2$ represents the meridian (resp., oppositely-oriented meridian) for $(S^3 - K)_i$, and $\pi$ denotes the branched covering projection $\pi: \boldSigma(S^3,K) \rightarrow S^3$.   In particular, $\boldSigma(S^3,K) - \widetilde{K}$ is the nontrivial double-cover of $S^3 - K$.  See Examples \ref{example:knot} and \ref{example:prodsurface} to understand how to identify the above as sutured manifolds.

\begin{figure}
\begin{center}
\input{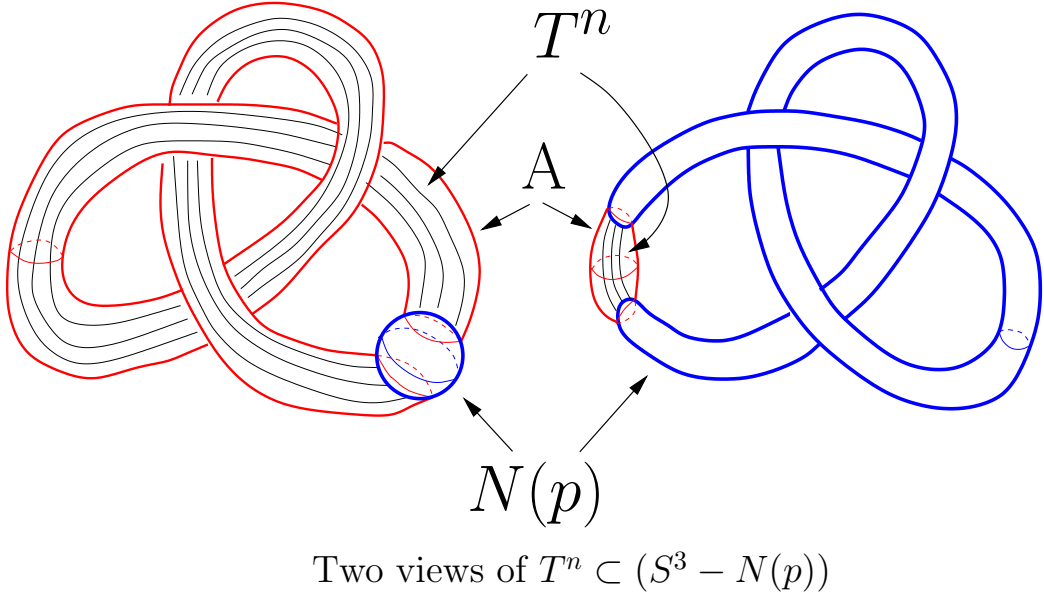}
\end{center}
\caption{Two views of the $n$-cable of $K- N(p)$ for $n=3$, where $K \subset S^3$ and $p$ is a point on $K$.  In the picture on the left, $N(p)$ appears as a small ball, whereas in the picture on the right, $N(p)$ has been stretched to include almost all of $N(K)$. Note that the pair $(S^3-N(p), T^n)$ can be obtained by gluing $(S^3 - N(K), \emptyset)$ to $(D \times I, \{p_1, \ldots p_n\} \times I)$ along the red annulus $A$.}
\label{fig:AmbIsotopy}
\end{figure}

\begin{proof}
By an ambient isotopy in $S^3$, we can view the pair $(D \times I,T^n)$ as in Figure \ref{fig:AmbIsotopy}.  By splitting along the imbedded annulus, $A$ (pictured), we see that 
\[(D \times I,T^n) = (S^3 - K) \,\, \cup_{A} \,\, (D \times I,\vec{p} \times I).\]
To understand $\boldSigma(D \times I,T^n)$ as a sutured manifold, we lift this decomposition to the double-branched cover.  Note that the branch set is contained in the $(D \times I)$ piece downstairs.  Therefore, $\boldSigma(D \times I, T^n)$ is the union of two pieces:

\begin{enumerate}
\item $\boldSigma(D \times I, \vec{p} \times I)$,
\item one of the two double-covers of $S^3 - K$.
\end{enumerate}

Let $Y$ denote the appropriate double-cover of $S^3-K$ and \[\pi: Y \rightarrow S^3-K\] the covering map.  Note that the $2$--fold covers of $S^3-K$ are distinguished by whether or not $\pi^{-1}(\mu)$ is connected, where $\mu$, as usual, denotes the image of the oriented meridian of $K$ in $\partial(S^3 - N(K))$.  But \[\pi^{-1}(\mu) = \partial\boldSigma(D_{\frac{3}{2}},\vec{p}).\] Therefore, $\pi^{-1}(\mu)$ is connected if $n$ is odd and disconnected if $n$ is even; hence, $Y$ is the nontrivial double-cover if $n$ is odd, and the trivial double-cover if $n$ is even.

The result follows.
\end{proof}

We will now address the even and odd cases separately.

{\flushleft \bf Odd Case:}
Recall \cite[Prop. 9.2]{MR2253454} that if $K \subset Y$ is an oriented knot in a closed, connected, oriented $3$--manifold $Y$, then 
\[SFH(Y-K) \cong \widehat{HFK}(Y,K).\]
This is proved by comparing a sutured Heegaard diagram $\Sigma$ for $Y-K$ satisfying \[\partial \Sigma = \mu \cup \mu',\] where $\mu$ represents the meridian of $K$ and $\mu'$ represents the meridian with the opposite orientation, to the doubly-pointed Heegaard diagram $\widehat{\Sigma}$ obtained by capping off the $\mu$ boundary component with a disk containing a basepoint labeled $z$ and the $\mu'$ boundary component with a disk containing a basepoint labeled $w$.  One easily checks that the associated chain complexes are isomorphic.

Let $(\Sigma, \boldalpha,\boldbeta)_F$ be the standard balanced sutured Heegaard diagram for $\boldSigma(D \times I, \vec{p} \times I)$:

\begin{enumerate}
\item $\Sigma_F = \boldSigma(D,\vec{p})$, and 
\item ${\boldalpha}_F = {\boldbeta}_F = \emptyset$.  
\end{enumerate}

When $n=2k+1$ for $k \in \mathbb{Z}_{\geq 0}$, $\Sigma_F = \boldSigma(D,\vec{p})$ is the surface $F_{k,1}$ with boundary $\gamma$.  If $(\Sigma, \boldalpha, \boldbeta)_K$ is any balanced sutured Heegaard diagram for $\boldSigma(S^3-K)$ as above, then Proposition \ref{prop:GlueSutMan} and Lemma \ref{lemma:GlueHD} together imply that 
\begin{equation} \label{eqn:GlueHDodd}
(\Sigma,\boldalpha,\boldbeta)_n := (F_{k,1} \cup_{\mu \sim -\gamma} \Sigma_K, \boldalpha_K,\boldbeta_K)
\end{equation}
is a balanced, sutured Heegaard diagram for $\boldSigma(D \times I,T^n)$.  See Figure \ref{fig:OddHD}.

\begin{figure}
\begin{center}
\input{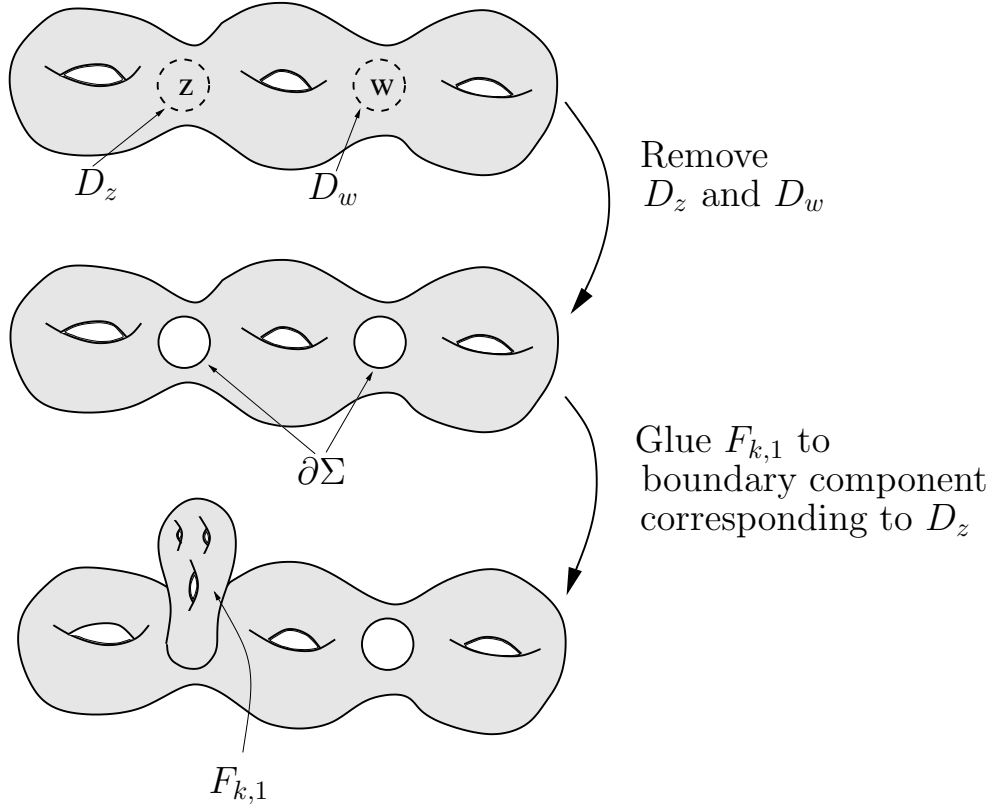}
\end{center}
\caption{Obtaining a balanced sutured Heegaard diagram $(\Sigma, \boldalpha,\boldbeta)_n$ for $\boldSigma(D \times I,T^n)$, $n=2k+1$.  The figure in the top row represents a doubly-pointed Heegaard diagram for $(\boldSigma(S^3,K),\widetilde{K})$.  One obtains a sutured Heegaard diagram for $\boldSigma(S^3-K)$ (middle row) by removing small disks around $w$ and $z$ and a sutured Heegaard diagram for $\boldSigma(D\times I, T^n)$ by gluing $F_{k,1}$ along the boundary component corresponding to $D_z$.}
\label{fig:OddHD}
\end{figure}

When $k=0$, $\Sigma_F = F_{0,1}$ is just a disk.  The sutured Floer chain complex associated to $(\Sigma,\boldalpha, \boldbeta)_1$ is then easily seen to be isomorphic to the Heegaard Floer chain complex which computes $\widehat{HF}(\boldSigma(S^3,K)),$  since $\widehat{HF}(\boldSigma(S^3,K))$ is computed using $\widehat{\Sigma}$, the Heegaard diagram obtained from $\Sigma$ by gluing a disk containing a basepoint, $w$, to the lone boundary component of $\Sigma$ and counting holomorphic disks in the differential that miss $w$.  Hence,
\[SFH(\boldSigma(D \times I,T^1)) \cong \widehat{HF}(\boldSigma(S^3,K)).\]

Now suppose $k\geq 1$.  Let $Y$ denote the sutured manifold $\boldSigma(D \times I,T^n)$.  Since $k$ is the genus of $F_{k,1} = (\boldSigma(D,\vec{p}))$, there exists some oriented simple closed curve $C \subset \Sigma_F = F_{k,1}$ for which $[C] \neq 0$ in $H_1(F_{k,1};\Z)$ and, hence, $\neq 0$ in $H_1(R(\Gamma),\Z)$.  Recall (see Definition \ref{defn:suturedman}) that $R(\Gamma)$ is the complement of the suture neighborhoods in $\partial Y$.

Define $S = C \times I$, endowed with the orientation compatible with $\partial S = (-C \times \{-1\}) \cup (C \times \{4\})$.  $S$ is now a {\em decomposing surface} for the sutured manifold $Y = \boldSigma(D \times I,T^n)$ which satisfies the conditions of \cite[Thm. 1.3]{MR2390347}.  We claim:

\begin{lemma} \label{lemma:CutalongS}
Let $(Y,\Gamma)$ and $S$ be as above, and $(Y',\Gamma')$ the sutured manifold obtained by decomposing along $S$.  Then \[SFH(Y',\Gamma') \cong SFH(Y,\Gamma)\]
\end{lemma}

\begin{proof}
An application of \cite[Thm. 1.3]{MR2390347} tells us that if $(Y',\Gamma')$ is the sutured $3$--manifold obtained from $(Y,\Gamma)$ by decomposing along $S$, then
\[SFH(Y',\Gamma') = \bigoplus_{\mathfrak{s} \in O_S} SFH(Y,\Gamma,\mathfrak{s})\,\]
where $O_S$ is the set of Spin$^c$ structures in Spin$^c(Y,\partial Y)$ which are {\em outer} with respect to $S$, in the sense of \cite[Defn. 1.1]{MR2390347}.  

Now, define $P$ to be the closed annulus $\overline{N(C)} \subset \Sigma$, oriented compatibly with $\Sigma$, with $\partial P = A \cup B$, where $A$ and $B$ are each connected.  $(\Sigma, \boldalpha, \boldbeta,P)$ is then a sutured Heegaard diagram for $(Y,\Gamma)$ adapted to $S' =  P \cup (A \times [-1,\frac{3}{2}]) \cup (B \times [\frac{3}{2},4])$ in the sense of \cite[Defn. 4.3]{MR2390347}.  Furthermore, $S'$ is equivalent to $S$ in the sense of \cite[Defn. 4.1]{MR2390347}.  

\cite[Lem. 5.4]{MR2390347} then tells us that $\mathfrak{s}({\bf x}) \in O_S$ iff ${\bf x} \cap P = \emptyset$.  However, there are no $\alpha$ or $\beta$ curves in the region $P$, so $\mathfrak{s}({\bf x}) \in O_S$ for all ${\bf x} \in \mathbb{T}_\alpha \cap \mathbb{T}_\beta$.

The result follows.
\end{proof}

It now remains to observe (using \cite[Defn. 4.3]{MR2390347}) that the Heegaard diagram $(\Sigma',\boldalpha',\boldbeta')$ representing the sutured manifold $(Y',\Gamma')$ obtained by decomposing along $S'$ is simply 
\[(\Sigma', \boldalpha', \boldbeta') = (\Sigma - \mbox{Int}(P),\boldalpha,\boldbeta).\]
I.e., $\Sigma'$ is precisely the Heegaard diagram obtained from $\Sigma$ by cutting out the interior of a neighborhood of $C$ (which introduces two new boundary components to $\Sigma'$).  Let $F'$ denote $p^{-1}(F_{k,1} - P)$ (where $p$ is the local diffeomorphism $p:\Sigma' \rightarrow \Sigma$ described in \cite[Defn. 4.3]{MR2390347}) and choose a point $a \in F'$.  Then if $\phi$ represents a differential in the sutured Floer complex for $(Y',\Gamma')$, $n_a(\phi) = 0$, since $F'$ is adjacent to the boundary.

Now note that, by using a Heegaard diagram obtained from $(\Sigma',\boldalpha',\boldbeta')$ by replacing $F' \subset \Sigma'$ with a disk $D_z$ containing a basepoint, $z$, and only counting holomorphic disks $\phi$ in the differential for which $n_z(\phi) = 0$, one obtains a chain complex which computes $\widehat{HFK}(\boldSigma(S^3,K),\widetilde{K})$ .

Therefore, up to a shift in homological gradings between generators in different relative Spin$^c$ structures, \[CFH(Y',\Gamma') \cong \widehat{CFK}(\boldSigma(S^3,K), \widetilde{K})\] as chain complexes. Hence, (again, up to a grading shift), \[SFH(\boldSigma(D \times I,T^n)) \cong \widehat{HFK}(\boldSigma(S^3,K),\widetilde{K})\] as desired.

{\flushleft {\bf Even Case:}}

Again, denote by $(\Sigma, \boldalpha,\boldbeta)_F$ the standard sutured Heegaard diagram for $\boldSigma(D \times I, \vec{p} \times I)$:

\begin{enumerate}
\item $\Sigma = \boldSigma(D,\vec{p})$, and 
\item $\boldalpha = \boldbeta = \emptyset$.  
\end{enumerate}

When $n=2k$ for $k \in \mathbb{Z}_{\geq 1}$, $\Sigma_F = \boldSigma(D,\vec{p})$ is the surface $F_{k-1,2}$ with boundary components $\gamma_1, \gamma_2$.  Let \[(\Sigma, \boldalpha, \boldbeta)_{K_1} \amalg (\Sigma, \boldalpha, \boldbeta)_{K_2}\] be a balanced sutured Heegaard diagram for $(S^3-K)_1 \amalg (S^3 - K)_2$.  Then, again, Proposition \ref{prop:GlueSutMan} and Lemma \ref{lemma:GlueHD} together imply that \[(\Sigma,\boldalpha,\boldbeta) := ((F_{k-1,2}) \cup_{(\mu_1 \sim -\gamma_1),(\mu_2' \sim -\gamma_2)} (\Sigma_{K_1} \amalg \Sigma_{K_2}), \boldalpha,\boldbeta)\] is a balanced sutured Heegaard diagram for $\boldSigma(D \times I,T^n)$.  See Figure \ref{fig:EvenHD}.

\begin{figure}
\begin{center}
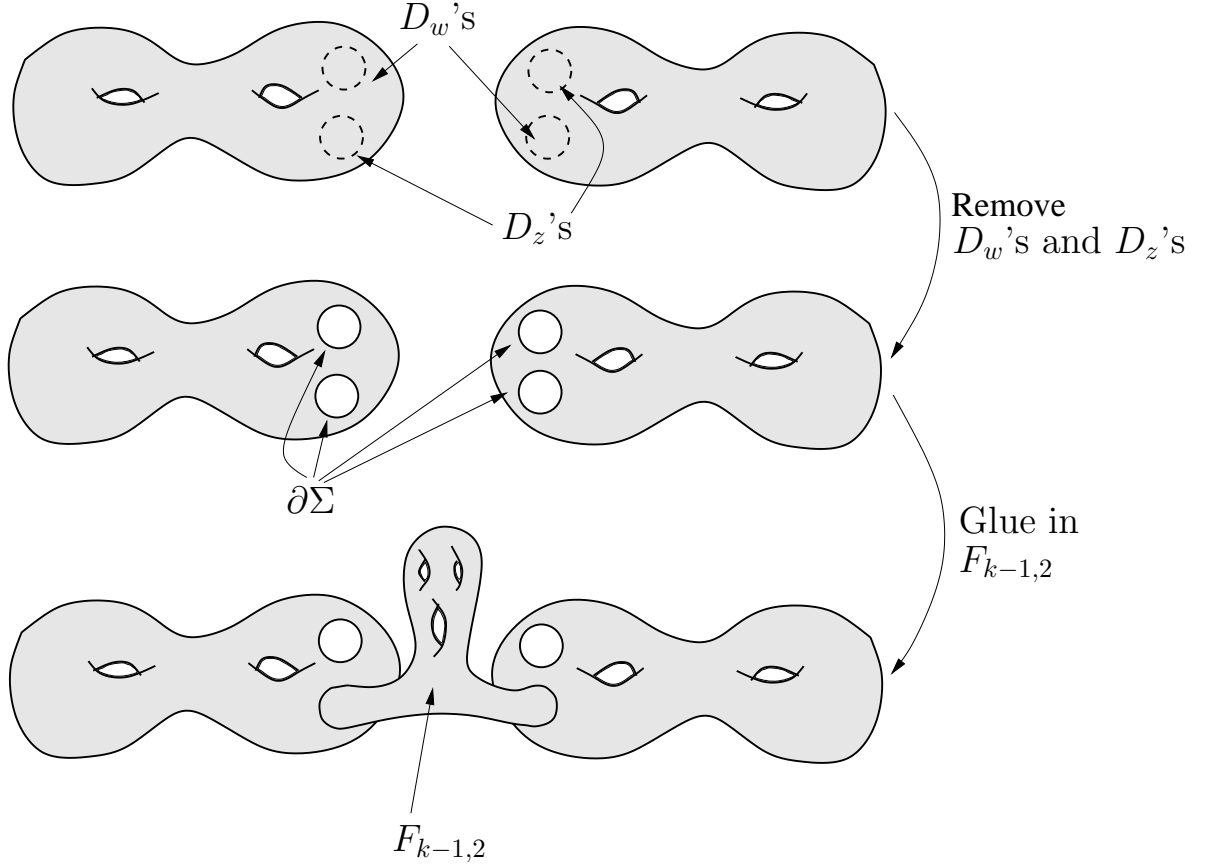
\end{center}
\caption{Obtaining a balanced sutured Heegaard diagram $(\Sigma, \boldalpha,\boldbeta)_n$ for $\boldSigma(D \times I,T^n)$, $n=2k$.  Begin with two doubly-pointed Heegaard diagrams for $(S^3,K)$ (top row), and remove neighborhoods of the $w$ and $z$ basepoints to obtain a sutured Heegaard diagram representing $(S^3 - K) \amalg (S^3 - K^r)$ (middle row).  Then glue in $F_{k-1,2}$ as shown to obtain a sutured Heegaard diagram for $\boldSigma(D \times I,T^n)$ (bottom row).}
\label{fig:EvenHD}
\end{figure}

As in the odd case, choose an oriented simple closed curve $C \subset \Sigma_F = F_{k-1,2}$ for which $[C] \neq 0 \in H_1(F_{k-1,2};\Z)$.  Then, once again, we can decompose along the surface $S = C \times I$.  An argument completely analogous to the one in the odd case then implies that, up to a homological grading shift, \[SFH(\boldSigma(D \times I, T^n)) \cong \widehat{HFK}(S^3,K \# K^r),\] where $K^r$ is the orientation reverse of $K$.
\end{proof}

\subsection{Gradings}\label{sec:Gradings}
We wish to relate the natural gradings arising in sutured Floer homology with the gradings in Heegaard knot Floer homology, and compare both with the Khovanov bigrading appearing in the categorification of the reduced colored Jones polynomial.

We begin by defining relative $\Z$ homological (Maslov) and filtration (Alexander) gradings in the sutured Floer chain complex for the sutured manifolds $\boldSigma(D \times I, T^n)$ discussed in the previous subsection.

\begin{definition} (Maslov grading) \label{defn:SFHMaslovGrading}
Let $(\Sigma, \boldalpha, \boldbeta)$ be a Heegaard diagram representing the sutured manifold $\boldSigma(D \times I,T^n)$,  and ${\bf x}, {\bf y} \in \mathbb{T}_\alpha \cap \mathbb{T}_\beta$ with $\pi_2({\bf x},{\bf y})$ non-empty, $\phi \in \pi_2({\bf x},{\bf y})$.  Then
\[{\bf M}_{SF}({\bf x},{\bf y}) := \mu(\phi).\]
\end{definition}

We shall treat the even and odd cases separately when discussing Alexander gradings.

\begin{definition} (Alexander grading, $n$ odd) \label{defn:SFHAlexGradingOdd}
Let $(\Sigma,\boldalpha, \boldbeta)$ be a Heegaard diagram representing $\boldSigma(D \times I,T^n)$ for $n \in \Z_{>0}$ odd, obtained as in the proof of Theorem 4.1, with product region $F_{g,1} \times I$.  Choose a point $p \in F_{g,1} \subset \Sigma$, ${\bf x}, {\bf y} \in \mathbb{T}_\alpha \cap \mathbb{T}_\beta$ satisfying $\pi_2({\bf x},{\bf y}) \neq \emptyset$, and $\phi \in \pi_2({\bf x}, {\bf y})$.  Then
\[{\bf A}_{SF}({\bf x},{\bf y}) := n_p(\phi).\]
\end{definition}

\begin{definition} (Alexander grading, $n$ even) \label{defn:SFHAlexGradingEven}
Let $(\Sigma,\boldalpha, \boldbeta)$ be a Heegaard diagram representing $\boldSigma(D \times I,T^n)$ for $n \in \Z_{>0}$ even, obtained as in the proof of Theorem \ref{thm:KnotFloerRel}.  Then $\Sigma$ has two boundary components, $\mu$ and $\mu'$, where $\mu$ (resp., $\mu'$) is the image in $\Sigma$ of a meridian (resp., oppositely-oriented meridian) of $K \# K^r$.  Consider the Heegaard diagram $\Sigma'$ obtained by gluing a disk $D_z$ to $\mu$.  Choose a point $z \in D_z \subset \Sigma'$, ${\bf x}, {\bf y} \in \mathbb{T}_\alpha \cap \mathbb{T}_\beta$ satisfying $\pi_2({\bf x},{\bf y}) \neq \emptyset$, and $\phi \in \pi_2({\bf x}, {\bf y})$.  Then
\[{\bf A}_{SF}({\bf x},{\bf y}) = n_z(\phi).\]
\end{definition}

\begin{lemma} \label{lemma:Welldefinedgradings} ${\bf M}_{SF}$ and ${\bf A}_{SF}$ are well-defined as relative $\Z$ gradings on $SFH(\boldSigma(D \times I,T^n))$.
\end{lemma}

\begin{proof} 
The result will follow from the fact that there is a unique $\phi$ in $\pi_2({\bf x},{\bf y})$ whenever $\pi_2({\bf x},{\bf y}) \neq \emptyset.$  This, in turn, follows from Proposition \ref{prop:AffineIdent} combined with the fact that $H_2(\boldSigma(D \times I,T^n);\Z) = 0$.

%Theorem 5.2 (see also Definition 8.1) of \cite{MR2253454} implies that for any sutured manifold $(Y,\Gamma)$, ${\bf M}_{SF}({\bf x},{\bf y})$ is well-defined modulo 

%\[\mathfrak{d}(\mathfrak{s}) := \mbox{gcd}_{\xi \in H_2(Y;\Z)} \langle c_1(\mathfrak{s}),\xi\rangle\]

%on pairs ${\bf x}, {\bf y}$ satisfying $\mathfrak{s}({\bf x}) = \mathfrak{s}({\bf y})$.

To see that $H_2(\boldSigma(D \times I,T^n);\Z) = 0$ for all $n \in \Z_{>0}$, we will use the decomposition of $\boldSigma(D \times I,T^n)$ discussed in Proposition \ref{prop:GlueSutMan}.

First, note that $\boldSigma(S^3,K)$ and $S^3$ are both rational homology spheres, which implies (using Mayer-Vietoris) that $H_2(\boldSigma(S^3,K) - \widetilde{K};\Z)$ and $H_2(S^3 - (K \# K^r);\Z)$ are both $\cong 0$.

Another application of Mayer-Vietoris, using the decomposition 
\[\boldSigma(D \times I,T^n) = Y \cup \boldSigma(D \times I, \vec{p} \times I)\] where\[Y = \left\{ \begin{array}{cc}
                        \boldSigma(S^3,K) - \widetilde{K} & \mbox{when $n$ is odd},\\
			S^3 - (K \# K^r) & \mbox{when $n$ is even}
		\end{array}\right.\] 
tells us that $H_2(\boldSigma(D \times I,T^n);\Z) = 0$, as desired.  Since $\pi_2({\bf x},{\bf y})$ is an affine set for the action of $H_2(\boldSigma(D \times I,T^n);\Z)$ when non-empty, if $\exists \,\,\phi \in \pi_2({\bf x},{\bf y})$, it is unique.  Hence ${\bf M}_{SF}$ and ${\bf A}_{SF}$ are well-defined.
\end{proof}

\begin{definition}
Let $(\Sigma, \boldalpha, \boldbeta,w,z)$ be a doubly-pointed Heegaard diagram representing a nullhomologous knot $K$ in a rational homology sphere $Y$, and suppose that ${\bf x}, {\bf y} \in \mathbb{T}_\alpha \cap \mathbb{T}_\beta$ with $\pi_2({\bf x},{\bf y})$ non-empty and $\phi \in \pi_2({\bf x},{\bf y})$.  Then

\begin{eqnarray*}
{\bf M}_{HF}({\bf x},{\bf y}) &:=& \mu(\phi) - 2n_w(\phi)\\
{\bf A}_{HF}({\bf x},{\bf y}) &:=& n_z(\phi) - n_w(\phi).
\end{eqnarray*}
\end{definition}

%For a general sutured manifold $(Y,\Gamma)$, there is, as yet, no notion of absolute Maslov or Alexander grading.  However, in the cases of interest to us, there is a natural way to normalize the gradings.

The following two propositions explain the correspondence between relative ${\bf M}$ and ${\bf A}$ gradings on the two sides of the isomorphisms stated in Theorem \ref{thm:KnotFloerRel}.

\begin{proposition} \label{prop:gradingsodd}
Implicit in each isomorphism in Theorem \ref{thm:KnotFloerRel} is a set bijection
\[(\mathbb{T}_{\alpha} \cap \mathbb{T}_{\beta})_{SF} \leftrightarrow (\mathbb{T}_{\alpha} \cap \mathbb{T}_{\beta})_{HF},\] 
where $(\Sigma,\boldalpha,\boldbeta)_{HF}$ is a particular doubly-pointed Heegaard diagram computing the Heegaard Floer homology on the right-hand side of the isomorphism, and $(\Sigma, \boldalpha,\boldbeta)_{SF}$ is the sutured Heegaard diagram obtained from $\Sigma_{HF}$ by the procedure described in the proof of Theorem \ref{thm:KnotFloerRel}.  Under this correspondence, if 
\[{\bf x}, {\bf y} \in (\mathbb{T}_{\alpha} \cap \mathbb{T}_{\beta})_{SF} \cong (\mathbb{T}_{\alpha} \cap \mathbb{T}_{\beta})_{HF}\]
and $\pi_2({\bf x},{\bf y}) \neq \emptyset$, then 

\begin{eqnarray*}
{\bf M}_{SF}({\bf x},{\bf y}) &=& {\bf M}_{HF}({\bf x},{\bf y}) - (n-1){\bf A}_{HF}({\bf x},{\bf y})\\
{\bf A}_{SF}({\bf x},{\bf y}) &=& {\bf A}_{HF}({\bf x},{\bf y})
\end{eqnarray*}

for $n \in \mathbb{Z}_{>0}$ odd and 

\begin{eqnarray*}
{\bf M}_{SF}({\bf x}_1 \otimes {\bf x}_2, {\bf y}_1 \otimes {\bf y}_2) &=& {\bf M}_{HF}({\bf x_1},{\bf y}_1) + {\bf M}_{HF}({\bf x}_2,{\bf y}_2) - (n-2){\bf A}_{HF}({\bf x}_1,{\bf y}_1)\\
{\bf A}_{SF}({\bf x}_1 \otimes {\bf x}_2,{\bf y}_1 \otimes {\bf y}_2) &=& {\bf A}_{HF}({\bf x}_1,{\bf y}_1) + {\bf A}_{HF}({\bf x}_2, {\bf y}_2)\\
\end{eqnarray*}

for $n \in \mathbb{Z}_{>0}$ even.

In the odd case above, ${\bf x},{\bf y}$ refer to generators of 
\[\widehat{CFK}(\boldSigma(S^3,K),\widetilde{K}) \leftrightarrow CFH(\boldSigma(D \times I,T^n)),\] while in the even case, ${\bf x}_i, {\bf y}_i$ refer to generators of $\widehat{CFK}(S^3,K)$, and, hence, ${\bf x}_i \otimes {\bf y}_i$ refer to generators of 
\[\widehat{CFK}(S^3,K \# K^r) \leftrightarrow CFH(\boldSigma(D \times I,T^n)).\]
\end{proposition}

%Let $(\Sigma, \boldalpha, \boldbeta, w,z)$ be a doubly-pointed Heegaard diagram associated to the knot $K \subset S^3$, and let $(\Sigma, {\boldalpha},{\boldbeta})_n$ be the sutured Heegaard diagram associated to the sutured manifold $\boldSigma(D \times I, T^n)$ for even $n \in \Z_{>0}$ obtained as in the proof of Theorem \ref{thm:KnotFloerRel}.  Let 

%\[{\bf x}_1 \otimes {\bf x}_2,{\bf y}_1 \otimes {\bf y}_2  \in (\mathbb{T}_\alpha \cap \mathbb{T}_\beta)^{\otimes 2} \cong (\mathbb{T}_\alpha \cap \mathbb{T}_\beta)_n.\]

%Then
\begin{proof}
We address the odd and even cases separately.

{\flushleft {\bf Odd Case:}}

Let $n = 2k+1$ for $k \in \mathbb{Z}_{\geq 0}$.  In the proof of Theorem \ref{thm:KnotFloerRel}, we showed that we can obtain a Heegaard diagram $(\Sigma,\boldalpha,\boldbeta)_n$ for $\boldSigma(D \times I,T^n)$ from a doubly-pointed Heegaard diagram $(\Sigma,\boldalpha,\boldbeta,w,z)$ for the pair $(\boldSigma(S^3,K),\widetilde{K})$,  by 

\begin{enumerate} 
  \item removing a small disk $D_w$ around $w$ from $\Sigma$, and 
  \item replacing a small disk $D_z$ around $z$ with $\boldSigma(D,\vec{p}) = F_{k,1}$.  
\end{enumerate}

See Figure \ref{fig:OddHD}.

Now, consider ${\bf x},{\bf y} \in \mathbb{T}_\alpha \cap \mathbb{T}_\beta = (\mathbb{T}_{{\alpha}} \cap \mathbb{T}_{\beta})_n$, and let $\pi_2({\bf x},{\bf y})$ and $\pi_2({\bf x},{\bf y})_n$ denote the homotopy classes of maps connecting ${\bf x}$ to ${\bf y}$ in $Sym^d(\Sigma)$ and $Sym^d(\Sigma_n)$, respectively.

First, note that 
\[\pi_2({\bf x},{\bf y}) \neq \emptyset \,\, \Leftrightarrow \,\,\pi_2({\bf x},{\bf y})_n \neq \emptyset.\]
 Furthermore, if $\phi \in \pi_2({\bf x},{\bf y})$, there is a corresponding $\phi_n \in \pi_2({\bf x},{\bf y})_n$ obtained as follows.  Let $\mathcal{D}_z$  be the elementary domain in $\Sigma$ containing $D_z$ and $\mathcal{D}_F$ be the elementary domain in $\Sigma_n$ containing $F_{k,1}$.  Then, if we express $\phi$ as a sum of elementary domains:
\[\phi = a_z \mathcal{D}_z + \sum_{\mathcal{D}_i \neq \mathcal{D}_z} a_i \mathcal{D}_i,\]
$\phi_n$ is given by
\[\phi_n = a_z \mathcal{D}_F + \sum_{\mathcal{D}_i \neq \mathcal{D}_F} a_i \mathcal{D}_i.\]
Less formally, we say that $\phi_n$ is obtained from $\phi$ by ``replacing $D_z$ with $F_{k,1}$.''  See Figure \ref{fig:ReplaceOdd}.  Note that $n_z(\phi) = n_p(\phi_n) = a_z$, and, hence,
\[{\bf A}_{SF}({\bf x},{\bf y}) = {\bf A}_{HF}({\bf x},{\bf y}).\]

\begin{figure}
\begin{center}
\input{Figures/ReplaceOdd.pstex_t}
\end{center}
\caption{The domain $\mathcal{D}_z \subset \Sigma$ containing $D_z$ and the corresponding domain $\mathcal{D}_F \subset \Sigma_n$ containing $F_{k,1}$.}
\label{fig:ReplaceOdd}
\end{figure}

Furthermore, we shall see that 
\[\mu(\phi_n) = \mu(\phi) - 2kn_z(\phi).\]
This follows from Proposition \ref{prop:MasIndex}, Lipshitz's Maslov index formula for domains.  In particular,
\[\mu(\mathcal{D}_F) = \mu(\mathcal{D}_z) - 2k,\]
since 
\[\mbox{genus}(\mathcal{D}_F) = k + \mbox{genus}(\mathcal{D}_z),\] and hence
\[e(\mathcal{D}_F) = e(\mathcal{D}_z) - 2k.\] 
All other terms in Lipshitz's formula agree for $\mathcal{D}_F$ and $\mathcal{D}_z$.  The additivity of Maslov index for domains then implies that
\[{\bf M}_{SF}({\bf x},{\bf y}) = {\bf M}_{HF}({\bf x},{\bf y}) - (n-1){\bf A}_{HF}({\bf x},{\bf y}),\]
as desired.

{\flushleft {\bf Even case:}}

Suppose $n = 2k$ for $k \in \mathbb{Z}_{\geq 0}$, and let $(\Sigma,\boldalpha,\boldbeta)_F$ be the sutured Heegaard diagram associated to the product sutured manifold \[\boldSigma(D \times I, \vec{p} \times I) = (F_{k-1,2} \times I,\partial F_{k-1,2} \times I).\]  Let $(\Sigma,\boldalpha,\boldbeta,w,z)_K$ be a doubly-pointed Heegaard diagram for $(S^3,K)$, and $(\Sigma,\boldalpha,\boldbeta,w,z)_{K^r}$ a doubly-pointed Heegaard diagram for $(S^3,K^r)$, where $K^r$ denotes $K$'s orientation reverse.\footnote{Note that we may obtain $\Sigma_{K^r}$ from $\Sigma_{K}$ by switching the positions of $w$ and $z$.}  

In the proof of Theorem \ref{thm:KnotFloerRel}, we showed that we can obtain a Heegaard diagram $(\Sigma,\boldalpha,\boldbeta)_n$ for $\boldSigma(D \times I,T^n)$ by:

\begin{enumerate} 
  \item first removing a small disk $(D_z)_K$ around $z$ in $\Sigma_K$ and $(D_w)_{K^r}$ around $w$ in  $\Sigma_{K^r}$,
  \item gluing $F_{k-1,2}$ to $(\Sigma_K - D_z) \amalg (\Sigma_{K^r} - D_w)$ along their circular boundary components, and finally
  \item removing small disks $(D_w)_K$ around $w$ in $\Sigma_K$ and $(D_z)_{K^r}$ around $z$ in $\Sigma_{K^r}$.
\end{enumerate}

See Figure \ref{fig:EvenHD}.  Note that, to compute ${\bf A}_{SF}$ in this setting, we will replace the disk $(D_z)_{K^r}$ we removed in step $3$ above, and count intersections, $n_z$, with it, as detailed in Definition \ref{defn:SFHAlexGradingEven}.

Since
\[(\mathbb{T}_{\alpha} \cap \mathbb{T}_{\beta})_n \cong (\mathbb{T}_\alpha \cap \mathbb{T}_\beta)_{K} \otimes (\mathbb{T}_\alpha \cap \mathbb{T}_\beta)_{K^r},\]we will from now on denote each generator of $(\mathbb{T}_{\alpha} \cap \mathbb{T}_{\beta})_n$ as a tensor product of a generator ${\bf m}_1 \in (\mathbb{T}_{\alpha} \cap \mathbb{T}_\beta)_K$ and ${\bf m}_2 \in (\mathbb{T}_{\alpha} \cap \mathbb{T}_\beta)_{K^r}$.

Now, consider ${\bf x}_1,{\bf y}_1 \in (\mathbb{T}_\alpha \cap \mathbb{T}_\beta)_{K}$ and ${\bf x}_2,{\bf y}_2 \in (\mathbb{T}_{{\alpha}} \cap \mathbb{T}_{\beta})_{K^r}$.  Let $\pi_2({\bf x}_1,{\bf y}_1)_K$ (resp., $\pi_2({\bf x}_1,{\bf y}_2)_{K^r}$) denote the homotopy classes of maps connecting ${\bf x}_1$ to ${\bf y}_1$ (resp., ${\bf x}_2$ to ${\bf y}_2$) in $Sym^{d_K}(\Sigma_K)$ (resp., $Sym^{d_{K^r}}(\Sigma_{K^r})$), and let $\pi_2({\bf x}_1 \otimes {\bf x}_2,{\bf y}_1 \otimes {\bf y}_2)_n$ denote the homotopy classes of maps in $Sym^{d}(\Sigma_n)$.  Here, $d = d_K + d_{K^r}$.  

\begin{figure}
\begin{center}
\input{Figures/ReplaceEven1.pstex_t}
\end{center}
\caption{The domain $\mathcal{D}_z \subset \Sigma$ containing $D_z$ and the corresponding domain $\mathcal{D}_n \subset \Sigma_n$ containing $F_{k-1,2}\amalg_{\partial}(\Sigma_{K^r} - (D_w)_{K^r})$.}
\label{fig:ReplaceEven1}
\end{figure}

If $\phi \in \pi_2({\bf x}_1, {\bf y}_1)$, then there is a corresponding $\phi_n \in \pi_2({\bf x}_1 \otimes {\bf m},{\bf y}_1 \otimes {\bf m})$ for any ${\bf m} \in (\mathbb{T}_\alpha \cap \mathbb{T}_\beta)_{K^r}$, obtained by ``replacing $D_z$ with $[F_{k-1,2} \amalg_{\partial} (\Sigma_{K^r} - (D_w)_{K^r})]$'' as we did in the odd case.  See Figure \ref{fig:ReplaceEven1}.  It is clear that $n_z(\phi) = n_z(\phi_n),$ and, hence, that
\[{\bf A}_{HF}({\bf x}_1,{\bf y}_1) = {\bf A}_{SF}({\bf x}_1 \otimes {\bf m},{\bf y}_1 \otimes {\bf m})\]

for any ${\bf m} \in (\mathbb{T}_\alpha \cap \mathbb{T}_\beta)_{K^r}.$

To compute the relative Maslov grading associated to $\phi_n$, we need to analyze the Maslov index associated to the smallest domain $\mathcal{D}_n \subset \Sigma_n$ containing $F_{k-1,2} \amalg (\Sigma_{K^r} - (D_w)_{K^r})$ and compare it to the Maslov index associated to the smallest domain $\mathcal{D} \subset \Sigma_{K}$ containing $(D_z)_{K}$.  Note that $d_{K^r}$ is the genus of $\Sigma_{K^r}$ and $k-1$ is the genus of $F_{k-1,2}$.

Lipshitz's Maslov index formula tells us:
\[ \mu(\mathcal{D}) = e(\mathcal{D}) + \sum n_{{\bf x}_1}(\mathcal{D}) + \sum n_{{\bf y}_1}(\mathcal{D})\]
and

\begin{eqnarray*}
\mu(\mathcal{D}_n) &=& e(\mathcal{D}_n) + \sum n_{{\bf x}_1 \otimes {\bf m}}(\mathcal{D}_n) + \sum n_{{\bf y}_1 \otimes {\bf m}}(\mathcal{D}_n)\\
                 &=& e(\mathcal{D}_n) + \sum n_{{\bf x}_1}(\mathcal{D}_n) + \sum n_{{\bf y}_1}(\mathcal{D}_n) + \sum 2n_{\bf m}(\mathcal{D}_n)\\
                   &=& e(\mathcal{D}_n) + \sum n_{{\bf x}_1}(\mathcal{D}_n) + \sum n_{{\bf y}_1}(\mathcal{D}_n) + 2d_{K^r}\\
\end{eqnarray*}  

Since $\sum n_{{\bf x}_1}(\mathcal{D}) = \sum n_{{\bf x}_1}(\mathcal{D}_n)$ and $\sum n_{{\bf y}_1}(\mathcal{D}) = \sum n_{{\bf y}_1}(\mathcal{D}_n)$, the difference of the two Maslov indices is:
\[\mu(\mathcal{D}) - \mu(\mathcal{D}_n) = \left[e(\mathcal{D})\right] - \left[e(\mathcal{D}_n) + 2d_{K^r}\right]\]
But 
\[\mbox{genus}(\mathcal{D}_n)= \mbox{genus}(\mathcal{D}) + (k-1) + d_{K^r};\] 
hence, 
\[e(\mathcal{D}_n) = e(\mathcal{D}) - 2[(k-1) + d_{K^r}],\] which implies:
\[\mu(\mathcal{D}) - \mu(\mathcal{D}_n) = 2(k-1) = n-2.\]

Using the additivity of the Maslov index of domains, we therefore conclude that \[ {\bf M}_{SF}({\bf x}_1 \otimes {\bf m}, {\bf y}_1 \otimes {\bf m}) = {\bf M}_{HF}({\bf x}_1,{\bf y}_1) - (n-2){\bf A}_{HF}({\bf x}_1,{\bf y}_1)\]

if ${\bf m} \in (\mathbb{T}_\alpha \cap \mathbb{T}_\beta)_{K^r}$.

\begin{figure}
\begin{center}
\input{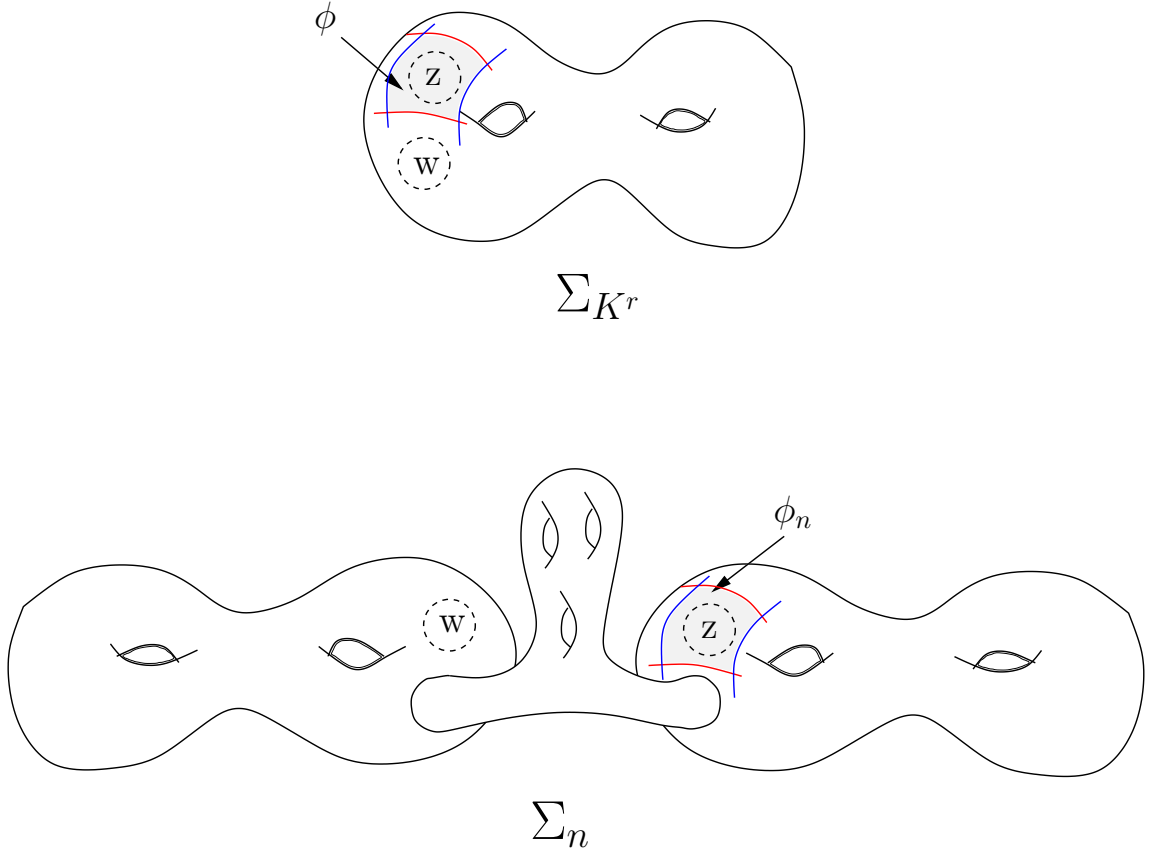}
\end{center}
\caption{A domain $\mathcal{D} \subset \Sigma_{K^r}$ representing $\phi \in \pi_2({\bf x}_2,{\bf y}_2)_{K^r}$ along with the corresponding domain $\mathcal{D}_n \subset \Sigma_n$ representing $\phi_n \in \pi_2({\bf m} \otimes {\bf x}_2, {\bf m} \otimes {\bf y}_2)_n$.}
\label{fig:ReplaceEven2}
\end{figure}

Similarly, if $\phi \in \pi_2({\bf x}_2, {\bf y}_2)_{K^r}$, then there is a corresponding $\phi_n \in \pi_2({\bf m} \otimes {\bf x}_2, {\bf m} \otimes {\bf y}_2)_n$ for any ${\bf m} \in (\mathbb{T}_\alpha \cap \mathbb{T}_\beta)_K$ which has precisely the same domain, just considered as a subset of $\Sigma_n$, rather than $(\Sigma)_{K^r}$.  See Figure \ref{fig:ReplaceEven2}.  It is therefore clear that:

\begin{eqnarray*}
{\bf M}_{SF}({\bf m} \otimes {\bf x}_2,{\bf m} \otimes {\bf y}_2) &=& {\bf M}_{HF}({\bf x}_2, {\bf y}_2)\\
{\bf A}_{SF}({\bf m} \otimes {\bf x}_2,{\bf m} \otimes {\bf y}_2) &=& {\bf A}_{HF}({\bf x}_2, {\bf y}_2)\\
\end{eqnarray*}

for any ${\bf m} \in (\mathbb{T}_\alpha \cap \mathbb{T}_\beta)_K$.

Hence, we can use additivity:
\begin{eqnarray*}
{\bf M}_{SF}({\bf x}_1 \otimes {\bf x}_2, {\bf y}_1 \otimes {\bf y}_2) &=& {\bf M}_{SF}({\bf x}_1 \otimes {\bf x}_2, {\bf y}_1 \otimes {\bf x}_2) + {\bf M}_{SF}({\bf y}_1 \otimes {\bf x}_2, {\bf y}_1 \otimes {\bf y}_2),\\
{\bf A}_{SF}({\bf x}_1 \otimes {\bf x}_2, {\bf y}_1 \otimes {\bf y}_2) &=& {\bf A}_{SF}({\bf x}_1 \otimes {\bf x}_2, {\bf y}_1 \otimes {\bf x}_2) + {\bf A}_{SF}({\bf y}_1 \otimes {\bf x}_2, {\bf y}_1 \otimes {\bf y}_2),\\
\end{eqnarray*}
to conclude that
\begin{eqnarray*}
{\bf M}_{SF}({\bf x}_1 \otimes {\bf x}_2, {\bf y}_1 \otimes {\bf y}_2) &=& {\bf M}_{HF}({\bf x}_1, {\bf y}_1) + {\bf M}_{HF}({\bf x}_2, {\bf y}_2) - (n-2){\bf A}_{HF}({\bf x}_1, {\bf y}_1),\\
{\bf A}_{SF}({\bf x}_1 \otimes {\bf x}_2, {\bf y}_1 \otimes {\bf y}_2) &=& {\bf A}_{HF}({\bf x}_1, {\bf y}_1) + {\bf A}_{HF}({\bf x}_2, {\bf y}_2)
\end{eqnarray*}
as desired.

%First, note that 
%\[\pi_2({\bf x},{\bf y}) \neq \emptyset \,\, \Leftrightarrow \,\,\pi_2^n({\bf x},{\bf y}) \neq \emptyset,\]
 
%Furthermore, if $\phi \in \pi_2({\bf x},{\bf y})$, there is a corresponding $\phi_n \in \pi_2^n({\bf x},{\bf y})$ obtained by replacing any domain $\mathcal{D} \subset \phi$ containing $D_z$ with the corresponding domain $\mathcal{D}_n$ containing $F_{k,1}$.  It is then clear that $n_z(\phi) = n_p(\phi_n)$, and, hence,
%\[{\bf A}_{SF}({\bf x},{\bf y}) = {\bf A}_{HF}({\bf x},{\bf y}).\]

%Furthermore,

%\[\mu(\phi_n) = \mu(\phi) - 2kn_z(\phi).\]

%This follows from Proposition \ref{prop:MasIndex}, Lipshitz's Maslov index formula for domains.  In particular, suppose that $\mathcal{D}$ (resp., $\mathcal{D}_n$ is the unique domain on $\Sigma$ (resp., $\Sigma^n$) containing $D_z$ (resp., $F_{k,1}$).  Then

%\[\mu(\mathcal{D}_n) = \mu(\mathcal{D}) - 2k,\]

%since $e(\mathcal{D}_n) = e(\mathcal{D}) - 2k$ and all other terms in Lipshitz's formula agree.  The additivity of Maslov index for domains implies that

%\[{\bf M}_{SF}({\bf x},{\bf y}) = {\bf M}_{HF}({\bf x},{\bf y}) - (n-1){\bf A}_{HF}({\bf x},{\bf y}),\]

%as desired.

\end{proof}

\section{Relationship with Khovanov's gradings}
In this section, we discuss a conjectural relationship between
the two gradings on $\widetilde{Kh}_n(K)$
and the relative Maslov grading on $SFH(\boldSigma(D\times I,T^n))$.

Recall that the spectral sequence of Proposition~\ref{prop:SpecSeq}
$$
E^1=CV(\cP(T))\Rightarrow E^{\infty}=SFH(\boldSigma(D\times I,T))
$$
was defined in terms of a filtered
complex $X^{(0,1)}$, whose differential can be written
as a sum $D^{(0,1)}=D_0+D_1+\ldots+D_\ell$, where $D_k$ counts
homolorphic $(k-2)$-gons.
If $D_0$ vanishes (which we can always
achieve by choosing a suitable Heegaard
multi-diagram), then the $E^1$ term of the spectral sequence
coincides with the $E^0$ term, and hence there is an
identification of vector spaces
$$
X^{(0,1)}=E^0=E^1=CV(\cP(T))
$$
which we can use to define a bigrading on
the vector space $X^{(0,1)}$, by setting
$$
(X^{(0,1)})^{i,j}:=CV(\cP(T))^{i,j}.
$$
The differential $D^{(0,1)}=D_0+D_1+\ldots+D_\ell$ does not preserve
this bigrading; however, it is easy to see that $D_k$
\begin{itemize}
\item
raises the $i$ grading by $k$,
\item
raises the $j$ grading by $2(k-1)$,
\end{itemize}
and thus lowers the $\delta$ grading
$$
\delta(a):=\frac{j(a)}{2}-i(a)
$$
by $1$. It follows that
$D^{(0,1)}$ lowers the $\delta$ grading by $1$,
and hence $\delta$ can be regarded as a
homological grading for the complex $X^{(0,1)}$.
In particular, $\delta$ induces homological gradings
on each of the pages of the spectral sequence $\{E^r\}_{r\geq 0}$,
as well as on the $E^{\infty}$ term.
Let $E^{\infty}_{\delta=d}$ be the
subspace of $E^{\infty}$ which sits in $\delta$ degree $d$.

\begin{conjecture}\label{conj:Gradings}
Let $(\Sigma,\boldalpha,\boldbeta)$ be
a sutured Heegaard diagram for $\boldSigma(D\times I,T)$.
Then there is a function
${\bf M}:\mathbb{T}_{\alpha}\cap\mathbb{T}_{\beta}
\rightarrow\Q$
which satisfies
$$
{\bf M}({\bf x})-{\bf M}({\bf y})=
{\bf M}_{SF}({\bf x},{\bf y})
$$
whenever ${\bf x},{\bf y}\in\mathbb{T}_\alpha\cap\mathbb{T}_\beta$
are two intersection points with $\pi_2({\bf x},{\bf y})\neq\emptyset$,
and such that
$$
E^{\infty}_{\delta=d}=SFH(\boldSigma(D\times I,T))_{{\bf M}=d}
$$
where
$SFH(\boldSigma(D\times I,T))_{{\bf M}=d}\subset
SFH(\boldSigma(D\times I,T))$ is the subspace
which consists of all homology classes
which can be written as linear combinations
of intersection points
${\bf x}\in\mathbb{T}_\alpha\cap\mathbb{T}_\beta$
with ${\bf M}({\bf x})=d$.
\end{conjecture}

The following definition is essentially taken
from \cite{MR2034399}.

\begin{definition}
The {\em homological width} of $\widetilde{Kh}_n(K)$ is
the positive integer
$$
hw(\widetilde{Kh}_n(K)):=
\max_a\delta(a)-\min_a\delta(a)+1
$$
where here $a$ is allowed to vary
over all nonzero homology
classes $a\in\widetilde{Kh}_n(K)$ which are
homogeneous with respect to the bigrading
on $\widetilde{Kh}_n(K)$.
\end{definition}

\begin{corollary}\label{cor:Gradings} (of Conjecture~\ref{conj:Gradings})
Let $K$ be a knot in $S^3$ and let $g(K)$ denote its genus.
Then
$$
\liminf_{n \rightarrow\infty}\frac{hw(\widetilde{Kh}_n(K))}{2(n-1)}\geq g(K)
$$
where the limit is taken over odd $n$ only, and
$$
\liminf_{n\rightarrow\infty}\frac{hw(\widetilde{Kh}_n(K))}{2(n-2)}\geq g(K)
$$
where the limit is taken over even $n$ only.
\end{corollary}

\begin{proof}{\em (Sketch)}
We only discuss the odd case, the even
case being completely analogous.
Thus, let $n\in\mathbb{Z}_{>0}$ be odd, and let
$(\Sigma,\boldalpha,\boldbeta)_n$ and
$(\Sigma,\boldalpha,\boldbeta,w,z)$
be admissible Heegaard diagrams for $\boldSigma(D\times I,T^n)$
and $(\boldSigma(S^3,K),\widetilde{K})$, respectively,
as in the proof of Proposition~\ref{prop:gradingsodd}.
For a Spin$^c$ structure
$\mathfrak{s}\in\mathfrak{S}:=Spin^c(\boldSigma(S^3,K))$,
define
$$
M_\mathfrak{s}(K,n):=
\max{\bf M}_{SF}({\bf x},{\bf y})
$$
$$
M_\mathfrak{s}(\widetilde{K}):=
\max{\bf M}_{HF}({\bf x},{\bf y})
$$
$$
A_\mathfrak{s}(\widetilde{K}):=
\max{\bf A}_{HF}({\bf x},{\bf y})
$$
where the maxima are taken over all pairs of intersection points
${\bf x},{\bf y}\in\mathbb{T}_\alpha\cap\mathbb{T}_\beta$
which represent the given Spin$^c$ structure
$\mathfrak{s}\in\mathfrak{S}$.
Using Conjecture~\ref{conj:Gradings} and Proposition~\ref{prop:gradingsodd},
one can easily check that
$$
hw(\widetilde{Kh}_n(K))>\max_{\mathfrak{s}\in\mathfrak{S}}
M_\mathfrak{s}(K,n)\geq
\max_{\mathfrak{s}
\in \mathfrak{S}}\left[(n-1)A_\mathfrak{s}(\widetilde{K})
-M_\mathfrak{s}(\widetilde{K})\right]
$$
and hence
$$
\frac{hw(\widetilde{Kh}_n(K))}{n-1}>
\max_{\mathfrak{s}\in \mathfrak{S}}A_\mathfrak{s}(\widetilde{K})-\epsilon_n
$$
where $\epsilon_n:=
\max_{\mathfrak{s}\in\mathfrak{S}}M_\mathfrak{s}(\widetilde{K})/(n-1)$.
Since $\epsilon_n\rightarrow 0$ as $n\rightarrow\infty$,
the corollary now follows from the well-known
fact \cite{MR2023281} (see also \cite{GT064360}, \cite{MR2390347}) that
$$
\max_{\mathfrak{s}\in \mathfrak{S}}
A_\mathfrak{s}(\widetilde{K})= 2g(\widetilde{K})=2g(K).
$$
\end{proof}

\bibliography{SFHExact}
\end{document}